\documentclass[11pt]{article}

\usepackage{amssymb,amsfonts,amsthm,amsmath}

\usepackage{booktabs}

\usepackage{graphicx, verbatim, centernot}
\usepackage{caption}
\usepackage{xcolor}

\usepackage{multicol}

\usepackage{enumitem}
\usepackage[normalem]{ulem}

\usepackage[a4paper, hmargin=2cm, vmargin={3cm, 3cm}]{geometry}
\usepackage{fancyhdr}

\usepackage{hyperref}
\hypersetup{
	colorlinks=true,
	linkcolor=blue,
	citecolor=purple,
	filecolor=magenta,      
	urlcolor=purple,
}

\usepackage[
backend=bibtex,
style=alphabetic,
sortlocale=auto,
uniquelist=false,
natbib=false,
url=false, 
doi=false,
eprint=false,
safeinputenc=false,
giveninits=true,
isbn=false,
maxbibnames=10,
]{biblatex}

\renewbibmacro{in:}{}

\DeclareFieldFormat*[inbook]{citetitle}{#1}
\DeclareFieldFormat*{title}{#1}

\newtheorem{theorem}{Theorem}

\newtheorem{lemma}[theorem]{Lemma}

\newtheorem{proposition}[theorem]{Proposition}
\newtheorem{corollary}[theorem]{Corollary}

\newtheorem{claim}[theorem]{Claim}

\theoremstyle{definition}
\newtheorem*{definition*}{Definition}
\newtheorem{definition}[theorem]{Definition}

\newcommand{\theoremname}{testing}
\theoremstyle{remark}
\newtheorem*{remark*}{Remark}
\newtheorem{remark}[theorem]{Remark}

\numberwithin{theorem}{section}

\def\eps{{\varepsilon}}

\def\Ga{{\Gamma}}
\def\la{{\lambda}}

\newcommand{\bR}{\mathbb R}
\newcommand{\bC}{\mathbb C}
\newcommand{\bZ}{\mathbb Z}
\newcommand{\bD}{\mathbb D}

\newcommand{\bN}{\mathbb N}

\newcommand{\bP}{\mathbb P}
\newcommand{\bS}{\mathbb S}
\newcommand{\bE}{\mathbb E}
\newcommand{\cJ}{\mathcal J}

\renewcommand{\P}{\bP}

\newcommand{\footremember}[2]{%
	\footnote{#2}
	\newcounter{#1}
	\setcounter{#1}{\value{footnote}}%
}

\AtBeginDocument{\pagestyle{plain}}

\begin{document}
	
	\title{Limit law for root separation in random polynomials}
	
	\date{}
	\author{%
		Marcus Michelen\footremember{UIC}{University of Illinois, Chicago. \emph{Email address}: {\tt michelen.math@gmail.com}}%
		\and Oren Yakir\footremember{Stanford}{Stanford University. \emph{Email address}: {\tt oren.yakir@gmail.com}}%
	}
	
	\maketitle
	\begin{abstract}
		Let $f_n$ be a random polynomial of degree $n\ge 2$ whose coefficients are independent and identically distributed random variables. We study the separation distances between roots of $f_n$ and prove that the set of these distances, normalized by $n^{-5/4}$, converges in distribution as $n\to \infty$ to a non-homogeneous Poisson point process. As a corollary, we deduce that the minimal separation distance between roots of $f_n$, normalized by $n^{-5/4}$  has a non-trivial limit law.
        In the course of the proof, we establish a related result which may be of independent interest: a Taylor series with random i.i.d.\ coefficients almost-surely does not have a double zero anywhere other than the origin.
	\end{abstract}
	
	\thispagestyle{empty}
	
	{ 
    \tableofcontents }
	\
	
	\thispagestyle{empty}

	\section{Introduction}
	Consider the random polynomial
	\begin{equation}
		\label{eq:intro_def_f_n}
		f_n(z) = \sum_{k=0}^{n} \xi_k z^k \, ,
	\end{equation}
	where $\xi_0,\ldots,\xi_n$ are i.i.d.\ random variables. This model,  usually referred to as the 
    \emph{Kac polynomial} after Mark Kac~\cite{Kac}, was studied in classical works by Bloch-P\'olya~\cite{Bloch-Polya}, Littlewood-Offord~\cite{Littlewood-Offord}, Salem-Zygmund~\cite{Salem-Zygmund}, Erd\H{o}s-Offord~\cite{Erdos-Offord}, Konyagin-Schlag~\cite{Konyagin-Schlag}, Tao-Vu~\cite{Tao-Vu-IMRN}, and many more; this list is by no means complete. As the coefficients of~\eqref{eq:intro_def_f_n} are random, it is natural to ask about the typical spatial distribution of the roots as the degree $n$ grows large.  The case when the coefficients follow the Gaussian distribution often leads to the most tractable analysis, with exact formulas for correlations of roots.  It is widely believed (and only sometimes proved) that many properties of the random roots do not depend on the specific choice of coefficients distribution, as $n\to \infty$. This type of meta-conjecture is commonly referred to as the \emph{universality} phenomena.  In the special case of when $\xi_0$ takes the values $\{-1,+1\}$ with equal probability, $f_n$ is called a random Littlewood polynomial in part due to the works \cite{Littlewood-Offord, Littlewood}.  Although we expect similar asymptotic phenomena, studying the roots of a random Littlewood polynomial is, generally speaking, a more challenging task than its Gaussian counterpart. 
	
	A striking and now-classical fact 
	is that under mild integrability conditions on $\xi_0$, most of the roots of $f_n$ tend to cluster uniformly around the unit circle as $n\to \infty$. Special cases of this phenomenon were shown by 
	Erd\H{o}s-Tur\'an~\cite{Erdos-Turan}, \v Sparo-\v Sur~\cite{Sparo-Sur}, Hammersley~\cite{Hammersley}, and in great generality by the work of Ibragimov--Zaporozhets~\cite{Ibragimov-Zaporozhets}. On a finer scale, this clustering was shown by Shepp-Vanderbei~\cite{Shepp-Vanderbei} and Ibragimov-Zeitouni~\cite{Ibragimov-Zeitouni} to hold on the scale of $n^{-1}$, meaning that most of the roots are present in the annulus of radius $\{1-K/n \le |z| \le 1+K/n\}$ as soon as $K$ is large but fixed and $n\to \infty$ (see \cite{Michelen-Sahasrabudhe,Cook-Nguyen-Yakir-Zeitouni} for finer points on the distance of roots to the unit circle).  
    
    Additionally, the roots of random polynomials experience \emph{repulsion}, although a precise quantitative statement of this is  rather technical (see Figure \ref{fig:roots_vs_uniform} for a visual demonstration). In the case of Gaussian coefficients, this can be understood by looking at pair correlations.  This amounts to asking for the probability density of having roots at points $z$ and $w$ and demonstrating decay as $z$ and $w$ become close. The work of Shiffman-Zelditch~\cite[Theorem~4]{Shiffman-Zelditch-IMRN} 
    demonstrated repulsion for the case of Gaussian coefficients and local universality results for random polynomials show that repulsion occurs for a wide range of coefficients (see, e.g., \cite{Tao-Vu-IMRN}).  Repulsion of roots of random polynomials has also been studied in the physics literature,  see for example~\cite{Bogomolny-Bohigas-Leboeuf-JSP}.

	In this paper we address the following question: \emph{What is the typical separation distance between the roots of $f_n$?} Throughout the paper, we will assume that $\xi_0$ is a mean-zero, sub-Gaussian\footnote{That is, $\bP\big[|\xi_0|\ge t\big]\le 2\exp(-ct^2)$ for some $c>0$;} random variable. 
	Denote by $\alpha_1,\ldots,\alpha_n$ the (random) roots of $f_n$, taken with multiplicities and in an arbitrary order. Our main result is the following:
	\begin{theorem}
		\label{thm:poisson_limit_for_close_roots}
		Let $f_n$ be given by~\eqref{eq:intro_def_f_n}, and suppose $\xi_0$ is a mean-zero, sub-Gaussian random variable satisfying $\P[\xi_0 = 0] = 0$. Then the point process
 	\[
		\Big\{ n^{5/4} \, |\alpha_j - \alpha_{j^\prime}| \, : \, 1\le j < j^\prime \le n \Big\}
		\]
		converges in distribution (with respect to the vague topology) as $n\to \infty $ to a non-homogeneous Poisson point process on $\bR_{\ge 0}$ with intensity $\mathfrak{c}_\ast t^3$, for some $\mathfrak{c}_\ast >0$.  
	\end{theorem} 
    
    The value of $\mathfrak{c}_{\ast}>0$ from Theorem~\ref{thm:poisson_limit_for_close_roots} is somewhat explicit and is given by~\eqref{eq:limiting_constant_for_intensity} below. We remark that we may drop the assumption~$\bP[\xi_0 = 0] = 0$, with the cost of adding more wording to our main result. Indeed, in the latter case the probability of having a double root at the origin is exactly $\bP[\xi_0 = 0 ]^2$, so the limiting Poisson process is super-positioned by an independent unit delta-mass at the origin, which occurs independently of the Poisson process with probability $\bP[\xi_0 = 0 ]^2$. To avoid this technicality,  throughout the paper we adopt the assumption
	\begin{equation}
		\label{eq:intro_no_atom_at_origin}
		\bP\big[\xi_0 = 0 \big] = 0 \, .
	\end{equation}  
	We also remark that with minor modifications of the proof, Theorem~\ref{thm:poisson_limit_for_close_roots} continues to hold even if $\bE[\xi_0] \not=0$, but to simplify the presentation we adopt the mean-zero assumption throughout the paper.

	\subsection{Heuristic explanation for the scaling exponent}
    The scaling exponent $5/4$ in Theorem~\ref{thm:poisson_limit_for_close_roots} reflects the repulsion between the random roots. To better explain this point, we consider a simple toy model where we pretend the roots of $f_n$ do not interact. Suppose that $X_1,\ldots,X_n$ are i.i.d.\ random variables, all uniformly distributed in the annulus $$\mathcal{A} = \mathcal{A}_K = \Big\{ 1-\frac{K}{n} \le |z|\le 1+\frac{K}{n} \Big\}$$ for some $K\ge 1$ large but fixed.  We recall that if $K$ is a large constant, then the vast majority of roots of $f_n$ lie in this annulus \cite{Shepp-Vanderbei,Ibragimov-Zeitouni}, and so this is indeed a toy model for the roots of $f_n$ without interaction between roots. 
	Note that for any $1\le j < j^\prime \le n$ we have
	\[
	\bP\big[ |X_j - X_{j^\prime}| \le \eps \big] \approx n \eps^2 \, , 
	\] 
	since the disk of radius $\eps$ around $X_j$ is of area $\approx \eps^2$, and $\mathcal{A}$ has area $\approx n^{-1}$.  Linearity of expectation then shows	
	\[
	\bE\Big[\#\big\{ \text{pairs } j\not=j^\prime \ \text{such that } X_j, X_{j^\prime} \ \text{are } \eps\text{-close}\big\}\Big] \approx n^3 \eps^2 \, ,
	\]
	and hence, we expect some non-trivial scaling limit when $\eps = n^{-3/2}$.  Let us now contrast this with the roots $\alpha_1,\ldots,\alpha_n$ of the random polynomial $f_n$ given by~\eqref{eq:intro_def_f_n}. Letting $\bD(x,\eps)$ denote the disk of radius $\eps>0$ centered at $x\in \mathcal{A}$, we have that 
	\begin{equation*}
		\bE\Big[\#\big\{  j :  X_j \in \bD(x,\eps)\big\}\Big] \approx (n \eps)^2 \quad \text{ and } \quad 	\bE\Big[\#\big\{  j :  \alpha_j \in \bD(x,\eps)\big\}\Big] \approx (n\eps)^2 \,.
	\end{equation*}
	In other words, the first intensity of the point process $X_1,\ldots,X_n$ is comparable to that of random roots.  The key difference between the i.i.d.\ points $X_1,\ldots,X_n$ and the random roots $\alpha_1,\ldots,\alpha_n$ comes in at the $2$-point intensity (sometimes referred to in the literature as the pair correlations). By independence, for $\eps \ll n^{-1}$, one can compute 
	\begin{equation} \label{eq:intro_prob_for_two_uniform_pts_in_disk}
		\bP\Big[\exists j\not= j^\prime \ \text{such that} \ X_j, \, X_{j^\prime} \in \bD(x,\eps) \Big] \approx n^2 \times \big(n \eps^2\big)^2\,.
	\end{equation}
	In contrast, for random roots one has 
	\begin{equation} \label{eq:intro_prob_for_two_roots_in_disk}
		\bP\Big[\exists \ \text{two distinct roots of } f_n \in \bD(x,\eps) \Big] \approx n^2 \times \big(n \eps^2\big)^2 \times (n\eps)^2 \, ,
	\end{equation}
	instead of~\eqref{eq:intro_prob_for_two_uniform_pts_in_disk}. The extra $(n\eps)^2$ term in~\eqref{eq:intro_prob_for_two_roots_in_disk} can be thought of as a repulsion term between the roots, see Figure~\ref{fig:roots_vs_uniform} for a related simulation. 
	\captionsetup[figure]{labelfont={bf},labelformat={default},labelsep=period,name={Fig.}}
	\begin{figure}[!htbp]
		\begin{center}
			\qquad \qquad \scalebox{0.5}{\includegraphics{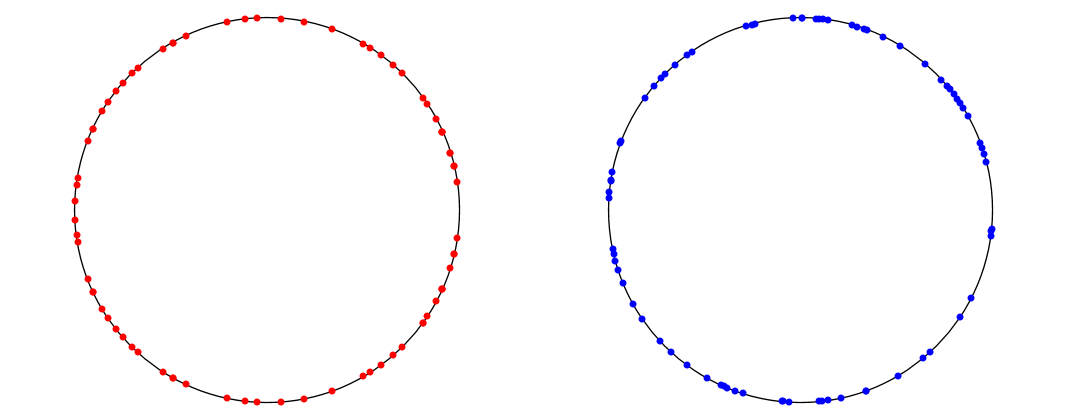}}
		\end{center}
		\caption{Left: roots on the unit circle of $\text{Re}(f_n)$ for random Littlewood $f_n$ (red points); Right: the same number of i.i.d.\ uniform points on the unit circle (blue points). }
		\label{fig:roots_vs_uniform}
	\end{figure}
	When the coefficients are Gaussian, one can formally derive~\eqref{eq:intro_prob_for_two_roots_in_disk} via the Kac-Rice formula for the 2-point function of the random roots, see for example~\cite[Chapter~3.4]{GAFbook} or~\cite[Theorem~4]{Shiffman-Zelditch-IMRN}. One of the main achievements of this paper is obtaining the asymptotic~\eqref{eq:intro_prob_for_two_roots_in_disk} for a large class of coefficients distribution 
    for $\eps$ which is polynomially small in $n$, including in the Littlewood case (see Section~\ref{subsection:outline_of_prood} below, and~\eqref{eq:intro-probability_of_local_event} in particular). Once~\eqref{eq:intro_prob_for_two_roots_in_disk} is established, union bounding over an $\eps$-net of $\mathcal{A}$ of size $n^{-1} \eps^{-2}$ yields
	\[
	\bE\Big[\#\big\{ \text{pairs of roots of } f_n  \ \text{which are } \eps\text{-close}\big\}\Big] \approx n^5 \eps^4 \, ,
	\]
	and hence the scaling $\eps= n^{-5/4}$ is a natural one for our problem.
	\subsection{Double roots of random polynomials}
	Recalling that $\alpha_1,\ldots,\alpha_n$ are the roots of $f_n$ given by~\eqref{eq:intro_def_f_n}, we denote by
	\[
	m_n = \min\big\{|\alpha_j - \alpha_{j^\prime}|\, : \, 1\le j < j^\prime \le n  \big\} 
	\]
	the distance between the closest distinct roots of $f_n$. Clearly, $m_n = 0$ if and only if $f_n$ has a double root. An immediate consequence of Theorem~\ref{thm:poisson_limit_for_close_roots} is the following limit law for the minimal separation between the roots. 
	\begin{corollary}
		\label{cor:no_double_roots_whp}
		Let $f_n$ be given as in~\eqref{eq:intro_def_f_n} with mean-zero and sub-Gaussian coefficients, and assume that~\eqref{eq:intro_no_atom_at_origin} holds. Then for all $s\ge 0$ we have
		\[
		\lim_{n\to \infty} \bP\Big[n^{5/4} m_n \ge s\Big] = \exp\Big(-\frac{\mathfrak{c}_\ast}{4} \, s^4\Big) \, .
		\]
		In particular, the probability that $f_n$ has a double root tends to zero as $n\to \infty$.
	\end{corollary}
	\begin{proof}
		For a Poisson process with intensity $\mathfrak{c}_\ast t^3$, the number of points in $[0,s]$ has a Poisson distribution with parameter
		\[
		\int_{0}^{s} \mathfrak{c}_\ast t^3  \, {\rm d} t =  \frac{\mathfrak{c}_\ast}{4}s^4 \, .
		\]
		Hence, the probability there are no points present in this interval is $\exp\big(-\frac{\mathfrak{c}_\ast}{4} \, s^4\big)$, and the corollary follows from Theorem~\ref{thm:poisson_limit_for_close_roots}. 
	\end{proof}
	For polynomials of the form~\eqref{eq:intro_def_f_n} with random integer coefficients, the probability of having a double root has been studied in prior works. For random Littlewood polynomials, Peled, Sen and Zeitouni~\cite{Peled-Sed-Zeitouni} computed the asymptotic of this probability as $n\to \infty$. In particular, they showed that the rate of decay depends on divisibility properties of the degree $n$. Later on, Feldheim and Sen~\cite{Feldheim-Sen} extended this result to more general integer-valued distributions $\xi_0$, and obtained similar findings. We remark that the problem of investigating the minimal gap between the roots was also suggested in~\cite[Section~7]{Feldheim-Sen}. It is worth noting that the approach in both~\cite{Peled-Sed-Zeitouni,Feldheim-Sen} are highly algebraic, with heavy use of the fact that the coefficients are integers. Additionally, a key assumption in both works is that the maximal atom of $\xi_0$ has probability at most $1/\sqrt{3}$ and that the $\xi_0$ has bounded support. To the best of our knowledge, Corollary~\ref{cor:no_double_roots_whp} appears to be the first result in the literature that establishes the absence of double roots with high probability for Kac polynomials for a wide class of non-degenerate coefficient distributions.

	\subsection{Double zeros of random analytic functions in the disk}
	En route of proving our main result Theorem~\ref{thm:poisson_limit_for_close_roots}, we must address the following natural question: \emph{Can a random analytic function in the disk 
	have a double zero?} Consider an infinite sequence of i.i.d.\ (possibly complex) random variables $\{a_k\}$, and assume that
	\begin{equation}
		\label{eq:intro-log_moment_assumption}
		\bE\big[\log(1+|a_0|)\big] < \infty \, .
	\end{equation}
	By applying the Borel-Cantelli lemmas, one can verify that the condition~\eqref{eq:intro-log_moment_assumption} is both necessary and sufficient for the random power series
	\begin{equation}
		\label{eq:intro-def_of_infinite_power_series}
		F(z) = \sum_{k=0}^\infty a_k z^k
	\end{equation}
	to almost surely define a random analytic function in the unit disk $\bD$. 
	
	Random analytic functions of the form~\eqref{eq:intro-def_of_infinite_power_series} naturally emerge as distributional limits of random Kac polynomials on compact subsets of $\bD$. In fact, Ibragimov and Zaporozhets~\cite{Ibragimov-Zaporozhets} showed that~\eqref{eq:intro-log_moment_assumption} is also necessary and sufficient for the clustering to the unit circle of the random roots of the corresponding Kac polynomial as the degree tends to infinity. The study of the behavior of the zeros of $F$ has a rich history. Rather than providing a comprehensive overview, we refer the interested reader to the classical book of Kahane~\cite{Kahane} to learn more about this subject. Our focus here is on the question of double zeros of $F$, which leads to the following result, proved in Section~\ref{sec:proof_of_almost_sure_result_in_the_disk}.
	\begin{theorem}
		\label{thm:almost_sure_double_zeros_in_disk}
		Let $F$ be a random Taylor series defined via~\eqref{eq:intro-def_of_infinite_power_series}, 
        assuming that~\eqref{eq:intro-log_moment_assumption} holds. Then
		\[
		\bP\Big[\exists \alpha\in \bD  \ \text{such that } \, F(\alpha) = F^\prime(\alpha) = 0\Big] =  \bP\Big[F(0) = F^\prime(0) = 0\Big]  = \big(\bP[a_0 = 0]\big)^2 \, .
		\]
	\end{theorem}
	In the special case when $a_0$ is Gaussian, Theorem \ref{thm:almost_sure_double_zeros_in_disk} was proven in \cite[Lemma 28]{Peres-Virag} (see also \cite[Lemma 2.4.1]{GAFbook}).  Among the main challenges of proving Theorem \ref{thm:almost_sure_double_zeros_in_disk} is that the behavior of zeros of $F$ away from the unit circle is \emph{not} universal.  As a basic example, if $a_j \in \{-1,1\}$ for all $j$, then $F$ has no roots of modulus less than $1/2$; conversely, in the Gaussian case $F$ can have a root arbitrarily close to the origin. The proof of Theorem~\ref{thm:almost_sure_double_zeros_in_disk} is self-contained and we outline the argument more extensively in Section \ref{sec:proof_of_almost_sure_result_in_the_disk}. We  also point out that Section~\ref{sec:proof_of_almost_sure_result_in_the_disk} can be read independently from all other parts of the paper, yet is necessary for the proof of Theorem \ref{thm:poisson_limit_for_close_roots}.

	\subsection{Outline of the proof}
	\label{subsection:outline_of_prood}
	We already mentioned that most of the roots of $f_n$ are present in the annulus $$\mathcal{A}_K = \Big\{ 1-\frac{K}{n}\le |z|\le 1+\frac{K}{n}\Big\}$$ as soon as $K$ is large but fixed and $n\to \infty$. Indeed, the main contribution to the limiting Poisson process described in Theorem~\ref{thm:poisson_limit_for_close_roots} is from pairs of roots from $\mathcal{A}_K$. Broadly speaking, the proof is split into three main steps:
	\begin{enumerate}
		\item[(A)] Proving that the distances between the roots in $\mathcal{A}_K$ converge to a Poisson point process at scale $n^{-5/4}$ with intensity $\mathfrak{c}_\ast(K) t^3 {\rm d}t$ on non-negative reals $\bR_{\ge 0}$ as $n\to \infty$. We then also verify that $\displaystyle \mathfrak{c}_\ast = \lim_{K\to\infty } \mathfrak{c}_\ast(K)$ exists and is non-zero. See Theorem~\ref{thm:poisson_convergence_near_unit_circle} for a precise formulation of this step. 
		\item[(B)] Proving that roots which are present in $\big\{r_0\le |z| \le 1-\frac{K}{n}\big\}$ do not contribute to the limiting Poisson point process described above, with high probability as $n\to \infty$ followed by $K\to \infty$ and then $r_0 \to 1$. See Proposition~\ref{prop:no_close_double_roots_in_the_bulk} for a precise formulation of this step.
		\item[(C)] Finally, apply Theorem~\ref{thm:almost_sure_double_zeros_in_disk} to show there are no close roots in the disk $\{|z|\le r_0\}$ which contribute to the limiting process as $n\to \infty$, for all $r_0<1$ fixed. This step follows from a simple monotonicity argument, while the proof of Theorem~\ref{thm:almost_sure_double_zeros_in_disk} is given in Section~\ref{sec:proof_of_almost_sure_result_in_the_disk}.
	\end{enumerate}
	We remark that since $f_n(z)$ has the same law as the polynomial $z^n f_n(z^{-1})$, then steps (B) and (C) in the above program automatically take care of roots present in $\{|z| \geq 1 + K/n\}$.
    
    The main step in the above program is (A), and to explain it the reader can assume for the moment that $f_n$ is in fact a Gaussian polynomial. The covariance kernel of $f_n$ is given by
	\[
	\bE\big[f_n(z) \overline{ f_n (w)} \big] = \sum_{k=0}^{n} (z\overline w)^k =  \frac{1-(z\overline{w})^{n+1}}{1-z\overline w} \, .
	\] 
	Since $\bE[|f_n(z)|^2] \asymp n$ for $z\in \mathcal{A}_K$, this implies in particular that for $z,w\in \mathcal{A}_K$ for which $|z-w| = \omega(n^{-1})$, the Gaussian variables $f_n(z)$ and $f_n(w)$ are roughly independent. As we are after a ``local statistic" (specifically, finding roots which are very close), a natural idea is to consider a sufficiently dense net of points, with the hope that the random vector $\big(f_n(z),f_n^\prime(z),f_n^{\prime\prime}(z)\big)$ for a net point $z\in \mathcal{A}_K$ will dictate whether or not there exist close roots nearby. The reason for considering the polynomial and its first two derivatives (and nothing more) is that under a typical event for third derivative, the polynomial is close to its quadratic approximation on scales $\ll n^{-1}$, which in principle can predict the existence of a pair of close roots. Although the above reasoning is formally correct, to get the asymptotic~\eqref{eq:intro_prob_for_two_roots_in_disk} we need to use the randomness to show that typically, solving the quadratic approximation is similar to solving two \emph{linear approximations}.  The latter is much better suited for computing small-ball probabilities. 
	
	To be more precise, we will consider a $\delta$-net of points in $\mathcal{A}_K$, where $\delta = n^{-5/4 - \beta}$ and $\beta>0$ is small but fixed. Let $\eps = n^{-5/4}$. To capture roots which are $\eps$-separated, we will first use $(f_n(z),f_n^\prime(z))$ to capture roots which are $\delta$-close to a given net point. By applying a standard linearization, we predict a root should occur near
	\[
	\alpha = z- \frac{f_n(z)}{f_n^\prime(z)} \, ,
	\]
	and hence
	\[
	\Big\{ \exists \alpha\in \bD(z,\delta) \, : \, f_n(\alpha) = 0 \Big\} \approx \Big\{ \frac{f_n(z)}{f_n^\prime(z)} \in \bD(0,\delta) \Big\}	
	\] 
	where by $\approx$ we just mean that with high-probability both event happen (or don't happen) simultaneously. Once we know there is a root $\alpha\in \bD(z,\delta)$, we will Taylor expand \emph{at the root}, which gives the prediction for another root at
	\[
	\alpha^\prime = \alpha - \frac{2f_n^\prime(\alpha)}{f_n^{\prime\prime}(\alpha)}\, .
	\]
	On a typical event for close roots to occur, we have the asymptotic 
	\[
	\alpha^\prime \approx z+\frac{f_n(z)}{f_n^\prime(z)} - \frac{2 f_n^\prime(z)}{f_n^{\prime\prime}(z)}
	\]
	(see Lemma~\ref{lemma:quadratic_approximation_predicts_second_root}) and by neglecting multiplicative constants, we conclude that
	\begin{equation}
		\label{eq:intro-explanation_how_to_compute_probability_double_roots}
		\bigg\{\exists \alpha\in \bD(z,\delta) \, , \exists \alpha^\prime\in \bD(z,\eps)\setminus \{\alpha \} \, : \, f_n(\alpha) = f_n(\alpha^\prime) = 0 \bigg\} \approx \bigg\{\Big|\frac{f_n(z)}{f_n^\prime(z)}\Big| \le \delta \, , \, \Big|\frac{2f_n^\prime(z)}{f_n^{\prime\prime}(z)}\Big| \le \eps \bigg\} \, .
	\end{equation}
	For Gaussian polynomials, one can compute to first order the probability of the event on the right-hand side of~\eqref{eq:intro-explanation_how_to_compute_probability_double_roots} (see Section~\ref{sec:limiting_intensity_for_gaussian}) and get that
	\begin{equation}
		\label{eq:intro-probability_of_local_event}
		\bP\bigg[ \, \Big|\frac{f_n(z)}{f_n^\prime(z)}\Big| \le \delta \, , \, \Big|\frac{2f_n^\prime(z)}{f_n^{\prime\prime}(z)}\Big| \le \eps \bigg] \approx \eps^4 n^6 \delta^2 = n^{-3/2-2\beta} \, .
	\end{equation}
	Since the number of net points is $n^{-1} \delta^{-2} = n^{3/2 + 2\beta}$, we get that the number of $\eps$-separated roots is of constant order. The full Poisson limit follows from the method of moments, so we need to compute the probability of the event~\eqref{eq:intro-explanation_how_to_compute_probability_double_roots} intersected over finite collections of net points in $\mathcal{A}_K$. We also remark that if a net point predicts two close roots nearby then all other net points of distance $o(n^{-1})$ are likely not to predict two close roots (see Lemma~\ref{lemma:existence_of_two_roots_implies_macroscopic_separation}). Therefore, the contribution from the net sum to the set of close roots are asymptotically independent, which is crucial for a Poisson limit to hold.  
	
	For non-Gaussian polynomials, the asymptotic~\eqref{eq:intro-probability_of_local_event} is not true for certain points $z\in \mathcal{A}_K$ with bad arithmetic properties. As a toy example, consider the (much simpler) event that $\{|f_n(z)| \leq \sqrt{n} \gamma \}$ for some small $\gamma>0$.  When the coefficients are Gaussian and $z \in \mathcal{A}_K$, then one has $f_n(z)/\sqrt{n}$ is a (uniformly non-degenerate) two-dimensional Gaussian and so one anticipates \begin{equation}\label{eq:gaussian-small-ball}
		\P\big[|f_n(z)| \leq \sqrt{n} \gamma\big] \asymp \gamma^2
	\end{equation} for all $z \in \mathcal{A}_K$ (away from the real axis) and all $\gamma > 0$.  On the other hand, in the case the coefficients are, say, uniformly distributed in $\{-1,0,1\}$, one can show that \begin{equation*}
		\P\big[f_n(e^{i\pi/2}) = 0\big] \asymp \frac{1}{n} \, .
	\end{equation*}
	That is, at the point $e^{i\pi/2} \in \mathcal{A}_K$ the Gaussian heuristic \eqref{eq:gaussian-small-ball} fails for small $\gamma$. If we have any hope of proving~\eqref{eq:intro-probability_of_local_event}, which handles a much more complicated event than \eqref{eq:gaussian-small-ball}, we need to remove poorly-behaved points such as $e^{i\pi/2}$. By now, it is well-understood in the random polynomial literature that the key property of $e^{i\pi/2}$ that causes the failure of \eqref{eq:gaussian-small-ball} is that its argument is close to a rational multiple of $\pi$ with small denominator (of course, $\frac{\pi}{2}$ is genuinely equal to such a multiple). 
    
    The relationship between arithmetic structure and small-ball probabilities is a recurring  theme in the study of both random polynomials and random matrices.  Investigation into this relationship goes back at least to Hal\'asz \cite{Halasz} and is now often referred to as the \emph{inverse Littlewood-Offord problem} (see, e.g., \cite{Tao-Vu-RSA,Nguyen-Vu-survey}). We will need to remove these points with bad arithmetic properties from our net and prove~\eqref{eq:intro-probability_of_local_event} for \emph{smooth points} (see Definition~\ref{def:smooth_angle}). In turn, this follows from proving a quantitative local Gaussian comparison for those smooth points (and for tuples of well-separated smooth points), see Theorem~\ref{thm:small_ball_comparison_to_Gaussian_for_tuples}. The analysis here is based on an idea by Konyagin and Schlag~\cite{Konyagin-Schlag} which was later generalized substantially by Cook-Nguyen in~\cite{Cook-Nguyen}. For our application, we need to push the method a bit further, and prove the local Gaussian comparison also for points which are \emph{not} on the unit circle (both \cite{Konyagin-Schlag, Cook-Nguyen} work only on the unit circle). Proving different versions of this local Gaussian comparison spans through Sections~\ref{sec:small_ball_for_general_points}, \ref{sec:small_ball_bounds_for_smooth_points} and \ref{sec:gaussian_comparison_for_tuples} of this paper.
	
	To prove item (B) of our program, we follow similar steps as in item (A). The key observation is that in the annulus
	\[
	\big\{r\le |z| \le r+ \frac{1-r}{2}\big\}
	\]
	the original polynomial $f_n$ ``behaves" effectively like a random Kac polynomial with degree $(1-r)^{-1}$. Therefore, to show there are no close roots in the annulus $\{r_0\le |z|\le 1-K/n\}$, we sum dyadically over concentric annuli, and show that the sum of contributions is arbitrary small as $r_0\to 1$ and $K\to \infty$. En route of proving item (B), we will need to prove small ball probability bounds for the polynomial at arbitrary points $z\in \bD$ which takes into account the effective degree of the polynomial at the point $d_n(z) = \min\{n,(1-|z|)^{-1}\}$. These bounds are developed in Section~\ref{sec:small_ball_for_general_points} and Section~\ref{sec:small_ball_bounds_for_smooth_points}. As we want to ``push" our results for discrete coefficients as well, we will again need to remove from consideration non-smooth points and obtain the desired small-ball bounds (of the form~\eqref{eq:intro-probability_of_local_event}) for smooth points, see Lemma~\ref{lemma:small_ball_probability_smooth_points_bound_single_point}. 
	
	Finally, we need to complete item (C), that is, we need to deal with those roots which stay strictly inside the unit disk as $n\to \infty$.  As there is no Gaussian-like behavior strictly inside the unit disk (unless, of course, the coefficients $\xi_k$ are Gaussian) we need to come up with a different strategy. The idea is to consider the ``full" Taylor series instead of the finite polynomial (that is, $n=\infty$) and apply a perturbation argument to show that this Taylor series (of the form~\eqref{eq:intro-def_of_infinite_power_series}) does not have a double zero almost surely. This is precisely Theorem~\ref{thm:almost_sure_double_zeros_in_disk}, which we prove in Section~\ref{sec:proof_of_almost_sure_result_in_the_disk}. We also refer the reader to the beginning of Section~\ref{sec:proof_of_almost_sure_result_in_the_disk} for a more elaborate sketch of proof for Theorem~\ref{thm:almost_sure_double_zeros_in_disk}. Once Theorem~\ref{thm:almost_sure_double_zeros_in_disk} is established, we can argue  that with high probability, all roots of $f_n$ in the disk will be uniformly separated as $n\to \infty$, which concludes item (C).
	
	\subsection{Organization of the paper and notation}
	To ease the readability, we outline a bird's eye view for each section below:  
	\begin{itemize}
		\item[$\circ$] In Section~\ref{sec:proof_of_almost_sure_result_in_the_disk} we give the proof of Theorem~\ref{thm:almost_sure_double_zeros_in_disk} concerning double roots of random power series. As we already mentioned, this section can be read independently of all other parts of the paper. 
		
		\item[$\circ$] In Section~\ref{sec:reducing_main_result_to_a_net_argument} we state our program more precisely and formulate Proposition~\ref{prop:no_close_double_roots_in_the_bulk} and Theorem~\ref{thm:poisson_convergence_near_unit_circle} described above.  This is where we describe how to reduce the proof of Theorem \ref{thm:poisson_limit_for_close_roots} to a Poisson limit for a sum over a net.  In this section, we state the relevant small-ball probability bounds/comparisons that will be used throughout the paper.  
		
		\item[$\circ$] Section~\ref{sec:reducing_to_the_annulus} is devoted to the proof of Proposition~\ref{prop:no_close_double_roots_in_the_bulk} which allows us to reduce our focus to the annulus $\mathcal{A}_K$ for large but fixed $K$.  This amounts to eliminating close pairs of roots that are more than distance $1/n$ from the unit circle. 
		
		\item[$\circ$] Section~\ref{sec:reducing_poisson_limit_to_count_over_a_net} is the starting point for our net argument described above. Here, we reduce the question of existence of close roots to a question about local events in $(f_n,f_n^\prime,f_n^{\prime\prime})$ on an appropriately chosen net of points in the annulus $\mathcal{A}_K$. 
		
		\item[$\circ$] In Section~\ref{sec:poisson_limit_for_sum_over_the_smooth_net} we show that the sum over the net points indeed has a Poisson scaling limit, assuming that the small-ball probability results from Section~\ref{sec:reducing_main_result_to_a_net_argument} hold. From this point onward, the rest of the paper is devoted to proving these small-ball bounds.  
		
		\item[$\circ$] In Section~\ref{sec:small_ball_for_general_points} we derive small-ball bounds which are valid to all sample points in $\bD$. As such, these small-balls are relatively weak, and do not allow us to compute to first order events of the form~\eqref{eq:intro-explanation_how_to_compute_probability_double_roots}.
		
		\item[$\circ$] In Section~\ref{sec:small_ball_bounds_for_smooth_points} we improve the small-ball bound from the previous section, in the case where the sample point $z\in \bD$ is \emph{smooth} (i.e.\ has some favorable arithmetic property, see Definition~\ref{def:smooth_angle} for a precise formulation). This improvement allows us to  estimate the probability of~\eqref{eq:intro-probability_of_local_event} for the remaining sample points accurately. 
		
		\item[$\circ$] In Section~\ref{sec:gaussian_comparison_for_tuples} we take one step further, and prove that small-ball probabilities for tuples of smooth, separated points in $\mathcal{A}_K$, scale as in the case of Gaussian coefficients. This type of Gaussian-comparison is crucial when computing moments of the sum of indicators over the net, and allow us to compute the actual intensity (and, in particular, pin-down the non-homogeneous effect) for the limiting Poisson point process. 
		
		\item[$\circ$] Finally, in Section~\ref{sec:limiting_intensity_for_gaussian} we explicitly compute the limiting intensity for the Poisson point process. By the program described above, this task is reduced to computing the probability of the event~\eqref{eq:intro-probability_of_local_event} to first asymptotic order in the case of Gaussian coefficients.  
	\end{itemize}
	
	We end the introduction by listing some notations that we will use freely across the paper:
	\begin{itemize}
		\item $\bC, \bR$ the complex plane and the real line. We will always identify $\bC^d \simeq \bR^{2d}$, $m$ the Lebesgue measure on $\bR^d$ (the dimension should be clear from the context);
		\item $\bR_{\ge 0}$ the non-negative reals; $\bD$ the unit disk; $\mathbb{H}$ the upper-half plane;
		\item $\bD(z,r)$ a disk centered at $z$ with radius $r$; $\mathcal{A}(a,b) = \bD(0,b)\setminus \bD(0,a)$; 
		\item $\xi_0$ the coefficient distribution of the random polynomial $f_n$ given by~\eqref{eq:intro_def_f_n}; $\varphi(t) = \bE\big[\exp(it\xi_0)\big]$ the characteristic function of $\xi_0$;
		\item For $z\in \bD$, the \emph{effective degree} of $f_n$ at $z$
		\begin{equation}
			\label{eq:def_of_effective_degree_intro}
			d_n(z) = \min\big\{n,(1-|z|)^{-1}\big\} ;
		\end{equation} 
		\item $\tau>\beta>0$ are absolute, sufficiently small constants (taking $\tau = 10^{-4}$ and $\beta = \tau/20$ is fine). 
	\end{itemize}
	We will use freely the Landau notations $O(\cdot),o(\cdot),\omega(\cdot),\Theta(\cdot)$ to denote inequalities up to non-asymptotic constants. We will also write $a\lesssim b$ if $a= O(b)$.
	
	\subsection*{Acknowledgments}
	We would like to thank Alon Nishry and Mikhail Sodin for helpful discussions. 
	M.M.\ is supported in part by NSF CAREER grant DMS-2336788 as well as DMS-2246624. O.Y.\ is supported in part by NSF postdoctoral fellowship DMS-2401136.

	\section{Almost sure result in the disk: Proof of Theorem~\ref{thm:almost_sure_double_zeros_in_disk}}
	\label{sec:proof_of_almost_sure_result_in_the_disk}
	
	Throughout this section we are concerned with the zeros of the random power series $F$ defined at \eqref{eq:intro-def_of_infinite_power_series}, whose i.i.d.\ coefficients satisfy \eqref{eq:intro-log_moment_assumption}.  
	To prove Theorem~\ref{thm:almost_sure_double_zeros_in_disk}, it is enough to show that for each $\delta >0$, almost surely there is no double zero of $F$ in the annulus $\mathcal{A}(\delta,1-\delta) = \{\delta < |z| < 1-\delta\}$. 
	
	The main idea is that if $F$ has a double zero at some $z_0 \in \mathcal{A}(\delta,1-\delta)$ for which $|F''(z_0)| > 0$, then we decompose $$F(z) = \sum_{k = 0}^m a_k z^k + \sum_{k > m} a_k z^k =: P_m(z) + T_m(z)$$ and show that $P_m^\prime$ has a root close to $z_0$ on which $F$ is atypically small (provided $m$ is large enough). To omit the possibility of a double zero, we first reveal (i.e.\ condition on) the polynomial $P_m$. On the event that $F$ has a double zero, it must be the case that among the finitely many roots of $P_m^\prime$, we have that $T_m$ nearly cancels $P_m$. We show that this latter event is rare by using a basic anti-concentration fact due to L\'evy-Kolmogorov-Rogozin (see Lemma~\ref{lemma:Levy-Kolmogorov-Rogozin}), which is a variant on Erd\H{o}s solution to the Littlewood-Offord problem.  
	
	A wrinkle to this story is that for this argument to work, we need $|F''(z_0)| > 0$.  To get around this, we observe that with a basic quasirandomness condition on the coefficients, there does not exists a zero with extremely high multiplicity, i.e.\ for a large $k$ it is not likely for there to exist a point $z_0$ for which $$F(z_0) = F'(z_0) = F''(z_0) = \cdots = F^{(k)}(z_0) = 0\,.$$
	See Claim \ref{claim:jensen_bound_on_the_number_of_zeros_E_K_beta_ell} for a precise statement.  We then let $k_\ast$ be the smallest $k$ for which there is positive probability that there is some zero of multiplicity $k_\ast$.  If $k_\ast \geq 2$, then we run the above argument with $G = F^{(k_\ast - 2)}$, and by the minimality of $k_\ast$ we have that each root of $F$ of multiplicity $k_\ast$ must be a double root of $G$ and have non-zero second derivative.

	\subsection{Preparations}
	In what follows, it will be convenient to assume bounds on the growth of the random coefficients, which will be shown to hold with high probability. For fixed constants $K,\beta>0$ define the set of analytic functions in the disk
	\begin{equation}
		\label{eq:def_of_E_k_beta}
		\mathcal{E}_{K,\beta} = \bigg\{ F(z) = \sum_{k=0}^{\infty} a_k z^k : \, |a_k| \le K e^{\beta k} \ \text{for all } k\ge 0\bigg\} \, .
	\end{equation} 
	\begin{claim}
		\label{claim:F_is_likely_in_E_K_beta}
		Let $\{a_k\}_{k\ge 0}$ be an i.i.d.\ sequence which satisfies~\eqref{eq:intro-log_moment_assumption}. Then for each $\beta>0$ and for all $\eps>0$ there exists $K\ge 1$ large enough so that $\bP\big[F\in \mathcal{E}_{K,\beta}\big] \ge 1-\eps$.
	\end{claim}
	\begin{proof}
		By the moment assumption~\eqref{eq:intro-log_moment_assumption} on the random coefficients we have
		\begin{equation*}
			\sum_{k\ge 0} \bP\big[|a_k| \ge e^{\beta k}\big] \le \sum_{k\ge 0} \bP\big[\log \big(1+ |a_0|\big) \ge \beta k\big] \lesssim \beta^{-1} \bE\log\big(1+|a_0|\big) < \infty \, .
		\end{equation*}
		Therefore, by the Borel-Cantelli lemma, the random variable $\max_{k\ge 0} |a_k| e^{-\beta k}$ is finite almost surely. Taking $K$ sufficiently large now gives
		\[
		\bP\big[\max_{k\ge 0} |a_k| e^{-\beta k} \ge K\big] \le \eps 
		\]
		as desired. 
	\end{proof}
	For what follows, we would also like to omit functions with deep zeros at the origin. Let $c_0>0$ be chosen so that $\tau = \bP\big[|a_0| \ge c_0\big]>0$ (this is possible as we assume that $a_0$ is non-degenerate). For $\ell \ge 0$ we define another set of analytic functions 
	\begin{equation}
		\label{eq:def_of_F_ell}
		\mathcal{F}_\ell = \bigg\{ F(z) = \sum_{k=0}^{\infty} a_k z^k : \, \exists j\in \{0,\ldots,\ell\} \ \text{such that } |a_j|\ge c_0 \bigg\}\, .
	\end{equation}
	\begin{claim}
		\label{claim:F_and_its_derivatives_are_in_F_ell}
		With the notations described above, for all $k\ge 0$ and $\ell \ge 0$ we have $$\bP\big[F^{(k)} \in \mathcal{F}_\ell \big] \ge 1-e^{-\tau \ell}\, .$$  
	\end{claim}
	\begin{proof}
		Simply note that
		\[
		\bP\big[F^{(k)} \not\in \mathcal{F}_\ell \big] \le \bP\big[|a_{k+j}|<c_0 \ \text{for all } j=0,\ldots,\ell \, \big] = \Big(1-\bP\big[|a_0|\ge c_0\big]\Big)^{\ell +1}
		\]
		and the bound follows from the definition of $\tau = \bP\big[|a_0|\ge c_0\big]$. 
	\end{proof}
	We will make use of the Cauchy estimates which we now recall. For $0<s<r$ and $F$ an analytic function in the disk $r\bD$, we have
	\begin{equation}
		\label{eq:cauchy_estimates}
		\max_{z\in s\bD} |F^{(k)}(z)| \le \frac{k!}{(r-s)^k} \max_{w\in r\bD} |F(w)| \, .
	\end{equation}
	In the next claim, we upper bound the number of zeros in $(1-\delta)\bD$ of analytic functions from the set $\mathcal{E}_{K,\beta,\ell} = \mathcal{E}_{K,\beta} \cap \mathcal{F}_\ell$. 
	\begin{claim}
		\label{claim:jensen_bound_on_the_number_of_zeros_E_K_beta_ell}
		For each $K,\ell ,\beta, \delta>0$ such that $1-\delta/2 < e^{-\beta}$ there is a constant $M=M(K,\ell,\beta,\delta)$ such that each function $h\in \mathcal{E}_{K,\beta,\ell}$ has at most $M$ zeros (counting multiplicity) in $(1-\delta)\bD$.
	\end{claim} 
	\begin{proof}
		Fix some analytic function in the disk $h$, and denote by $N$ the number of zeros it has in $(1-\delta)\bD$ with multiplicities. The classical Jensen-type bound gives
		\begin{equation}
			\label{eq:jensen_bound_for_analytic_functions}
			N \le C_\delta \log \frac{\max_{(1-\delta/2)\bD} |h|}{\max_{(1-\delta)\bD} |h|} \, .
		\end{equation}
		For the reader's convenience, we recall the proof of~\eqref{eq:jensen_bound_for_analytic_functions}. Let $\zeta_1,\ldots,\zeta_N$ be the multiple zeros of $h$ and recall that the Blaschke factor is given by
		\[
		B_\zeta(z) = \frac{\rho(z-\zeta)}{\rho^2 - z\overline\zeta} \, , \qquad \rho = 1-\delta/2 \, .
		\]
		Note that $|B_\zeta(z)|=1$ for $|z|= \rho$ and that $\max_{(1-\delta)\bD}|B_\zeta| \le \frac{2-2\delta}{4-3\delta} =: C_\delta^\prime.$ Now put $h= B_{\zeta_1} \cdots B_{\zeta_N} g$ and observe that by the maximum principle
		\begin{equation*}
			\max_{(1-\delta)\bD} |h| \le C_\delta^N \max_{(1-\delta)\bD} |g| \le C_\delta^N \max_{(1-\delta/2)\bD} |g| = (C_\delta^\prime)^N \max_{(1-\delta/2)\bD} |h| \, ,
		\end{equation*}
		which yields~\eqref{eq:jensen_bound_for_analytic_functions}. To conclude the proof, we note that for each $h\in \mathcal{F}_\ell$ the Cauchy estimates~\eqref{eq:cauchy_estimates} imply that
		\begin{equation*}
			c_0 \le \max_{j\le \ell} |a_j| = \max_{j\le \ell} \frac{|F^{(j)}(0)|}{j!} \le \frac{\max_{(1-\delta)\bD} |h|}{2\pi} \, ,
		\end{equation*}
		whereas for each $h\in \mathcal{E}_{K,\beta}$ we have
		\[
		\max_{(1-\delta/2)\bD}|h| \le K\sum_{j=0}^{\infty} (e^\beta(1-\delta/2))^j \, .
		\]
		Plugging both these inequalities into~\eqref{eq:jensen_bound_for_analytic_functions} yields the claim.
	\end{proof}
	
	\subsection{The perturbative argument}
	In the proof of Theorem~\ref{thm:almost_sure_double_zeros_in_disk}, we will make use of a standard anti-concentration inequality, which may be viewed as a generalization of Erd\H{o}s's solution to the Littlewood-Offord problem to arbitrary non-degenerate random variables.  A version in this generality follows from the L\'evy-Kolmogorov-Rogozin inequality, see for example~\cite{Esseen1968}.
	
	\begin{lemma}({\normalfont \cite[Corollary~1]{Esseen1968}})
		\label{lemma:Levy-Kolmogorov-Rogozin}
		Let $a_j$ be an i.i.d.\ sequence of non-degenerate (complex) random variables and let $c_j \in \bC$ be such that $|c_j| \ge 1$. There is a constant $C>0$ depending only on the distribution of $a_0$ so that for all $R\ge 1$ we have
		\[
		\sup_{ x \in \bC} \bP\bigg[\Big|\sum_{j=1}^{n} a_j c_j - x\Big| \le R\bigg] \le \frac{CR}{\sqrt{n}} \, .
		\]
	\end{lemma}
	
	For $m\ge 1$ we will decompose $F(z) = P_m(z) + T_m(z)$, where
	\begin{equation}
		\label{eq:taylor_polynomial_of_F_around_origin}
		P_m(z) = \sum_{k=0}^{m} a_k z^k 
	\end{equation}
	is the Taylor polynomial of degree $m$ of $F$ around the origin.  We first deduce uniform anti-concentration for the first $m$ derivatives of $T_m$ using Lemma \ref{lemma:Levy-Kolmogorov-Rogozin}: 
	\begin{claim}
		\label{claim:Levy_concentration_for_derivatives_of_tail}
		There exists $C>0$ depending only on the distribution of $a_0$ so that for all $z\in \bD$, $\eps\in(0,1)$ and $0\le k\le m$ we have
		\[
		\sup_{x \in \bC} \bP \Big[\big|T_m^{(k)}(z) - x \big| \le 
		\eps |z|^m\Big] \le C \big(\log (1/|z|) / \log (1/\eps)\big)^{-1/2} \, .
		\]
	\end{claim}
	\begin{proof}
		Let $N = \lfloor\log(\eps^{-1})/2\log(|z|^{-1})\rfloor$, and note that with this choice of $N$ for each $j \in\{m+1,\ldots,m+N\}$ we have 
		\[
		\eps^{-1} |z|^{j-k - m} \frac{j!}{(j-k)!} \ge \eps^{-1} |z|^{j-m} \ge 1 \, .
		\] 
		We thus have
		\begin{align*}
			\sup_{x \in \bC} \bP \Big[\big|T_m^{(k)}(z) - x \big| \le 
			\eps |z|^m\Big] &= \sup_{x \in \bC} \bP \left[\left|\sum_{j=m+1}^{\infty} a_j \frac{j!}{(j-k)!} z^{j-k} - x \right| \le 
			\eps |z|^m\right] \\ & \le \sup_{x \in \bC} \bP \left[\left|\sum_{j=m+1}^{m+N} a_j \frac{j!}{(j-k)!} \eps^{-1} |z|^{-m} z^{j-k} - x \right| \le 
			1\right] \stackrel{\text{Lemma}~\ref{lemma:Levy-Kolmogorov-Rogozin}}{\le} \frac{C}{\sqrt{N}} 
		\end{align*}
		which, by our choice of $N$, completes the proof.
	\end{proof}
	We want to show that if $F$ has a double zero, then for large $m$ the polynomial $P_m^\prime$ has a root nearby on which $F$ is small. This takes the form of the following lemma, which provides the engine behind the proof of Theorem~\ref{thm:almost_sure_double_zeros_in_disk}.  
	\begin{lemma}
		\label{lemma:double_zeros_and_large_second_derivative_implies_smallness_on_critical_point}
		Let $h(z) = \sum_{k \geq 0} h_k z^k\in \mathcal{E}_{K,\beta}$ for some $K>0$ and $\beta\in(0,1)$, and for $m\ge 1$ let $
		p_m(z) = \sum_{k=0}^{m} h_k z^k$ denote 
		its Taylor polynomial around the origin. For each $\gamma>0$ there exists $C=C(\gamma,K,\beta)$ so that the following holds. If 
		\[
		h(z_0) = h^\prime(z_0) = 0 \quad \text{and} \quad |h^{\prime\prime}(z_0)| \ge \gamma 
		\]
		for some $z_0\in \mathcal{A}(\beta,e^{-8\beta})$, then for each $m\ge C$ there exists $w\in \mathcal{A}(\beta e^{-\beta},e^{-7\beta})$ such that $p_m^\prime(w) = 0$ and $|h(w)| \le C e^{-\beta m} |w|^m$.  
	\end{lemma}
	\begin{proof}
		As $h\in \mathcal{E}_{K,\beta}$, we have that
		\begin{equation}
			\label{eq:bound_on_derivatives_of_h}
			\max_{0\le j\le 3} \max_{z \in e^{-2\beta} \bD} |h^{(j)}(z)|\le K\sum_{k=3}^{\infty} k(k-1)(k-2) e^{-\beta k} \le 6K\frac{e^{3\beta}}{(1-e^{-\beta})^4} =: C_0\, ,
		\end{equation}
		and similarly, for all $z\in e^{-2\beta} \bD$
		\begin{equation*}
			|(h-p_m)^{\prime}(z)|\le \frac{Ke^{2\beta}}{1-e^{-\beta}} m\, |ze^{\beta}|^m \le C_0 |ze^{\beta}|^m  \, .
		\end{equation*}
		By setting $\eta = 4C_0 \gamma^{-1}  m \,  |z_0 e^{2\beta}|^m$, we see that
		\begin{equation}
			\label{eq:bound_on_derivative_of_tail_on_circle}
			\max_{|y|=\eta} \big|(h-p_m)^\prime(z_0+y)\big| \le C_0 (|z_0| e^{\beta} + \eta)^m \le C_0 |z_0 e^{2\beta}|^m 
		\end{equation}
		for all $m$ large enough. Since $h^\prime(z_0) = 0$, for each $y$ on the circle $\{|y|=\eta\}$, we can bound
		\[
		|h^\prime(z_0 + y)| \stackrel{\eqref{eq:bound_on_derivatives_of_h}}{\ge} \gamma |y| - C_0|y|^2 \ge \frac{\gamma}{2} \eta = 2C_0 m \, |z_0 e^{2\beta}|^m \stackrel{\eqref{eq:bound_on_derivative_of_tail_on_circle}}{>} \max_{|y|=\eta} |(h-p_m)^{\prime}(z_0 + y)|
		\]
		where the second inequality is for all $m$ large enough. 
		Rouch\'e's theorem now implies that $p_m^\prime$ has a root $w$ in the disk of radius $\eta$ centered at $z_0$. Furthermore, note that by Taylor's theorem
		\[
		|h(w)| \le \frac{\eta^2}{2} \max_{e^{-3\beta} \bD} |h^{(2)}| \stackrel{\eqref{eq:bound_on_derivatives_of_h}}{\le} \frac{\eta^2}{2} C_0 = 8 C_0^3 \gamma^{-2} m^2  |z_0 e^{2\beta}|^{2m} \, .
		\]   
		We can take $m$ sufficiently large (depending on $\beta$) so that we can guarantee $w\in \mathcal{A}\big(e^{-\beta}|z_0|, e^{\beta}|z_0|\big)$, which in turn shows
		\[
		|h(w)| \le C_0 |w e^{3\beta}|^{2m} \le C_0 e^{-\beta m} |w|^m\, ,
		\]
		which completes the proof of the lemma.
	\end{proof}
	We are now ready to prove the main result of this section.
	\begin{proof}[Proof of Theorem~\ref{thm:almost_sure_double_zeros_in_disk}]
		As we already mentioned, it suffices to show that for each $\delta>0$ we have
		\begin{equation*}
			\bP\Big[\exists \alpha\in \mathcal{A}(\delta,1-\delta) \ \text{such that } \ F(\alpha) = F^\prime(\alpha) = 0\Big] = 0 \, .
		\end{equation*}
		Recall that $\mathcal{E}_{K,\beta,\ell} = \mathcal{E}_{K,\beta} \cap \mathcal{F}_\ell$ where the events $\mathcal{E}_{K,\beta}$ and $\mathcal{F}_\ell$ are defined by~\eqref{eq:def_of_E_k_beta} and~\eqref{eq:def_of_F_ell}, respectively. By combining Claim~\ref{claim:F_is_likely_in_E_K_beta} and Claim~\ref{claim:F_and_its_derivatives_are_in_F_ell}, the above equality will follow once we show that for all $K,\ell \ge 1$ and $\beta>0$ small enough we have
		\begin{equation*}
			\bP\Big[F\in \mathcal{E}_{K,\beta,\ell} \ \text{and there exists} \ \alpha\in \mathcal{A}(\delta,1-\delta) \ \text{such that } \ F(\alpha) = F^\prime(\alpha) = 0\Big] = 0 \, .
		\end{equation*}
		To be slightly more general, let us define for $k\ge 0$ the event $\mathcal{G}_k=\mathcal{G}_k(K,\beta,\ell, \delta)$ by
		\begin{equation*}
			\mathcal{G}_k = \Big\{F\in \mathcal{E}_{K,\beta,\ell} \ \text{and there exists} \ \alpha\in \mathcal{A}(\delta,1-\delta) \ \text{s.t.} \ F(\alpha) = F^\prime(\alpha) = \ldots=F^{(k)}(\alpha) = 0 \Big\} \, .
		\end{equation*}
		Denoting by 
		\begin{equation}
			\label{eq:def_of_k_ast}
			k_\ast = \min\big\{k\ge 0 \, : \ \bP\big[\mathcal{G}_{k}\big] = 0 \big\} \, ,
		\end{equation}
		we see that the proof of Theorem~\ref{thm:almost_sure_double_zeros_in_disk} will follow once we show that $k_{\ast}\le 1$. We note that by Claim~\ref{claim:jensen_bound_on_the_number_of_zeros_E_K_beta_ell} we have $k_\ast \le M$, as the total number of zeros for $F\in \mathcal{E}_{K,\beta,\ell}$ is at most $M$. Seeking a contradiction, assume that $k_\ast \ge 2$. We have 
		\begin{align}
			\label{eq:probabaility_of_G_k_ast_minus1_is_positive}
			0 &< \bP\big[\mathcal{G}_{k_\ast - 1}\big] \\ \nonumber &= \bP\big[F\in \mathcal{E}_{K,\beta,\ell} \ \text{and} \ \exists\alpha\in \mathcal{A}(\delta,1-\delta) \ \text{s.t.} \ F(\alpha) = F^\prime(\alpha) = \ldots=F^{(k_\ast - 1)}(\alpha) = 0 \big]  \\ \nonumber &= \bP\big[F\in \mathcal{E}_{K,\beta,\ell} \ \text{and} \ \exists \alpha\in \mathcal{A}(\delta,1-\delta) \ \text{s.t.} \ F(\alpha) = F^\prime(\alpha) = \ldots=F^{(k_\ast - 1)}(\alpha) = 0 \, , \ |F^{(k_\ast)}(\alpha)|>0 \big]  
		\end{align}
		where the last equality is by the definition~\eqref{eq:def_of_k_ast} of $k_\ast$. As the sequence of events
		\[
		\Big\{F\in \mathcal{E}_{K,\beta,\ell} \ \text{and} \ \exists \alpha\in \mathcal{A}(\delta,1-\delta) \ \text{s.t.} \ F(\alpha) = F^\prime(\alpha) = \ldots=F^{(k_\ast - 1)}(\alpha) = 0 \, , \ |F^{(k_\ast)}(\alpha)|\ge \gamma \Big\}
		\] 
		increase as $\gamma \downarrow 0$, we conclude from~\eqref{eq:probabaility_of_G_k_ast_minus1_is_positive} there exists some $\gamma>0$ for which
		\[
		\bP\Big[F\in \mathcal{E}_{K,\beta,\ell} \ \text{and} \ \exists \alpha\in \mathcal{A}(\delta,1-\delta) \ \text{s.t.} \ F(\alpha) = F^\prime(\alpha) = \ldots=F^{(k_\ast - 1)}(\alpha) = 0 \, , \ |F^{(k_\ast)}(\alpha)|\ge \gamma \Big]  > 0 \, .
		\]
		In particular, we have
		\begin{equation}
			\label{eq:by_contradition_positive_probability_for_double_zero}
			\bP\Big[F\in \mathcal{E}_{K,\beta,\ell} \ \text{and} \ \exists \alpha\in \mathcal{A}(\delta,1-\delta) \ \text{s.t.} \ F^{(k_\ast - 2)}(\alpha) = F^{(k_\ast-1)}(\alpha) = 0 \, , \ |F^{(k_\ast)}(\alpha)|\ge\gamma \Big]  > 0
		\end{equation}
		which, as we will now show, is a contradiction for $k_\ast \ge 2$. 
		
		We start by noting that for $F\in \mathcal{E}_{K,\beta,\ell}$, we may decrease $\beta$ to some $\beta^\prime$ and increase $K,\ell$ to some $K^\prime,\ell^\prime$ and get that $F^{(k_\ast-2)} \in \mathcal{E}_{K^\prime,\beta^\prime,\ell^\prime}$. Furthermore as the event in~\eqref{eq:by_contradition_positive_probability_for_double_zero} increases as $\beta$ decreases, we may also assume that $\beta^\prime<\delta<1-\delta<e^{-8\beta^\prime}$. Recalling the decomposition $F(z) = P_m(z) + T_m(z)$ given by~\eqref{eq:taylor_polynomial_of_F_around_origin}, we can apply Lemma~\ref{lemma:double_zeros_and_large_second_derivative_implies_smallness_on_critical_point} with $h = F^{(k_\ast-2)}$ and deduce that
		\begin{align}
			\label{eq:inclusion_of_events_existence_of_critical_point_where_F_is_small}
			\Big\{ &F\in \mathcal{E}_{K,\beta,\ell} \ \text{and} \ \exists \alpha\in \mathcal{A}(\delta,1-\delta) \ \text{s.t.} \ F^{(k_\ast - 2)}(\alpha) = F^{(k_\ast-1)}(\alpha) = 0 \, , \ |F^{(k_\ast)}(\alpha)|\ge\gamma \Big\} \\ \nonumber &\subset \Big\{ F^{(k_\ast-2)} \in \mathcal{E}_{K^\prime,\beta^\prime,\ell^\prime} \Big\} \cap \Big\{ \exists w\in \mathcal{A}(\beta^\prime e^{-\beta^\prime},e^{-7\beta^\prime}) \ :\ P_m^{(k_\ast - 1)}(w) = 0\, , \ |F^{(k_\ast - 2)}(w)| \le Ce^{-\beta^\prime m} |w|^m \Big\}
		\end{align}
		for all $m\ge C(K^\prime,\beta^\prime,\ell^\prime)$ large enough. For $F^{(k_\ast-2)} \in \mathcal{E}_{K^\prime,\beta^\prime,\ell^\prime}$, we set $$\mathcal{W} = \big\{w\in\mathcal{A}(\beta^\prime e^{-\beta^\prime},e^{-7\beta^\prime}) \, :  \, P_m^{(k_\ast - 1)}(w) = 0 \big\}$$ and note that by Claim~\ref{claim:jensen_bound_on_the_number_of_zeros_E_K_beta_ell} we have $|\mathcal{W}| \le M$. Since $P_m$ and $T_m$ are independent, we have
		\begin{align*}
			\bP\bigg[&\Big\{ F^{(k_\ast-2)} \in \mathcal{E}_{K^\prime,\beta^\prime,\ell^\prime} \Big\} \cap \Big\{ \exists w\in \mathcal{A}(\beta^\prime e^{-\beta^\prime},e^{-7\beta^\prime}) \ :\ P_m^{(k_\ast - 1)}(w) = 0\, , \ |F^{(k_\ast - 2)}(w)| \le Ce^{-\beta^\prime m} |w|^m \Big\} \bigg] \\ & \le \bE_{P_m} \bigg[\mathbf{1}_{\{P_m^{(k_\ast-2)} \in \mathcal{E}_{K^\prime,\beta^\prime,\ell^\prime}\}} \bP_{T_m} \Big[\exists w\in \mathcal{W} \, : \ \big|T_m^{(k_\ast-2)}(w) + P_m^{(k_\ast-2)}(w)\big| \le Ce^{-\beta^\prime m} |w|^m  \Big] \, \bigg] \\ & \le \bE_{P_m} \bigg[\mathbf{1}_{\{P_m^{(k_\ast-2)} \in \mathcal{E}_{K^\prime,\beta^\prime,\ell^\prime}\}} \sum_{w\in \mathcal{W}}\bP_{T_m} \Big[\big|T_m^{(k_\ast-2)}(w) + P_m^{(k_\ast-2)}(w)\big| \le Ce^{-\beta^\prime m} |w|^m  \Big] \, \bigg] \\ & \le M \max_{w\in \mathcal{A}(\beta^\prime e^{-\beta^\prime},e^{-7\beta^\prime})} \max_{x\in \bC} \, \bP\Big[\big|T_m^{(k_\ast-2)}(w) - x\big| \le Ce^{-\beta^\prime m} |w|^m \Big] \, .
		\end{align*}
		By Claim~\ref{claim:Levy_concentration_for_derivatives_of_tail}, there exists a constant $C^\prime$ so that
		\[
		\max_{w\in \mathcal{A}(\beta^\prime e^{-\beta^\prime},e^{-7\beta^\prime})} \max_{x\in \bC} \, \bP\Big[\big|T_m^{(k_\ast-2)}(w) - x\big| \le Ce^{-\beta^\prime m} |w|^m \Big] \le \frac{C^\prime}{\sqrt{m}} \, ,
		\]
		which, by combining with~\eqref{eq:inclusion_of_events_existence_of_critical_point_where_F_is_small} implies that
		\[
		\bP\Big[F\in \mathcal{E}_{K,\beta,\ell} \ \text{and} \ \exists \alpha\in \mathcal{A}(\delta,1-\delta) \ \text{s.t.} \ F^{(k_\ast - 2)}(\alpha) = F^{(k_\ast-1)}(\alpha) = 0 \, , \ |F^{(k_\ast)}(\alpha)|\ge\gamma \Big] \lesssim m^{-1/2}
		\]
		for arbitrarily large $m$. This implies that
		\[
		\bP\Big[F\in \mathcal{E}_{K,\beta,\ell} \ \text{and} \ \exists \alpha\in \mathcal{A}(\delta,1-\delta) \ \text{s.t.} \ F^{(k_\ast - 2)}(\alpha) = F^{(k_\ast-1)}(\alpha) = 0 \, , \ |F^{(k_\ast)}(\alpha)|\ge\gamma \Big] = 0
		\]
		which is a contradiction to~\eqref{eq:by_contradition_positive_probability_for_double_zero}. This shows that $k_\ast \le 1$ which completes the proof. 
	\end{proof}
	\section{Reducing Theorem~\ref{thm:poisson_limit_for_close_roots} to a net argument}
	\label{sec:reducing_main_result_to_a_net_argument}
	Most of the roots of our random polynomial $f_n$ lie inside the annulus $\mathcal{A}(1-K/n,1+K/n)$ with $K\ge1$ large (see for instance \cite[Theorem~1]{Ibragimov-Zeitouni}) in the sense that one typically has $1 - \eps$ proportion of the roots in $\mathcal{A}(1 - K/n, 1 + K/n)$ for $K$ large enough as a function of $\eps$. In view of that observation, the proof of our main result Theorem~\ref{thm:poisson_limit_for_close_roots} will naturally be split into two main steps (recall also the outline of the proof from Section~\ref{subsection:outline_of_prood}). First, we will show that roots of $f_n$ that are present outside of the annulus do not contribute to the limiting Poisson process. This is seen by the following proposition.
	\begin{proposition}
		\label{prop:no_close_double_roots_in_the_bulk}
		For each $\eps>0$ and $L\ge1$, there is a $K\ge 1$ so that
		\[
		\limsup_{n\to \infty} \bP\Big[\exists \alpha \in \bD\big(0,1-K/n\big) \, , \ \exists \alpha^\prime\in \bD(\alpha,Ln^{-5/4})\setminus\{\alpha\} \ {\normalfont \text{such that }} \ f_n(\alpha) = f_n(\alpha^\prime) = 0 \Big] \le \eps\, .
		\]
	\end{proposition}
	Next, we will show that the close roots from inside the annulus $\mathcal{A}(1-K/n,1+K/n)$, scaled appropriately, converge in law to a Poisson point process as $n\to \infty$. Before that, we get a technical issue out of the way. Since $f_n$ has real coefficients, we have $f_n(\overline z) = \overline{f_n(z)}$, which in turn implies that the set of roots is symmetric with respect to complex conjugate. To avoid this redundancy, we will only consider the roots of $f_n$ which lie in the upper-half plane $\mathbb{H}$.  For $K\ge 1$ fixed and $U\subset \bR_{\ge 0}$ a finite union of intervals, we set 
	\begin{equation}
		\label{eq:def_of_Omega_k}
		\Omega_K = \mathbb{H} \cap \mathcal{A}(1-K/n,1+K/n)
	\end{equation}
	and 
	\begin{equation}
		\label{eq:def_of_X_n_U}
		X_n(U) = \# \Big\{ \alpha \in \Omega_K \, : \, \exists \alpha^\prime \ \, \text{with} \ \, |\alpha-\alpha^\prime| \in n^{-5/4} U \ \, \text{and } \ f_n(\alpha) = f_n(\alpha^\prime) = 0 \Big\} \, .
	\end{equation} 
	\begin{theorem}
		\label{thm:poisson_convergence_near_unit_circle}
		For all $K\ge 1$ fixed, we have $X_n(U) \xrightarrow{ \ d \ } \text{\normalfont Poisson}(\la_{K,U})$ as $n\to \infty$, with 
		\begin{equation}
			\label{eq:def_of_la_K_U}
			\la_{K,U} = \mathfrak{c}_\ast(K) \int_U t^3 \, {\rm d} t 
		\end{equation}
		and $\mathfrak{c}_\ast(K)>0$ is given by~\eqref{eq:intensity_constant_for_K_annulus} below.
	\end{theorem}
	Before we dive into the actual details, we first show how Proposition~\ref{prop:no_close_double_roots_in_the_bulk} combined with Theorem~\ref{thm:poisson_convergence_near_unit_circle} yields the proof of our main result of the paper.
	\begin{proof}[Proof of Theorem~\ref{thm:poisson_limit_for_close_roots}]
		Denoting by $\alpha_1,\ldots,\alpha_m$ the roots of $f_n$ in $\mathbb{H}$, $m\le n$, we consider the random measure
		\[
		\mu_n = \sum_{1\le j<j^\prime \le m} \delta_{n^{5/4} \, |\alpha_j - \alpha_{j^\prime}| } \, .
		\]
		Furthermore, for $K\ge 1$, we set 
		\[
		\mu_n^K = \sum_{1\le j<j^\prime \le m} \delta_{n^{5/4} \, |\alpha_j - \alpha_{j^\prime}| } \mathbf{1}_{\{\alpha_j,\alpha_{j^\prime} \in \Omega_K\}} \, .
		\]
		Note that the polynomial $g_n(z) = z^{n} f_n(z^{-1})$ has the same distribution of $f_n$ and has roots $\{\alpha_j^{-1}\}$. Further, for $\alpha_i, \alpha_j \in \bD$ we have that $|\alpha_i - \alpha_j| < |\alpha_j^{-1} - \alpha_i^{-1}|$.  As such, Proposition \ref{prop:no_close_double_roots_in_the_bulk} also eliminates pairs of close roots in $\bC \setminus \{\Omega_K \cap \bD\}$. Thus, for any $U\subset \bR_{\ge 0}$ a finite union of intervals, Proposition~\ref{prop:no_close_double_roots_in_the_bulk} implies that
		\[
		\limsup_{n\to \infty}\bP\Big[\big|\mu_n(U) - \mu_n^K(U)\big| \ge 1\Big] \le \eps 
		\] 
		for all $K\ge1$ large. Furthermore, Theorem~\ref{thm:poisson_convergence_near_unit_circle} implies that $\mu_n^K(U) = X_n(U)$ converges in distribution to a Poisson random variable with parameter $\la_{K,U}$. Since
		\[
		\lim_{K\to \infty} \la_{K,U} = \lim_{K\to \infty} \mathfrak{c}_\ast(K) \int_{U} t^3 \, {\rm d}t = \mathfrak{c}_\ast \int_{U} t^3 \, {\rm d}t
		\] 
		(see Claim~\ref{claim:integral_properties_of_intensity_mathfrak_F}) where $\mathfrak{c}_\ast>0$ is given by~\eqref{eq:limiting_constant_for_intensity} below, we observe that for all $\ell \in \bZ_{\ge 0}$,
		\begin{align*}
			\lim_{n\to \infty} \bP\big[\mu_n(U) = \ell \big] &= \lim_{K\to \infty} \lim_{n\to \infty}   \bP\big[\mu_n^K(U) = \ell \big] \\ &= \lim_{K\to \infty} e^{-\la_{K,U}} \frac{\la_{K,U}^\ell}{\ell!} = \exp\Big(-\mathfrak{c}_\ast \int_{U} t^3 \, {\rm d}t\Big) \frac{1}{\ell!} \Big(\mathfrak{c}_\ast \int_{U} t^3 \, {\rm d}t\Big)^\ell \, .
		\end{align*}  
		That is, $\mu_n(U)$ converge in distribution to a Poisson random variable with intensity $\mathfrak{c}_\ast \int_{U} t^3 \, {\rm d}t$. Since this is true for all $U\subset \bR_{\ge 0}$ a finite collection of intervals, a standard compactness argument (see, for instance, Kallenberg~\cite[Theorem~4.7]{Kallenberg}) implies that the sequence of point processes $\{\mu_n\}$ converge in the vague topology to a Poisson point process with intensity $\mathfrak{c}_\ast t^3 {\rm d}t$, and we are done. 
	\end{proof}

	\subsection{Making use of the sub-Gaussian assumption}
	
   In this section we show how the sub-Gaussian assumption on the coefficients is used in our argument. First, to ensure that the net argument indeed captures roots which are relatively close, we will also need a control on the maximum of $f_n$ along with its first few derivatives inside the disk. This is made possible by a variant of the classical Salem-Zygmund inequality, see for example~\cite[Chapter~6]{Kahane}.
	\begin{lemma}
		\label{lemma:control_on_maximum_of_polynomial_and_derivatives}
		For $r\in(\tfrac{1}{2},1)$ we write $d= \min\{n,(1-r)^{-1} \}$. There exists $d_0\ge 1$ so that for all $d\ge d_0$, the probability that
		\begin{equation*}
			\max_{z\in r\bD} \big|f_n^{(\ell)}(z)\big| \le d^{1/2+\ell} \log^2 d \, , \qquad \text{for} \ \ell\in\{0,1,2,3\} 
		\end{equation*}
		is at least $1-2\exp\left(- \log^2  d\right)$.
	\end{lemma}
	\begin{proof}
		For $\ell \ge 0$, the Salem-Zygmund inequality~\cite[Chapter~6, Theorem~1]{Kahane} applied with $\kappa = e^{\log^2 d}$ implies that
		\[
		\bP\bigg[ \max_{|z| = r} \big|f_n^{(\ell)}(z)\big| \ge C_\ell \Big(\sum_{k=0}^{n} k^{2\ell} r^{2k}\Big)^{1/2} \log d  \bigg] \le 2 \exp\big(-\log^2 d\big) \, ,
		\]
		for some constant $C_\ell>0$. Now, by taking $d_0$ to be large enough so that 
		\[
		C_\ell \Big(\sum_{k=0}^{n} k^{2\ell} r^{2k}\Big)^{1/2} \le d^{1/2+\ell} \log d
		\]
		for all $d\ge d_0$ we get what we wanted.
	\end{proof}
	
	The proofs of both Proposition~\ref{prop:no_close_double_roots_in_the_bulk} and of Theorem~\ref{thm:poisson_convergence_near_unit_circle} are based on net arguments in the respective annulus. To make those net arguments work, we will need to prove effective small-ball probability bounds for samples of $(f_n(z),f_n^\prime(z),f_n^{\prime\prime}(z))$ along net points $z$.  
	The standard tool for bounding small-ball probabilities via the corresponding characteristic function is through the Esseen inequality~\cite{Esseen}: for any random vector $X$ with characteristic function $\varphi_{X}(\xi) = \bE[e^{i\langle X,\xi \rangle}]$ we have
	\begin{equation}
		\label{eq:classical_esseen_inequality}
		\sup_{x\in X} \bP\big[|X-x| \le \delta \big] \lesssim  m\big(B(0,\delta)\big) \int_{B(0,\delta^{-1})} |\varphi_{X}(\xi)|\, {\rm d} m(\xi) \, .
	\end{equation}
	In what follows, a typical application of~\eqref{eq:classical_esseen_inequality} would be with $X= (f_n(z),d^{-1} f_n^\prime(z),d^{-2} f_n^{\prime\prime}(z))$, where $d=d_n(z)=\min\{n,(1-|z|)^{-1}\}$ is the effective degree of the polynomial $f_n$ at $z\in \bD$. In fact, for our specific application we will also need an `exponentially-tilted' version of the Esseen inequality, which takes the form of Proposition~\ref{propsition:exponentially_tilded_esseen_ineq} below and is inspired by~\cite[Lemma~5.2]{Campos-Jenssen-Michelen-Sahasrabudhe}. 
	
	Recall that the random coefficients of $f_n$ are given by $\{\xi_k\}$, and let $\{\tilde\xi_k\}$ and independent copy of $\{\xi_k\}$. Let $\eta_k$ be the random variable which is equal to $0$ with probability $1/2$ and to $\xi_k - \tilde \xi_k$ with probability $1/2$. With that, we can define a new random polynomial
	\begin{equation}
		\label{eq:symmetrized_version_of_f_n}
		g_n(z) = \sum_{k=0}^{n} \eta_k z^k 
	\end{equation}
	which is also a Kac polynomial. Let $\varphi_{\eta_0}$ denote the characteristic function of $\eta_0$.

	\begin{proposition}
		\label{propsition:exponentially_tilded_esseen_ineq}
		There exists a constants $C,c>0$ such that for all $z\in \mathcal{A}(\tfrac{1}{2},2)$ we have
		\begin{multline*}
			\max_{\alpha,\beta\in \bC} \bP \Big[|f_n(z)-\alpha| \le \delta d^{1/2} \, , \, |f_n^\prime(z)-\beta| \le \delta  d^{3/2} \, , \, |f_n^{\prime \prime}(z)| \geq t d^{5/2} \Big] \\ \le C e^{-ct^2} \delta^4 \, \int_{\bR^4} \Bigg(\prod_{k=0}^{n} \varphi_{\eta_0}\Big( d^{-1/2} \Big\langle \xi , \begin{pmatrix}
				z^k \\ \tfrac{k}{d} z^k
			\end{pmatrix} \Big\rangle \Big) \Bigg)^{1/4} e^{-c\delta^2|\xi|^2} \, {\rm d}m(\xi) \, ,
		\end{multline*}
		for all $\delta>0$ and $t\ge 0$, where $d=d_n(z)$ is given by~\eqref{eq:def_of_effective_degree_intro}.
	\end{proposition}
	\begin{proof}
		For the proof, we will denote by $$X = d^{-1/2}\big(f_n(z),d^{-1} f_n^\prime(z)\big)\in \bC^2\simeq \bR^4$$ and $Y= f_n^{\prime\prime}(z)/d^{5/2} \in \bC$. We will also denote by
		\begin{equation}
			\label{eq:def_of_cs_version_of_char_function}
			\Phi(t) = \bE\Big[\cos^2\big((\xi_0 - \xi_0^\prime)  t\big)\Big]  \, , \qquad t\in \bR\, .
		\end{equation}
		Then the proof will follow once we prove the bound
		\begin{equation}
			\label{eq:proof_of_esseen_ineq_what_we_want}
			\bP\Big[ |X - x| \le \delta  \, , |Y|^2 \ge t^2 \Big] \lesssim e^{-ct^2} \delta^4 \int_{\bR^4}\Bigg(\prod_{k=0}^{n} \Phi\Big( d^{-1/2} \Big\langle \xi , \begin{pmatrix}
				z^k \\ \tfrac{k}{d} z^k
			\end{pmatrix} \Big\rangle \Big) \Bigg)^{1/4} e^{-\delta^2 |\xi|^2/2} \, {\rm d}m(\xi)
		\end{equation}
		uniformly for $x\in \bR^4$. Indeed, by the elementary inequality $\cos^2(x)\le \frac{1}{2} + \frac{1}{2} \cos(2x)$ we have that
		\begin{equation*}
			\Phi(t) \le \varphi_{\eta_0} (2t) \, ,
		\end{equation*}
		and by plugging this into~\eqref{eq:proof_of_esseen_ineq_what_we_want} and changing variables $2\xi \mapsto \xi$ we get the inequality we are after.
		
		To prove~\eqref{eq:proof_of_esseen_ineq_what_we_want}, we first recall that the random coefficients $\{\xi_k\}$ are assumed to be sub-Gaussian, and hence $Y$ is a sum of independent sub-Gaussian random variables. By~\cite[Propsotion~2.6.1]{Vershynin} there exists a constant $c>0$ so that
		\begin{equation}
			\label{eq:second_derivative_is_subGaussian}
			\bE\Big[\exp\big(c|Y|^2\big)\Big] \le 4
		\end{equation}
		for all $z\in \mathcal{A}(\tfrac{1}{2},2)$. Hence, for all $\la>0$ we can apply Markov's inequality and get that
		\begin{align}
			\label{eq:esseen_ineq_proof_apply_markov}
			\nonumber 
			\bP\Big[ |X - x| \le \delta \, , |Y|^2 \ge t^2 \Big] & = \bP\Big[ e^{-\la^2|X - x|^2/2} \ge e^{-\la^2\delta^2/2} \, ,\ e^{\tfrac{c}{2}|Y|^2} \ge e^{ct^2/2} \Big] \\ & \le e^{\la^2\delta^2/2 - ct^2/2} \bE\Big[e^{-\la^2|X - x|^2/2}e^{\tfrac{c}{2}|Y|^2}\Big] \, .
		\end{align}
		Fourier inversion implies that for all $z\in \bR^4$ we have
		\[
		e^{-|z|^2/2} = \frac{1}{4\pi^2} \int_{\bR^4} e^{-|y|^2/2} e\big(\langle y, z\rangle\big) \, {\rm d}m(y) \, ,
		\]
		and by plugging in $z = \la (X-x)$ and applying Fubini we get that
		\begin{align*}
			\bE\Big[e^{-\la^2|X - x|^2/2}e^{\tfrac{c}{2}|Y|^2}\Big] &= C \, \bE \bigg[\exp(\tfrac{c}{2}|Y|^2) \int_{\bR^4} e^{-|y|^2/2} e\big(\langle \la y, X-x \rangle\big) {\rm d} m(y)\bigg] \\ & \lesssim \int_{\bR^4} \Big|\bE\Big[\exp(\tfrac{c}{2}|Y|^2) \,  e\big(\la \langle  y, X \rangle\big)\Big] \Big| \, e^{-|y|^2/2} \, {\rm d} m(y) \, .
		\end{align*}
		Plugging this into~\eqref{eq:esseen_ineq_proof_apply_markov}, we can apply the change of variables $\xi = \la y$ along with the choice $\la = \delta^{-1}$ and get that
		\begin{equation}
			\label{eq:proof_of_esseen_ineq_after_chagne_of_variables}
			\bP\Big[ |X - x| \le \delta \, , |Y|^2 \ge t^2 \Big] \lesssim \delta^4 \int_{\bR^4} \Big|\bE\Big[\exp(\tfrac{c}{2}|Y|^2) \,  e\big(\langle  \xi , X \rangle\big)\Big] \Big| \, e^{-\delta^2|\xi|^2/2} \, {\rm d} m(\xi) \, .
		\end{equation}
		It remains to bound the expectation inside the integral. Recall that $\{\tilde{\xi}_k\}$ are independent copies of $\{\xi_k\}$, and let $\tilde{X},\tilde{Y}$ be the corresponding independent copies of the random vectors $X,Y$. We have
		\begin{align*}
			\Big|\bE\Big[\exp(\tfrac{c}{2}|Y|^2) \,  e\big(\langle  \xi , X \rangle\big)\Big] \Big|^2 &= \bE\Big[\exp(\tfrac{c}{2}|Y|^2) \,  e\big(\langle  \xi , X \rangle\big)\Big] \cdot \bE\Big[\exp(\tfrac{c}{2}|\tilde Y|^2) \,  e\big(-\langle  \xi , \tilde X \rangle\big)\Big] \\ &= \bE\Big[\exp(\tfrac{c}{2}|Y|^2 + \tfrac{c}{2}|\tilde Y|^2) \,  e\big(\langle  \xi , X - \tilde X \rangle\big)\Big] \\  (\text{definition of $X$}) \qquad &= \bE\Bigg[\exp(\tfrac{c}{2}|Y|^2 + \tfrac{c}{2}|\tilde Y|^2) \,  \prod_{k=0}^{n} e\bigg( d^{-1/2} (\xi_k - \tilde \xi_k) \Big\langle  \xi , \begin{pmatrix}
				z^k \\ \tfrac{k}{d} z^k
			\end{pmatrix} \Big\rangle\bigg)\Bigg] 
			\\ (\text{symmetry of $\xi_k - \tilde \xi_k$}) \qquad &= \bE\Bigg[\exp(\tfrac{c}{2}|Y|^2 + \tfrac{c}{2}|\tilde Y|^2) \,  \prod_{k=0}^{n} \cos \bigg( d^{-1/2} (\xi_k - \tilde \xi_k) \Big\langle  \xi , \begin{pmatrix}
				z^k \\ \tfrac{k}{d} z^k
			\end{pmatrix} \Big\rangle\bigg)\Bigg]  \, .
		\end{align*}		
		Now, combining the Cauchy-Schwarz inequality with~\eqref{eq:second_derivative_is_subGaussian}, we see that
		\begin{align*}
			\Big|\bE\Big[\exp(\tfrac{c}{2}|Y|^2) \,  e\big(\langle  \xi , X \rangle\big)\Big] \Big|^2 & \lesssim \Bigg(\bE\Big[\prod_{k=0}^{n} \cos^2 \bigg( d^{-1/2} (\xi_k - \tilde \xi_k) \Big\langle  \xi , \begin{pmatrix}
				z^k \\ \tfrac{k}{d} z^k
			\end{pmatrix} \Big\rangle\bigg)\Big] \Bigg)^{1/2} \\ &= \Bigg( \prod_{k=0}^{n} \Phi\bigg( d^{-1/2} \Big\langle  \xi , \begin{pmatrix}
				z^k \\ \tfrac{k}{d} z^k
			\end{pmatrix} \Big\rangle\bigg) \Bigg)^{1/2}
		\end{align*}
		where the equality is due to the independence of the sequence $\{\xi_k-\tilde \xi_k \}$. Plugging this into~\eqref{eq:proof_of_esseen_ineq_after_chagne_of_variables} yields~\eqref{eq:proof_of_esseen_ineq_what_we_want}, and we are done.
	\end{proof}

	\subsection{``Worst case'' small-ball probability bounds}
	\label{sec:small_ball_probability_general_statements}
	We will require three small-ball probability bounds for the the tuple $(f_n,f_n')$ evaluated at arbitrary points in $\bD$.  We think of these as ``worst-case'' bounds since they require no arithmetic structure on $z$.  The three results stated in this subsection vary depending on how close $z$ is to the real axis.
	
	Recall that $d = d_n(z) = \min\{n,(1-|z|)^{-1}\}$ is the effective degree of the polynomial $f_n$ at a point $z\in \bD$.  We first state a rather weak small-ball probability bound, which is valid for all points $z\in \bD$ which are separated from the real axis.
	\begin{claim}
		\label{claim:small_ball_away_from_real_axis}
		There exists $C,c>0$ such that for all $z=re^{i \theta}$ with $\theta \in [d^{-1},\pi-d^{-1}]$ we have
		$$\max_{\alpha,\beta\in \bC} \bP \Big[|f_n(z)-\alpha| \le a d^{1/2} \, , \, |f_n^\prime(z)-\beta| \le b d^{3/2}  \Big] \leq C  (ab)^2$$
		for all $a,b\ge (\log d)/d^{1/2}$.
	\end{claim}
	The proof of Claim~\ref{claim:small_ball_away_from_real_axis} is not difficult, and is based on a Gaussian approximation on the Fourier-side. We will also need a quantitative variant of the Claim~\ref{claim:small_ball_away_from_real_axis} as the evaluation point $z$ approaches the real axis, which is provided by the next claim.
	\begin{claim}
		\label{claim:small_ball_points_near_the_real_axis}
		There exists $C,c>0$ such that for all $z=re^{i\theta}$ with $\theta \in [d^{-1 - \tau/4},d^{-1}]$ we have
		$$\max_{\alpha,\beta\in \bC} \bP \Big[|f_n(z)-\alpha| \le a d^{1/2} \, , \, |f_n^\prime(z)-\beta| \le b d^{3/2}\,  \Big] \leq C d^{2\tau} (ab)^2$$
		for all $a,b\ge d^{-1/2 + 3\tau}$.
	\end{claim}
	The proofs of Claim~\ref{claim:small_ball_away_from_real_axis} and Claim~\ref{claim:small_ball_points_near_the_real_axis} 
    are given in Section~\ref{sec:small_ball_for_general_points}. Finally, we will also need a small-ball bound for the polynomial evaluated on points very close to the real axis. For this, the classical Littlewood-Offord bound (Lemma \ref{lemma:Levy-Kolmogorov-Rogozin}) is sufficient:
	\begin{claim}
		\label{claim:small_ball_real_points}
		There exists $C>0$ so that for all $z=re^{i\theta}\in \bD$ with $\theta\in [0,d^{-1-\tau/4}]$ we have
		$$\max_{\alpha,\beta\in \bC} \bP \Big[|f_n(z)-\alpha| \le a d^{1/2}\Big] \leq C  \big(a + d^{-1/2} \big)$$
		for all $a>0$.	
	\end{claim}
	
	\subsection{A small-ball probability bound for smooth points}
	The small-ball probability bounds stated above are relatively weak, in the sense that the parameters $a,b$ must be taken to be rather large (specifically, larger than $d^{-1/2+o(1)}$) for the bounds to apply. For the proof of Proposition~\ref{prop:no_close_double_roots_in_the_bulk} this will not suffice as we shall need to deal with probabilities of much smaller order. To improve our bounds, we lean on a method pioneered by Konyagin and Schlag~\cite{Konyagin, Konyagin-Schlag} in their study of the minimum modulus of the polynomial on the unit circle. 
	\begin{definition}
		\label{def:smooth_angle}
		For $A\ge1$, we say that $z=re^{i\theta} \in \bD \cap \mathbb{H} $ is $A$-smooth if 
		\[
		\|p\theta/\pi \|_{\bR/\bZ} \ge A/d_n(z) \qquad \forall p\in [1,A+1] \cap \bZ \, ,
		\]
		where as before $d_n(z) = \min\{n,(1-r)^{-1}\}$. We say that $z\in \mathbb{H} \setminus \bD$ is $A$-smooth if $z^{-1}$ is $A$-smooth.
	\end{definition}
	To carry out the net argument for the proof of Proposition~\ref{prop:no_close_double_roots_in_the_bulk}, we will separate the net to those points where the random vector $\big(f_n(z),f_n^\prime(z)\big)$ is relatively well-spread (which happens exactly when $z$ is smooth) and the rest of the points, where we shall apply Claim~\ref{claim:small_ball_away_from_real_axis}, Claim~\ref{claim:small_ball_points_near_the_real_axis} and Claim~\ref{claim:small_ball_real_points} instead. Concretely, we will use the following small-ball probability bound, which we prove in Section~\ref{sec:small_ball_bounds_for_smooth_points}.
	\begin{lemma}
		\label{lemma:small_ball_probability_smooth_points_bound_single_point}
		There exists universal constants $r_0\in(\tfrac{1}{2},1)$ and $c,C>0$ so that for all $r\in[r_0,1]$ and for all $t\ge 0$ the following holds. For each $z=re^{i\theta}$ which is $d^{7\tau}$-smooth we have
		\begin{equation*}
			\max_{\alpha,\beta\in \bC} \bP \Big[|f_n(z)-\alpha| \le a d^{1/2} \, , \, |f_n^\prime(z)-\beta| \le b d^{3/2}\, , \, |f_n^{\prime \prime}(z)| \geq t d^{5/2}\Big] \leq C e^{-c t^2} (ab)^2 \, ,
		\end{equation*}
		where $a,b\ge d^{-3/2+11\tau}$ and $d=d_n(z) = \min\{n,(1-r)^{-1}\}$.
	\end{lemma}
	\subsection{Gaussian comparison near the unit circle}
	While for the proof of Proposition~\ref{prop:no_close_double_roots_in_the_bulk} we will only need upper bounds on the small-ball probabilities, to prove the limit law as stated in Theorem~\ref{thm:poisson_convergence_near_unit_circle} we will need to compute the first-order approximation of small-ball probabilities along the net. Recall the definition of $\Omega_K$ given by
	\[
	\Omega_K = \mathbb{H} \cap \mathcal{A}(1-K/n,1+K/n) \, .
	\]  
	\begin{definition}
		\label{def:spread_tuples}
		For $m\ge 2$ and $\gamma>0$ we say that $(z_1,\ldots,z_m)\in \Omega_K^m$ is $\gamma$-spread if
		\[
		|z_j -z_{j^\prime}| \ge \frac{\gamma}{n} \qquad \forall 1\le j<j^\prime \le m \, .
		\]
		Furthermore, we say that a tuple $(z_1,\ldots,z_m)\in \Omega_K^m$ is $A$-smooth if $z_j$ is $A$-smooth (in the sense of Definition~\ref{def:smooth_angle}) for each $1\le j\le m$.
	\end{definition}
	\begin{remark}
		\label{remark:effective_degree_in_omega_K}
		In a moment, we will state our small-ball bounds/comparisons for tuples of smooth, separated points in $\Omega_K$. Even though Definition~\ref{def:smooth_angle} depends on the effective degree $d_n(z)$, in the annulus $\Omega_K$ all effective degrees have $$d_n(z) \gtrsim_K n.$$ Therefore, any $A$-smooth tuple in $\Omega_K$ would be $A^\prime$-smooth if $d_n(z)$ is replaced by $n$ in Definition~\ref{def:smooth_angle}, for some $A^\prime\gtrsim_K A$. For our application, the degree of ``smoothness" $A$ will always tend to infinity as $n\to \infty$, so we will safely ignore this point and not keep track of the effective degree $d_n(\cdot)$ when considering points in $\Omega_K$. 
	\end{remark}
	For a tuple $\mathbf{z} = (z_1,\ldots,z_m)\in \Omega_K^m$ we denote by
	\begin{equation}
		\label{eq:def_of_S_n_z}
		S_n(\mathbf{z}) = \frac{1}{\sqrt{n}} \Big(f_n(z_1),\ldots,f_n(z_m), \frac{z_1^{-1}}{n}f_n^\prime(z_1),\ldots, \frac{z_m^{-1}}{n}f_n^\prime(z_m),\frac{z_1^{-2}}{n^2} f_n^{\prime \prime}(z_1),\ldots,\frac{z_m^{-2}}{n^2} f_n^{\prime \prime}(z_m)\Big)\, .
	\end{equation}
	To prove Theorem~\ref{thm:poisson_convergence_near_unit_circle}, we will need to compare small ball probabilities of $S_n(\mathbf{z})\in \bC^{3m} \simeq \bR^{6m}$ with that of the corresponding Gaussian. Let $\Gamma_n(\mathbf{z})\in \bC^{3m}$ be a Gaussian vector with the same covariance matrix as $S_n(\mathbf{z})$ (which is spelled out in Section~\ref{sec:gaussian_comparison_for_tuples} below, see~\eqref{eq:S_n_z_as_a random_walk}). Note that in the case where the random coefficients $(\xi_k)$ are standard Gaussian, then $\Gamma_n(\mathbf{z})$ has the same distribution of $S_n(\mathbf{z})$.
	\begin{theorem}
		\label{thm:small_ball_comparison_to_Gaussian_for_tuples}
		Let $\mathbf{z} = (z_1,\ldots,z_m)\in \Omega_K^m$ be $n^\kappa$-smooth and $1$-spread for some $\kappa\in(0,1)$, and let $Q\subset \bC^{3m}\simeq \bR^{6m}$ be a rectangle (that is, a Cartesian product of intervals) lying in some fixed compact set, and with side-lengths at least $n^{-M}$ for some $M>0$. Then
		\[
		\sup_{\zeta \in \bC^{3m}} \bigg|\bP\big[S_n(\mathbf{z}) \in Q+\zeta\, \big] - \bP\big[\Gamma_n(\mathbf{z}) \in Q+\zeta \, \big]\bigg| \lesssim n^{-1/2} \, m(Q) \, ,
		\]
		where the constant depends only on $m,\kappa,K,M$ and the distribution of $\xi$. 
	\end{theorem}
	We will also need a near optimal small-ball probability bound for tuples which are almost $1$-spread.
	\begin{proposition}
		\label{prop:small_ball_estimate_for_tuples_of_smooth}
		Let $\mathbf{z} = (z_1,\ldots,z_m)  \in \Omega_K^m$ be $n^\kappa$-smooth and $\gamma$-spread for some $\kappa\in(0,1)$ and $\gamma \ge n^{-{1/30m}}$. Then for any $K>0$ and for any $\delta>n^{-M}$ we have
		\[
		\sup_{\zeta \in \bC^{3m}} \bP\Big[ S_n(\mathbf{z}) \in B(\zeta,\delta) \Big] \lesssim \gamma^{-O(m^2)}\delta^{6m} \, ,
		\]
		where $B(\zeta,\delta)\subset \bC^{3m}$ is a ball of radius $\delta$ around $\zeta$, and the explicit constant depends only on $m,\kappa,K,M$ and the distribution of $\xi$.
	\end{proposition}
	Both Theorem~\ref{thm:small_ball_comparison_to_Gaussian_for_tuples} and Proposition~\ref{prop:small_ball_estimate_for_tuples_of_smooth} are proved in Section~\ref{sec:gaussian_comparison_for_tuples}.
	
	\section{Reducing to an annulus: Proof of Proposition~\ref{prop:no_close_double_roots_in_the_bulk}}
	\label{sec:reducing_to_the_annulus}
	The main goal of this section is to prove Proposition~\ref{prop:no_close_double_roots_in_the_bulk}, namely that with high probability there is no pair of roots at distance $\omega(1/n)$ from the unit circle which contributes to the limiting point process stated in Theorem~\ref{thm:poisson_limit_for_close_roots}. In what follows we will need different net arguments at different scales, so we start by giving a rather general definition of our net in an annulus.
	\begin{definition}
		\label{definition:net_in_annulus}
		For $r_1<r_2$ and $M\in \bN$ we define
		\begin{equation*}
			\mathcal{N} = \Big\{ (r_1+a/M)e\big(\pi b/M\big) \, : \, a\in \{0,1,\ldots, \lceil M(r_2-r_1) \rceil\} \, , \ b\in \{1,\ldots,2M\} \Big\} \, , 
		\end{equation*}
		and note that $\mathcal{N}$ is a $(2r_2/M)$-net of $\mathcal{A}(r_1,r_2)$. For each point $z = (r_1+a/M)e\big(\pi b/M\big) \in \mathcal{N}$, define the (polar) rectangle
		\begin{equation*}
			\mathcal{R}_z = \Big\{(r_1+(a+s)/M)e\big(\pi (b + t)/M\big) \, : \, t,s\in [-1/2,1/2] \Big\} \, .
		\end{equation*}
		Then the collection $\{\mathcal{R}_z\}_{z\in \mathcal{N}}$ consists of rectangles with disjoint interiors which cover $\mathcal{A}(r_1,r_2)$.
	\end{definition}
	Let $K\ge 1$ be the constant from the statement of Proposition~\ref{prop:no_close_double_roots_in_the_bulk}, which we will take sufficiently large (as a function of $\eps,L$) as the proof goes. As we will only consider the disk $(1-K/n) \bD$ in this section, for $r_0<1$ we define the sequence of decreasing radii
	\begin{equation}
		\label{eq:def_of_rho_k}
		\rho_k = 1  -  \frac{K + (2^k -1)}{n}\, ,
	\end{equation}
	where $k=0,1,\ldots,N$ with $N = \lceil \log_2 \big(n(1-r_0) - K +1\big) \rceil$. With this notation, we have $\rho_0  = 1- K/n$ and $r_0  - 2(1-r_0) \le \rho_N \le r_0$. We will also denote by $\mathcal{A}_k = \mathcal{A}(\rho_{k+1},\rho_{k})$ and $d_k = (1-\rho_k)^{-1}$. For reasons that will be evident later, we will also denote by $\rho_{-1} = 1$, and with that $\mathcal{A}_{-1} = \mathcal{A}(1-K/n,1)$. We always set $d_{-1} = n$ to be the maximal possible degree. 
	\subsection{The derivative cannot be too small on roots}
	We start our argument by showing that it is unlikely to have a root on which the derivative is very small. The key for that is the following deterministic fact.
	\begin{claim}
		\label{claim:determenistic_implication_of_having_a_root_with_small_derivative}
		For any $\beta>0$ and for any $d \ge d_0(\beta)$ large enough the following holds. For all $z\in \bC$, $\eps \ge d^{-\beta/2}$ and $\delta\in(0,d^{-5/4-\beta})$, suppose that $f$ is an analytic function for which
		\[
		\max_{\zeta \in \bD(z,\delta)} |f^{\prime\prime}(\zeta)| \le d^{5/2} \log^2 d\, .
		\]
		If there exists $w\in \bD(z,\delta)$ so that $f(w) = 0$ and $|f^\prime(w)|\le \eps d^{5/4}$, then
		\[
		|f(z)| \le \delta \eps d^{5/4} (1+d^{-\beta/3}) \qquad \text{and} \qquad |f^\prime(z)| \le \eps d^{5/4} (1+d^{-\beta/2}) \, .
		\] 
	\end{claim}
	\begin{proof}
		Applying Taylor's theorem around the point $z$ gives
		\begin{equation*}
			|f^\prime(z)| \le |f^\prime(w)| + \delta \max_{\zeta \in \bD(z,\delta)} |f^{\prime\prime}(\zeta)| \le \eps d^{5/4} + d^{5/4-\beta} \log^2 d \le \eps d^{5/4} (1+d^{-\beta/2}) \, , 
		\end{equation*}
		where in the last inequality we used the bound $\eps \ge d^{-\beta/2}$. To get the corresponding bound for $|f^\prime(z)|$, we use Taylor expansion again to get
		\begin{align*}
			|f(z)| &= |f(z) - f(w)| \le \delta|f^\prime(z)| + \frac{\delta^2}{2}  \max_{\zeta \in \bD(z,\delta)} |f^{\prime\prime}(\zeta)| \\ & \le \delta \Big(\eps d^{5/4} \big(1+d^{-\beta/2}\big) + \frac{1}{2} d^{5/4 - \beta} \log^2 d\Big) \le \delta \eps d^{5/4}\big(1+d^{-\beta/3}\big) \, ,
		\end{align*}
		which is what we wanted.
	\end{proof} 
	Recall the definition of $\rho_k$ for $k=0,\ldots,N$ as given by~\eqref{eq:def_of_rho_k}. For each $k\in \{0,\ldots,N\}$ fixed, we will consider the net $\mathcal{N} = \mathcal{N}_k$ of $\mathcal{A}(\rho_{k+1},\rho_k)$ which is given by Definition~\ref{definition:net_in_annulus} with the parameters $\delta = d_k^{-5/4-\beta}$ and $M = \lfloor10/\delta\rfloor$. Recall Definition~\ref{def:smooth_angle} of smooth points $z\in \bC$, and let $\tau>0$ be a sufficiently small absolute constant. We will also need further decomposition of the net $\mathcal{N} = \mathcal{N}_k$. We denote by
	\begin{equation}
		\label{eq:decomposition_into_smooth_and_non_smooth_in_the_bulk}
		\mathcal{N}_{\tt s} = \big\{z\in \mathcal{N}\, : \, z \ \text{is}\ d_k^{7\tau}\text{-smooth}\big\}
	\end{equation}
	and by $\mathcal{N}_{\tt ns} = \mathcal{N} \setminus \mathcal{N}_{\tt s}$. We will further decompose the non-smooth part into
	\[
	\mathcal{N}_{\tt ns} = \mathcal{N}_{\tt ns}^{(1)} \cup \mathcal{N}_{\tt ns}^{(2)} \cup \mathcal{N}_{\tt ns}^{(3)} \, ,
	\]
	where
	\begin{align}
		\label{eq:further_decomposition_of_non_smooth_in_the_bulk}
		\nonumber
		\mathcal{N}_{\tt ns}^{(1)} &= \big\{z\in \mathcal{N}_{\tt ns} \, : \, \arg(z) \in [d_k^{-1},\pi - d_k^{-1}]\big\} \, , \\ \mathcal{N}_{\tt ns}^{(2)} &= \big\{z\in \mathcal{N}_{\tt ns} \, : \, \arg(z) \in [d_k^{-1-\tau/4},d_k^{-1}] \cup [\pi - d_k^{-1},\pi - d_k^{-1-\tau/4}]\big\} \, , \\ \nonumber \mathcal{N}_{\tt ns}^{(3)} &= \big\{z\in \mathcal{N}_{\tt ns} \, : \, \arg(z) \in [0,d_k^{-1-\tau/4}] \cup [\pi - d_k^{-1-\tau/4},\pi]\big\} \, . 
	\end{align}
	We note that we have the following bounds:
	\begin{align*}
		|\mathcal{N}_{\tt s}| &\lesssim M^2 (\rho_{k} - \rho_{k+1}) \lesssim \delta^{-2} d_k \, ,  \\ |\mathcal{N}_{\tt ns}^{(1)}| &\lesssim \delta^{-2} d_k^{7\tau} \, , \\ |\mathcal{N}_{\tt ns}^{(2)}| &\lesssim \delta^{-2}\, , \\ |\mathcal{N}_{\tt ns}^{(3)}| &\lesssim \delta^{-2} d_k^{-\tau/4} \, .
	\end{align*}
	\begin{lemma}
		\label{lemma:with_high_probability_no_roots_with_too_small_derivative}
		There exist absolute constants $C,\beta>0$, $r_0<1$ so that the following holds. For each $k\in \{-1,0,\ldots,N\}$ and $\eps\in(d_k^{-\beta},1)$ we have
		\[
		\bP\Big[\exists \alpha \in \mathcal{A}_k \, : \ f_n(\alpha) = 0\, , \ |f_n^\prime(\alpha)| \le \eps d_k^{5/4} \Big] \le C\eps \, .
		\]
	\end{lemma}
	\begin{proof}
		As $k\in \{-1,0,1,\ldots,N\}$ will remain fixed throughout the proof, we will sometimes omit the dependence of the parameters on it to lighten the notation. Define the event $\mathcal{E}$ via
		\[
		\mathcal{E} = \Big\{ \max_{\zeta \in \rho_k \bD} |f_n^{\prime \prime}(\zeta)| \le d^{5/2} \log^2 d \Big\}
		\]
		and note that by Lemma~\ref{lemma:control_on_maximum_of_polynomial_and_derivatives} we have $\bP\big[\mathcal{E}^c\big] \le \exp (-c\log^2 d)$. Using the fact that the set $\mathcal{N}$ is a $\delta$-net of $\mathcal{A}_k$, we can use Claim~\ref{claim:determenistic_implication_of_having_a_root_with_small_derivative} and the union bound to get
		\begin{align}
			\label{eq:with_high_probability_no_roots_with_too_small_derivative_after_union_bound}
			\nonumber
			\bP\Big[\exists \alpha \in \mathcal{A}_k \, : \ &f_n(\alpha) = 0\, , \ |f_n^\prime(\alpha)| \le \eps d^{5/4} \Big] \\  \nonumber & \le \bP\Big[\exists \alpha \in \mathcal{A}_k \, : \ f_n(\alpha) = 0\, , \ |f_n^\prime(\alpha)| \le \eps d_k^{5/4} \, , \mathcal{E} \Big] + e^{-c\log^2 d} \\ & \le \sum_{z\in \mathcal{N}} \bP\Big[|f_n(z)| \le 2\delta \eps d^{5/4} \, , \ |f_n^\prime(z)| \le 2 \eps d^{5/4} \Big] + e^{-c\log^2 d} \, .
		\end{align}
		Here we used the fact that $\min_k \rho_k = \rho_N$ can be made arbitrarily close to $1$ by increasing $r_0$. We will bound the sum in~\eqref{eq:with_high_probability_no_roots_with_too_small_derivative_after_union_bound} by splitting the net into $\mathcal{N}_{\tt s}$ and $\mathcal{N}_{\tt ns}^{(1)},\mathcal{N}_{\tt ns}^{(2)},\mathcal{N}_{\tt ns}^{(3)}$, as described in~\eqref{eq:decomposition_into_smooth_and_non_smooth_in_the_bulk} and~\eqref{eq:further_decomposition_of_non_smooth_in_the_bulk}.
		\noindent
		\underline{The sum over $\mathcal{N}_{\tt s}$:} For points $z\in \mathcal{A}_k$ that are $d^{7\tau}$-smooth, Lemma~\ref{lemma:small_ball_probability_smooth_points_bound_single_point} implies that
		\begin{multline*}
			\bP\Big[|f_n(z)| \le 2\delta \eps d^{5/4} \, , \ |f_n^\prime(z)| \le 2 \eps d^{5/4} \Big] \\ \le \bP\Big[|f_n(z)| \le 2\eps d^{1/2} d^{-\beta-1/2} \, , \ |f_n^\prime(z)| \le 2 \eps d^{3/2-1/4} \Big] \lesssim \eps^4 d^{-3/2-2\beta} \, .
		\end{multline*}
		Since $|\mathcal{N}_{\tt s}| \lesssim \delta^{-2} d = d^{3/2 + 2\beta}$, we get that
		\begin{equation}
			\label{eq:with_high_probability_no_roots_with_too_small_derivative_bound_on_smooth_points}
			\sum_{z\in \mathcal{N}_{\tt s}} \bP\Big[|f_n(z)| \le 2\delta \eps d^{5/4} \, , \ |f_n^\prime(z)| \le 2 \eps d^{5/4} \Big] \lesssim \eps^4 \, .
		\end{equation} 
		\noindent
		\underline{The sum over $\mathcal{N}_{\tt ns}^{(1)}$:} Here we argue similarly, but apply Claim~\ref{claim:small_ball_away_from_real_axis} instead of Lemma~\ref{lemma:small_ball_probability_smooth_points_bound_single_point} (note that $z\in \mathcal{N}_{\tt ns}^{(1)}$ are separated from the real axis) and we get a slightly worst bound 
		\begin{multline*}
			\bP\Big[|f_n(z)| \le 2\delta \eps d^{5/4} \, , \ |f_n^\prime(z)| \le 2 \eps d^{5/4} \Big] \\ \le \bP\Big[|f_n(z)| \le \log d \, , \ |f_n^\prime(z)| \le 2 \eps d^{3/2-1/4} \Big] \lesssim \eps^2 \,  d^{-3/2} \log^2 d  \, .
		\end{multline*}
		Since $|\mathcal{N}_{\tt ns}^{(1)}| \lesssim d^{1/2 + \beta + 7\tau}$, we get that
		\begin{equation}
			\label{eq:with_high_probability_no_roots_with_too_small_derivative_bound_on_nonsmooth_points_1}
			\sum_{z\in \mathcal{N}_{\tt s}^{(1)}} \bP\Big[|f_n(z)| \le 2\delta \eps d^{5/4} \, , \ |f_n^\prime(z)| \le 2 \eps d^{5/4} \Big] \lesssim d^{-1+2\beta + 7\tau} \eps^2 \leq \eps^2\, .
		\end{equation}
		\noindent
		\underline{The sum over $\mathcal{N}_{\tt ns}^{(2)}$:} 
		Here we also argue similarly, but apply Claim~\ref{claim:small_ball_points_near_the_real_axis} instead of Claim~\ref{claim:small_ball_away_from_real_axis}:
		\begin{multline*}
			\bP\Big[|f_n(z)| \le 2\delta \eps d^{5/4} \, , \ |f_n^\prime(z)| \le 2 \eps d^{5/4} \Big] \\ \le \bP\Big[|f_n(z)| \le d^{3\tau} \, , \ |f_n^\prime(z)| \le 2 \eps d^{3/2-1/4} \Big] \lesssim d^{-3/2 + 6\tau} \eps^2 \le \eps^2 \, .
		\end{multline*}
		As $|\mathcal{N}_{\tt ns}^{(2)}| \lesssim d^{1/2 + 2\beta}$, this yields that 
		\begin{equation}
			\label{eq:with_high_probability_no_roots_with_too_small_derivative_bound_on_nonsmooth_points_2}
			\sum_{z\in \mathcal{N}_{\tt s}^{(2)}} \bP\Big[|f_n(z)| \le 2\delta \eps d^{5/4} \, , \ |f_n^\prime(z)| \le 2 \eps d^{5/4} \Big] \lesssim d^{-1+2\beta + 6\tau} \eps^2 \le \eps^2 \, .
		\end{equation}
		\noindent
		\underline{The sum over $\mathcal{N}_{\tt ns}^{(3)}$:} Since $|\mathcal{N}_{\tt ns}^{(3)}| \lesssim d^{1/2+2\beta-\tau/4}$, it suffices to use the very crude Berry-Esseen bound. Note that for $z\in \mathcal{N}_{\tt ns}^{(3)}$, we have that $\min\{(1-\text{Re}(z))^{-1},n\} \ge d/2$, and so Claim~\ref{claim:small_ball_real_points} gives that
		\begin{equation*}
			\bP\Big[|f_n(z)| \le 2\delta \eps d^{5/4} \, , \ |f_n^\prime(z)| \le 2 \eps d^{5/4} \Big] \le \bP\Big[|f_n(z)| \le 1  \Big] \lesssim d^{-1/2}  
		\end{equation*}
		for all $z\in \mathcal{N}_{\tt ns}^{(3)}$. We get that
		\begin{equation}
			\label{eq:with_high_probability_no_roots_with_too_small_derivative_bound_on_nonsmooth_points_3}
			\sum_{z\in \mathcal{N}_{\tt s}^{(3)}} \bP\Big[|f_n(z)| \le 2\delta \eps d^{5/4} \, , \ |f_n^\prime(z)| \le 2 \eps d^{5/4} \Big] \lesssim d^{-\tau/4 + 2\beta} \, .
		\end{equation}
		
		To conclude the proof, we plug in the bounds~\eqref{eq:with_high_probability_no_roots_with_too_small_derivative_bound_on_smooth_points},~\eqref{eq:with_high_probability_no_roots_with_too_small_derivative_bound_on_nonsmooth_points_1},~\eqref{eq:with_high_probability_no_roots_with_too_small_derivative_bound_on_nonsmooth_points_2} and~\eqref{eq:with_high_probability_no_roots_with_too_small_derivative_bound_on_nonsmooth_points_3} into the inequality~\eqref{eq:with_high_probability_no_roots_with_too_small_derivative_after_union_bound} to get that
		\begin{equation*}
			\bP\Big[\exists \alpha \in \mathcal{A}_k \, : \ f_n(\alpha) = 0\, , \ |f_n^\prime(\alpha)| \le \eps d^{-5/4} \Big] \lesssim \eps^2 + d^{-\tau/4 + 2\beta} + e^{-c\log^2 d}\, ,
		\end{equation*}
		which, by taking $d$ to be large enough (which amounts to taking $r_0$ sufficiently close to $1$), concludes the proof of the lemma.
	\end{proof}
	\subsection{From close roots to quadratic approximations}
	Looking ahead to the proof of Proposition~\ref{prop:no_close_double_roots_in_the_bulk}, the main idea is to show that under the quasi-random condition that the derivatives of $f_n$ are controlled (Lemma~\ref{lemma:control_on_maximum_of_polynomial_and_derivatives}) and that the derivative evaluated at a root cannot be too small (Lemma~\ref{lemma:with_high_probability_no_roots_with_too_small_derivative}), two close roots implies some small-ball relation for the random vector $\big(f_n(z),f_n^\prime(z),f_n^{\prime\prime}(z)\big)$, evaluated at a net point $z\in \mathcal{N}$ nearby. The precise statement takes form in the next two claims.
	\begin{claim}
		\label{claim:close_roots_implies_ratio_of_derivative_and_second_derivative_to_be_small}
		For all $\beta\in(0,10^{-2})$ and for all $d\ge d_0(\beta)$ large enough the following holds. For all $z\in \bC$, $\eps>d^{-\beta}$ and $t\in(0,d^{-5/4+\beta})$, suppose that $f$ is an analytic function for which we have
		\[
		\max_{\zeta \in \bD(z,d^{-5/4+\beta})} |f^{(j)}(\zeta)| \le d^{j+1/2} \log^2 d \, , \qquad \text{for } j\in \{2,3\}. 
		\]
		If there exists $w\in \bD(z,t)$ for which $f(z)= f(w) = 0$ and $|f^\prime(z)| \ge \eps d^{5/4}$, then
		\[
		\frac{2|f^\prime(z)|}{|f^{\prime\prime}(z)|} \le t \cdot  \big(1+d^{-\beta}\big) \, .
		\]
	\end{claim}  
	\begin{proof}
		Taylor's theorem implies that
		\[
		\Big|f(w) - f(z) - (w-z) f^\prime(z) - \frac{(w-z)^2}{2} f^{\prime\prime}(z) \Big| \le |w-z|^3 d^{7/2} \log^2 d \, ,
		\]
		which shows that for $d$ large enough,
		\[
		|f^\prime(z)| \le \frac{t}{2}|f^{\prime\prime}(z)| + d^{1+3\beta} \, .
		\]
		Recalling the lower bound on $|f^\prime(z)|$, we see that
		\begin{equation*}
			|f^\prime(z)| \big(1+ d^{-\beta}\big)^{-1} \le |f^\prime(z)| - d^{1+3\beta} \le \frac{t}{2} |f^{\prime\prime}(z)| \, .  \qedhere
		\end{equation*}
	\end{proof}
	\begin{claim}
		\label{claim:root_with_small_derivative_implies_smallness_in_net_point}
		For all $\beta\in(0,10^{-2})$ and for all $d\ge d_0(\beta)$ large enough the following holds. For all $z\in \bC$, $\eps\ge d^{-\beta/2}$, $t\in(0,d^{-5/4+\beta})$ and $\delta \in (0,d^{-5/4-\beta})$, suppose that $f$ is an analytic function for which
		\[
		\max_{\zeta \in \bD(z,d^{-5/4+\beta})} |f^{(j)}(\zeta)| \le d^{j+1/2} \log^2 d \, , \qquad \text{for } j\in \{2,3\}. 
		\]
		If there exists $\alpha\in \bD(z,\delta)$ for which $f(\alpha) = 0$, $|f^\prime(\alpha)| \ge \eps d^{5/4}$ and $ \tfrac{2 |f^\prime(\alpha)|}{|f^{\prime\prime}(\alpha)|} \le  t$, then
		\[
		|f(z)| \le \delta \cdot  |f^\prime(z)| \cdot \big(1+d^{-\beta}\big) \, , \qquad \text{and} \qquad \frac{2|f^\prime(z)|}{|f^{\prime \prime}(z)|} \le t \cdot \big(1+ d^{-\beta/2}\big) \, .
		\]
	\end{claim}
	\begin{proof}
		We start by noting that
		\[
		\big|f^\prime(z) - f^\prime(\alpha)\big| \le \delta \max_{\zeta \in \bD(z,d^{-5/4+\beta})} |f^{\prime\prime}(\zeta)| \le d^{5/4-\beta} \log^2 d\, ,
		\]
		and so, by the lower bound on $|f^\prime(\alpha)|$, we have
		\[
		|f^\prime(z)| \ge |f^\prime(\alpha)| - |f^\prime(z) - f^\prime(\alpha)| \ge \eps d^{5/4 - \beta/2} 
		\] 
		for all $d$ large enough. Combining this with the inequality
		\[
		|f(z)| = |f(z) - f(\alpha)| \le \delta |f^\prime(z)| + \frac{\delta^2}{2} d^{5/2} \log^2 d \, ,
		\]
		gives
		\[
		|f(z)| \le \delta \cdot |f^\prime(z)| \cdot \big(1+ d^{-\beta}\big) \, .
		\]
		To obtain the second conclusion of the claim, we use a similar argument to bound
		\begin{align*}
			|f^\prime(z)| &\le |f^{\prime}(\alpha)| + d^{5/4-\beta} \log^2 d \\ &\le \frac{1}{2} t|f^{\prime\prime}(\alpha)| + d^{5/4-\beta} \log^2 d \\  & \le \frac{1}{2} t|f^{\prime\prime}(z)| + \frac{1}{2} t \delta d^{7/2} \log^2 d + d^{5/4-\beta} \log^2 d \le \frac{1}{2} t|f^{\prime\prime}(z)| + O(d^{5/4-\beta} \log^2 d) \, .
		\end{align*}
		Since $|f^\prime(z)| \ge \eps d^{5/4-\beta/2}$, we get that
		\[
		|f^\prime(z)| - O(d^{5/4-\beta} \log^2 d) \ge |f^\prime(z)| \big(1+d^{-\beta/2}\big)^{-1}
		\]
		which gives what we wanted.
	\end{proof}
	\subsection{Bounding the probability of close roots on each annulus}
	We will now apply a net argument once more to show that with high probability there are no close roots in each of the annuli $\mathcal{A}_k$ for $k\in \{0,\ldots,N\}$. 
	\begin{lemma}
		\label{lemma:with_high_probabiity_no_close_roots_in_given_annulus}
		There exists absolute constants $C,\beta>0$ and $r_0<1$ so that the following holds. For each $k\in \{0,1,\ldots,N\}$, $\eps\in( d_k^{-\beta/2},1)$ and $\gamma \in (d_k^{-\beta},\eps)$ we have 
		\[
		\bP\Big[\exists \alpha \in \mathcal{A}_k \ \text{and } \ \exists \alpha^\prime \in \bD(\alpha,\gamma d_k^{-5/4}) \setminus\{\alpha \} \ \text{such that } \ f_n(\alpha) = f_n(\alpha^\prime) = 0\Big] \le C \eps \, .
		\]
	\end{lemma}
	\begin{proof}
		As $k\in \{0,\ldots,N\}$ is kept fixed throughout the proof we will write $d=d_k = (1-\rho_k)^{-1}$ to lighten  the notation. Denote the event
		\[
		\mathcal{E}_1 = \Big\{ \max_{\zeta \in \rho_k \bD} |f_n^{(j)}(\zeta)| \le d^{j+1/2} \log^2 d \, , \ j\in \{2,3\} \Big\} \, ,
		\]
		and note that $\bP\big[\mathcal{E}_1^{c}\big] \le \exp(-c\log^2 d)$ by Lemma~\ref{lemma:control_on_maximum_of_polynomial_and_derivatives}. We will also consider the event
		\[
		\mathcal{E}_2 = \Big\{ \forall \alpha\in \mathcal{A}_k \ \text{such that} \ f_n(\alpha) = 0 \ \text{we have} \  |f_n^\prime(\alpha)| \ge \eps d^{5/4} \Big\}
		\]
		and note that, by Lemma~\ref{lemma:with_high_probability_no_roots_with_too_small_derivative} we have
		\[
		\bP\big[\mathcal{E}_2^c\big] = \bP\Big[\exists \alpha\in \mathcal{A}_k \ \text{such that} \ f_n(\alpha) = 0 \ \text{and} \ |f_n^\prime(\alpha)| \le \eps d^{5/4} \Big] \le C\eps \, .
		\]
		Recall the definition of the net $\mathcal{N} = \mathcal{N}_k$ from the proof of Lemma~\ref{lemma:with_high_probability_no_roots_with_too_small_derivative}, that is, the net as in Definition~\ref{definition:net_in_annulus} with $\delta = d^{-5/4-\beta}$ and $M = \lfloor 10/\delta \rfloor$. By combining Claim~\ref{claim:close_roots_implies_ratio_of_derivative_and_second_derivative_to_be_small} and Claim~\ref{claim:root_with_small_derivative_implies_smallness_in_net_point} with the bounds above, we see that  
		\begin{multline}
			\label{eq:with_high_probabiity_no_close_roots_in_given_annulus_after_union_bound}
			\bP\Big[\exists \alpha \in \mathcal{A}_k \ \text{and } \ \exists \alpha^\prime \in \bD(\alpha,\gamma d_k^{-5/4}) \setminus\{\alpha \} \ \text{such that } \ f_n(\alpha) = f_n(\alpha^\prime) = 0\Big] \\ \le \sum_{z\in \mathcal{N}} \bP\Big[|f_n(z)| \le 2 \delta \cdot  |f_n^\prime(z)| \, ,  \tfrac{1}{2}\eps d^{5/4} \le |f_n^\prime(z)| \le |f_n^{\prime\prime}(z)| \cdot \gamma d^{-5/4} \, , \, \mathcal{E}_1 \cap \mathcal{E}_2 \Big] + C\eps + e^{-c\log^2 d} \, .
		\end{multline}
		Therefore, it remains to show that the sum~\eqref{eq:with_high_probabiity_no_close_roots_in_given_annulus_after_union_bound} is $O(\eps^{-1} \gamma^2)$ and the proof will be finished. Indeed, we recall that notation $\mathcal{N}_{\tt s}$ and $\mathcal{N}_{\tt ns}^{(1)},\mathcal{N}_{\tt ns}^{(2)},\mathcal{N}_{\tt ns}^{(3)}$ as given by~\eqref{eq:decomposition_into_smooth_and_non_smooth_in_the_bulk} and~\eqref{eq:further_decomposition_of_non_smooth_in_the_bulk} and show the desired bound holds for each sum separately. 
		
		\noindent
		\underline{The sum over $\mathcal{N}_{\tt s}$:}
		We break down the probability in~\eqref{eq:with_high_probabiity_no_close_roots_in_given_annulus_after_union_bound} into a dyadic sum, according to the possible values of $|f_n^\prime(z)|$. We have
		\begin{align}
			\label{eq:with_high_probabiity_no_close_roots_in_given_annulus_breakdown_to_dyadic_values_of_derivative} \nonumber
			\bP&\Big[|f_n(z)| \le 2 \delta \cdot  |f_n^\prime(z)| \, ,  \tfrac{1}{2}\eps d^{5/4} \le |f_n^\prime(z)| \le |f_n^{\prime\prime}(z)| \cdot \gamma d^{-5/4} \, , \, \mathcal{E}_1\cap \mathcal{E}_2 \Big] \\ \nonumber & \le \sum_{k=0}^{\lfloor 4\log \log d\rfloor} \bP\Big[|f_n(z)| \le 2 \delta \cdot  |f_n^\prime(z)| \, ,  \, |f_n^\prime(z)| \le |f_n^{\prime\prime}(z)| \cdot \gamma d^{-5/4} \, , \, |f_n^\prime(z)| \in [\eps 2^{k-1} d^{5/4}, \eps 2^{k} d^{5/4}]\Big] \\ & \le \sum_{k=0}^{\lfloor 4\log \log d\rfloor} \bP\Big[|f_n(z)| \le 2^{k+1} \eps d^{-\beta} \, ,  \, |f_n^\prime(z)| \le \eps 2^{k} d^{5/4} \, ,  |f_n^{\prime\prime}(z)| \ge  \eps d^{5/2} 2^{k-1}/\gamma \Big] \, . 
		\end{align}
		Since the points $z\in \mathcal{N}_{\tt s}$ are $d^{7\tau}$-smooth, Lemma~\ref{lemma:small_ball_probability_smooth_points_bound_single_point} implies that for each $k=0,\ldots \lfloor 4 \log \log d \rfloor$ we have the bound
		\begin{multline*}
			\bP\Big[|f_n(z)| \le 2^{k+1} \eps d^{-\beta} \, ,  \, |f_n^\prime(z)| \le \eps 2^{k} d^{5/4} \, ,  |f_n^{\prime\prime}(z)| \ge  \eps d^{5/2} 2^{k-1}/\gamma \Big] \\ \lesssim 16^k \eps^4 d^{-3/2-2\beta} \exp\big(-c\eps^2  4^{k}/\gamma^2\big) \, .
		\end{multline*}
		Plugging into~\eqref{eq:with_high_probabiity_no_close_roots_in_given_annulus_breakdown_to_dyadic_values_of_derivative} gives that 
		\begin{multline*}
			\bP\Big[|f_n(z)| \le 2 \delta \cdot  |f_n^\prime(z)| \, ,  \tfrac{1}{2}\eps d^{5/4} \le |f_n^\prime(z)| \le |f_n^{\prime\prime}(z)| \cdot \gamma d^{-5/4} \, , \, \mathcal{E}_1\cap \mathcal{E}_2 \Big] \\ \lesssim d^{-3/2-2\beta} \eps^4  \sum_{k=0}^{\lfloor 4\log \log d\rfloor} 16^k  \exp\big(-c\eps^2  4^{k}/\gamma^2\big) \lesssim d^{-3/2-2\beta} \, \eps^4 \exp\big(-c\eps^2/\gamma^2\big) \, .
		\end{multline*}
		Since $|\mathcal{N}_{\tt s}| \lesssim d^{3/2 + 2\beta}$, we get that
		\begin{equation}
			\label{eq:with_high_probabiity_no_close_roots_in_given_annulus_after_union_bound_sum_over_smooth_points}
			\sum_{z\in \mathcal{N}_{\tt s}} \bP\Big[|f(z)| \le 2 \delta \cdot  |f^\prime(z)| \, ,  \tfrac{1}{2}\eps d^{5/4} \le |f^\prime(z)| \le |f^{\prime\prime}(z)| \cdot \gamma d^{-5/4} \, , \, \mathcal{E}_1 \cap \mathcal{E}_2 \Big] \lesssim \eps^4 \exp\big(-c\eps^2/\gamma^2\big) \lesssim \eps \, . 
		\end{equation}
		
		\noindent
		\underline{The sum over $\mathcal{N}_{\tt ns}^{(1)}$:}
		Here we do not need to break down the probability dyadically. We first simply bound 
		\begin{align}
			\bP\Big[|f(z)| \le 2 \delta \cdot|f^\prime(z)| \, ,  \tfrac{1}{2}\eps d^{5/4} \le |f^\prime(z)| \le& |f^{\prime\prime}(z)| \cdot \gamma d^{-5/4} \, , \, \mathcal{E}_1 \cap \mathcal{E}_2 \Big] \nonumber \\
			&\leq 	\bP\Big[|f(z)| \le 2 \gamma d^{-\beta} \log^2 d  \, , |f^{\prime}(z)| \leq  \gamma d^{5/4} \log^2 d \Big] \label{eq:non-dyadic-bound}
		\end{align}
		and subsequently bound 
		\begin{align*}
		\bP\Big[|f(z)| \le 2 \gamma d^{-\beta} \log^2 d  \, , |f^{\prime}(z)| \leq  \gamma d^{5/4} \log^2 d \Big] & \le \bP\Big[|f_n(z)| \le \eps \, ,  \, |f_n^\prime(z)| \le \eps  d^{5/4} \log^2 d  \Big] \\ &\lesssim  \frac{\log^4 d}{d^{3/2}} \, \eps^4\, .
\end{align*}
		Since $|\mathcal{N}_{\tt ns}^{(1)}| \lesssim d^{1/2 + \beta + 7\tau}$, we bound 
		\begin{multline}
			\label{eq:with_high_probabiity_no_close_roots_in_given_annulus_after_union_bound_sum_over_nonsmooth_points_1}
			\sum_{z\in \mathcal{N}_{\tt ns}^{(1)}} \bP\Big[|f(z)| \le 2 \delta \cdot  |f^\prime(z)| \, ,  \tfrac{1}{2}\eps d^{5/4} \le |f^\prime(z)| \le |f^{\prime\prime}(z)| \cdot \gamma d^{-5/4} \, , \, \mathcal{E}_1 \cap \mathcal{E}_2 \Big]  \lesssim \eps^4  \frac{d^{\beta+8\tau}}{d} \lesssim \eps \, . 
		\end{multline}
		
		\noindent
		\underline{The sum over $\mathcal{N}_{\tt ns}^{(2)}$:}
		The argument here is identical to the argument above for $\mathcal{N}_{\tt ns}^{(2)}$. The only difference is that we need to apply Claim~\ref{claim:small_ball_points_near_the_real_axis} instead of Claim~\ref{claim:small_ball_away_from_real_axis} (as points $z\in \mathcal{N}_{\tt ns}^{(2)}$ are at distance at least $d^{-\tau/4}$ from the real axis), which in turn adds an extra $d^{2\tau}$ factor on the right-hand side of~\eqref{eq:with_high_probabiity_no_close_roots_in_given_annulus_after_union_bound_sum_over_nonsmooth_points_1}, which does no harm. We get that
		\begin{multline}
			\label{eq:with_high_probabiity_no_close_roots_in_given_annulus_after_union_bound_sum_over_nonsmooth_points_2}
			\sum_{z\in \mathcal{N}_{\tt ns}^{(2)}} \bP\Big[|f(z)| \le 2 \delta \cdot  |f^\prime(z)| \, ,  \tfrac{1}{2}\eps d^{5/4} \le |f^\prime(z)| \le |f^{\prime\prime}(z)| \cdot \gamma d^{-5/4} \, , \, \mathcal{E}_1 \cap \mathcal{E}_2 \Big] \lesssim \eps^4 \frac{d^{\beta+10\tau}}{d} \lesssim \eps \, . 
		\end{multline}
		
		\noindent
		\underline{The sum over $\mathcal{N}_{\tt ns}^{(3)}$:} Recall that for all $z\in \mathcal{N}_{\tt ns}^{(3)}$ we have that $\min\{(1-\text{Re}(z))^{-1},n\} \ge d/2$. Using the definition of the event $\mathcal{E}_1$, we can bound the probability from~\eqref{eq:with_high_probabiity_no_close_roots_in_given_annulus_after_union_bound} by
		\begin{align*}
			\bP\Big[|f_n(z)| \le 2 \delta \cdot  &|f_n^\prime(z)| \, ,  \tfrac{1}{2}\eps d^{5/4} \le |f_n^\prime(z)| \le |f_n^{\prime\prime}(z)| \cdot \gamma d^{-5/4} \, , \, \mathcal{E}_1 \cap \mathcal{E}_2 \Big] \\ & \le \bP\Big[ |f_n(z)| \le 2\gamma d^{-\beta} \log^2 d \, , \, |f_n^\prime(z)| \le \gamma d^{5/4} \log^2 d\Big]  \le \bP\Big[ |f_n(z)| \le 1\Big] \lesssim  d^{-1/2} 
		\end{align*}
		where the last inequality follows Claim~\ref{claim:small_ball_real_points}. Since $|\mathcal{N}_{\tt ns}^{(3)}| \lesssim d^{1/2+2\beta-\tau/4}$ we get that
		\begin{equation}
			\label{eq:with_high_probabiity_no_close_roots_in_given_annulus_after_union_bound_sum_over_nonsmooth_points_3}
			\sum_{z\in \mathcal{N}_{\tt ns}^{(3)}} \bP\Big[|f(z)| \le 2 \delta \cdot  |f^\prime(z)| \, ,  \tfrac{1}{2}\eps d^{5/4} \le |f^\prime(z)| \le |f^{\prime\prime}(z)| \cdot \gamma d^{-5/4} \, , \, \mathcal{E}_1 \cap \mathcal{E}_2 \Big] \le d^{2\beta-\tau/4} \le \eps \, .
		\end{equation} 
		Plugging the bounds~\eqref{eq:with_high_probabiity_no_close_roots_in_given_annulus_after_union_bound_sum_over_smooth_points}, \eqref{eq:with_high_probabiity_no_close_roots_in_given_annulus_after_union_bound_sum_over_nonsmooth_points_1}, \eqref{eq:with_high_probabiity_no_close_roots_in_given_annulus_after_union_bound_sum_over_nonsmooth_points_2} and~\eqref{eq:with_high_probabiity_no_close_roots_in_given_annulus_after_union_bound_sum_over_nonsmooth_points_3} into the inequality~\eqref{eq:with_high_probabiity_no_close_roots_in_given_annulus_after_union_bound} finally gives that
		\begin{equation*}
			\bP\Big[\exists \alpha \in \mathcal{A}_k \ \text{and } \ \exists \alpha^\prime \in \bD(\alpha,\gamma d_k^{-5/4}) \setminus\{\alpha \} \ \text{such that } \ f_n(\alpha) = f_n(\alpha^\prime) = 0\Big] \lesssim  \eps\, . \qedhere
		\end{equation*}
	\end{proof}
	We are now ready to prove the main result of this section.
	\begin{proof}[Proof of Proposition~\ref{prop:no_close_double_roots_in_the_bulk}]
		We start by proving that for each $\eps>0$ and $L\ge 1$ there is an $r_0<1$ and $K>0$ so that
		\begin{equation}
			\label{eq:proof_no_close_double_roots_in_the_bulk_first_inequality_needed}
			\bP\Big[\exists \alpha \in \bD\big(r_0,1-K/n\big) \, , \ \exists \alpha^\prime\in \bD(\alpha,Ln^{-5/4})\setminus\{\alpha\} \ \text{s.t.} \ f_n(\alpha) = f_n(\alpha^\prime) = 0 \Big] \le \eps/2  \, ,
		\end{equation}
		for all $n\ge 1$ large enough. Recall that $d_k = (1-\rho_{k})^{-1}$ where $\rho_k$ is given by~\eqref{eq:def_of_rho_k}. We apply Lemma~\ref{lemma:with_high_probabiity_no_close_roots_in_given_annulus} with $\gamma_k = \max\{d_k^{-\beta}, L (d_k/n)^{5/4}\}$ and $\eps_k = \sqrt{\gamma_k}$ and see that
		\[
		\bP \Big[\exists \alpha \in \mathcal{A}_k \ \text{and } \ \exists \alpha^\prime \in \bD(\alpha,\gamma d_k^{-5/4}) \setminus\{\alpha \} \ \text{such that } \ f_n(\alpha) = f_n(\alpha^\prime) = 0\Big] \lesssim d_k^{-\beta/2} + L^4 \Big(\frac{d_k}{n}\Big)^{5} \,.
		\]
		Hence, we can apply the union bound and sum over the sequence of dyadic annuli $\{\mathcal{A}_k\}_{k=0}^{N}$ to get
		\begin{align*}
			\bP\Big[\exists \alpha \in &\bD\big(r_0,1-K/n\big) \, , \ \exists \alpha^\prime\in \bD(z,Ln^{-5/4})\setminus \{\alpha\} \ \text{s.t.} \ f_n(\alpha) = f_n(\alpha^\prime) = 0 \Big] \\ &\le \sum_{k=0}^{N}\bP \Big[\exists \alpha \in \mathcal{A}_k \ \text{and } \ \exists \alpha^\prime \in \bD(\alpha,\gamma d_k^{-5/4}) \setminus\{\alpha \} \ \text{such that } \ f_n(\alpha) = f_n(\alpha^\prime) = 0\Big] \\ & \lesssim \sum_{k=0}^{N} d_k^{-\beta/2} + L^4 \Big(\frac{d_k}{n}\Big)^{5} \lesssim d_N^{-\beta/2} + L^4 \Big(\frac{d_0}{n}\Big)^{5} \lesssim (1-r_0)^{\beta/2} + \frac{L^4}{K^5} \, .
		\end{align*}
		Taking now $r_0$ close to $1$ and $K\ge 1$ large enough yields~\eqref{eq:proof_no_close_double_roots_in_the_bulk_first_inequality_needed}. It remains to handle the close roots from the disk $r_0\bD$, for some $r_0 < 1$ fixed. Indeed, we will prove the slightly stronger statement: for each $r_0 <1 $ and $\eps>0$ there exists $\delta >0$ so that
		\begin{equation}
			\label{eq:proof_no_close_double_roots_in_the_bulk_second_inequality_needed}
			\limsup_{n\to \infty} \bP\Big[\exists \alpha \in r_0 \bD \, , \ \exists \alpha^\prime\in \bD(\alpha,\delta)\setminus\{\alpha\} \ \text{s.t.} \ f_n(\alpha) = f_n(\alpha^\prime) = 0 \Big] \le \eps/2  \, ,
		\end{equation}
		Since $L/n^{5/4} \le \delta$ for large enough $n$, we can combine~\eqref{eq:proof_no_close_double_roots_in_the_bulk_first_inequality_needed} and~\eqref{eq:proof_no_close_double_roots_in_the_bulk_second_inequality_needed} together with the union bound to conclude the proof of Proposition~\ref{prop:no_close_double_roots_in_the_bulk}, so it only remains to prove~\eqref{eq:proof_no_close_double_roots_in_the_bulk_second_inequality_needed}. 
		
		Recall the definition~\eqref{eq:intro-def_of_infinite_power_series} of the random Taylor series $F$. By Theorem~\ref{thm:almost_sure_double_zeros_in_disk}, we see that
		\begin{equation*}
			\lim_{\delta\downarrow 0} \bP\Big[\exists \alpha \in r_0 \bD \, , \ \exists \alpha^\prime\in \bD(\alpha,\delta)\setminus \{\alpha\} \ \text{s.t.} \ F(\alpha) = F(\alpha^\prime) = 0 \Big]  = \bP\Big[\exists \alpha \in r_0\bD \, : F(\alpha) = F^\prime(\alpha) = 0 \Big] = 0 \, .
		\end{equation*}
		Therefore, for some $\delta>0$ we have
		\[
		\bP\Big[\exists \alpha \in r_0 \bD \, , \ \exists \alpha^\prime\in \bD(\alpha,\delta)\setminus \{\alpha\} \ \text{s.t.} \ F(\alpha) = F(\alpha^\prime) = 0 \Big] \le \eps/2 \, .
		\]
		By Hurwitz's theorem, as $n\to \infty$ the roots of $f_n$ inside $r_0\bD$ converge almost surely to those of the analogous Taylor series $F$. In particular, we have the convergence
		\begin{multline*}
			\lim_{n\to \infty} \bP\Big[\exists \alpha \in r_0 \bD \, , \ \exists \alpha^\prime\in \bD(\alpha,\delta)\setminus \{\alpha\} \ \text{s.t.} \ f_n(\alpha) = f_n(\alpha^\prime) = 0 \Big] \\ = \bP\Big[\exists \alpha \in r_0 \bD \, , \ \exists \alpha^\prime\in \bD(\alpha,\delta)\setminus \{\alpha\} \ \text{s.t.} \ F(\alpha) = F(\alpha^\prime) = 0 \Big] \le \eps/2
		\end{multline*}
		which proves~\eqref{eq:proof_no_close_double_roots_in_the_bulk_second_inequality_needed}, and we are done.
	\end{proof}
	\section{Reducing Theorem~\ref{thm:poisson_convergence_near_unit_circle} to a count over a net}
	\label{sec:reducing_poisson_limit_to_count_over_a_net}
	Let $K\ge 1$ be a fixed large constant, and recall the notation $\Omega_K = \mathbb{H} \cap \mathcal{A}(1-K/n,1+K/n)$ given by~\eqref{eq:def_of_Omega_k}. The main goal of this section will be to prove Theorem~\ref{thm:poisson_convergence_near_unit_circle}, which states that the random variable $X_n(U)$ given by~\eqref{eq:def_of_X_n_U} converges in distribution to a Poisson random variable as $n\to \infty$, for all fixed $K$. Below we will reduce the count of close roots to a count over a net of smooth points in $\Omega_K$ whose quadratic approximation suggests there will be two close roots present nearby. After that, we will prove the Poisson limit holds for the count over the net via a moment computation (see Proposition~\ref{prop:poisson_moments_for_X_n_pm} below), which will be performed in Section~\ref{sec:poisson_limit_for_sum_over_the_smooth_net} below. We start by defining the following net in the annulus $\Omega_K$.
	\begin{definition}
		\label{definition:net_in_the_main_annulus}
		Set $\delta = n^{-5/4-\beta}$ and $M_1 = \lceil 4K/(\delta n) \rceil$, $M_2 = \lceil 4/\delta \rceil$. The $\delta$-net $\mathsf{N}$ of $\Omega_K$ by
		\[
		\mathsf{N} = \bigg\{ \Big(1-\frac{K}{n} + \frac{2K}{n} \cdot\frac{a}{M_1}\Big) e\big(\pi b/M_2\big) \, : \, a\in\{0,1,\ldots,M_1\} \, , \, b\in \{1,\ldots,M_2\} \bigg\} \, .
		\]
		For a point $z = \big(1-\frac{K}{n} + \frac{2K}{n} \cdot\frac{a}{M_1}\big) e\big(\pi b/M_2\big)  \in \mathsf{N}$ we further denote by
		\begin{equation*}
			\mathsf{R}_z = \bigg\{ \Big(1-\frac{K}{n} + \frac{2K}{n} \cdot\frac{a+s}{M_1}\Big) e\big(\pi (b+t)/M_2\big) \, : \, s,t\in[-1/2,1/2] \bigg\} \cap \mathbb{H} \, .
		\end{equation*} 
		Note that the collection $\{\mathsf{R}_z \}_{z\in \mathsf{N}}$ of (polar) rectangles cover the annulus $\Omega_K$. We will also use the notation $\mathsf{R}_z^\circ \subset \mathsf{R}_z$, to denote the smaller rectangles
		\begin{equation*}
			\mathsf{R}_z^\circ = \bigg\{ \Big(1-\frac{K}{n} + \frac{2K}{n} \cdot\frac{a+s}{M_1}\Big) e\big(\pi (b+t)/M_2\big) \, : \, s,t\in[-1/2+n^{-\beta/2},1/2-n^{-\beta/2}] \bigg\} \cap \mathbb{H} \, .
		\end{equation*} 
		We further split the net into the smooth points and non-smooth points, analogous to what we did in~\eqref{eq:decomposition_into_smooth_and_non_smooth_in_the_bulk}. Denote by
		\[
		\mathsf{N}_{\tt s} = \Big\{ z\in \mathsf{N} \, : \,  z \ \text{is } n^{7\tau}\text{-smooth}\Big\} \, ,
		\]
		and set $\mathsf{N}_{\tt ns} = \mathsf{N} \setminus \mathsf{N}_{\tt s}$. 	\end{definition}
	For $z\in \mathsf{N}_{\tt s}$ and an finite union of intervals $U\subset \bR_{\ge 0}$ we define the event
	\begin{equation}
		\label{eq:def_of_event_A_z_U_for_net_point}
		A_z(U) = \left\{\frac{f_n(z)}{f_n^\prime(z)} - z \in \mathsf{R}_z^{\circ} \, , \ \frac{2|f_n^\prime(z)|}{|f_n^{\prime\prime}(z)|} \in n^{-5/4} U \, , \ |f_n^\prime(z)| \ge \frac{n^{5/4}}{\log n} \right\}
	\end{equation}
	The basic idea is that the event~\eqref{eq:def_of_event_A_z_U_for_net_point} implies that the quadratic approximation of $f_n$ at the point $z$ will predict a root in the rectangle $\mathsf{R}_z$, and also a second root at distance $n^{-5/4} U$ from $z$. Heuristically, we want to show that with high probability as $n\to \infty$,
	\begin{equation}
		\label{eq:heuristic_formula_for_X_n}
		X_n(U) = \frac{1}{2} \sum_{z \in \mathsf{N}_{\tt s}}  \mathbf{1}_{A_z(U)} \, .
	\end{equation}
	Furthermore, via a moment computation, we will show that the right-hand side of~\eqref{eq:heuristic_formula_for_X_n} has a Poisson scaling limit. The reason for the prefactor $\frac{1}{2}$ on the right-hand side of~\eqref{eq:heuristic_formula_for_X_n} is that each pair of close roots makes two net points ``ring", while it is only counted once in $X_n(U)$. In turn, the goal of this section is to prove~\eqref{eq:heuristic_formula_for_X_n} in its more technical form, given by Proposition~\ref{prop:two_sided_inequality_for_X_n} below as a two-sided inequality. In Section~\ref{sec:poisson_limit_for_sum_over_the_smooth_net} we will perform the moment computation which shows that both random variables converges in law to the same Poisson variable as $n\to \infty$.
	
	\subsection{Setting up the two-sided inequality}
	We will consider the following ``bad" events
	\begin{align}
		\label{eq:def_of_bad_events_in_annulus}
		\nonumber
		\mathcal{B}_1 &= \bigg\{ \exists \alpha\in \Omega_K \, : \, f_n(\alpha) = 0 , \, |f_n^\prime(\alpha)| \le \frac{2n^{5/4}}{\log n} \bigg\} \, , \\ \nonumber \mathcal{B}_2 &= \bigg\{ \exists \alpha \in \bigcup_{z\in \mathsf{N}_{\tt ns}} \mathsf{R}_z  \, ,  \ \text{and } \ \exists \alpha^\prime \in \bD\Big(\alpha,\frac{\log n}{n^{5/4}} \Big) \setminus\{\alpha \}\, : \, f_n(\alpha) = f_n(\alpha^\prime) = 0 \bigg\} \, , \\ \nonumber  \mathcal{B}_3 &= \bigg\{ \exists \alpha \in \bigcup_{z\in \mathsf{N}_{\tt s}} \mathsf{R}_z\setminus \mathsf{R}_z^\circ \, , \ \text{and } \ \exists \alpha^\prime \in \bD\Big(\alpha,\frac{\log n}{n^{5/4}} \Big) \setminus\{\alpha \}\, : \, f_n(\alpha) = f_n(\alpha^\prime) = 0 \bigg\} \, , \\ \mathcal{B} &= \mathcal{B}_1 \cup \mathcal{B}_2 \cup \mathcal{B}_3 \, .
	\end{align}
	We will also consider the ``good" event
	\begin{equation}
		\label{eq:def_of_good_event_in_annulus}
		\mathcal{G} = \bigg\{ \max_{0\le j \le 3} \max_{z\in \Omega_{2K}} \frac{|f_n^{(j)}(z)|}{n^{j+1/2}} \le \log^2 n \bigg\} \, ,
	\end{equation}
	and note that by Lemma~\ref{lemma:control_on_maximum_of_polynomial_and_derivatives} we have $\bP\big[\mathcal{G}^c\big] \le \exp(-c\log^2 n)$.
	\begin{lemma}
		\label{lemma:in_annulus_good_events_are_likely}
		With the above notations, we have $\displaystyle \lim_{n\to \infty} \bP\big[\mathcal{G}\cap \mathcal{B}^c\big] = 1$. 
	\end{lemma}
	On the likely event $\mathcal{G}\cap \mathcal{B}^c$, we can formulate~\eqref{eq:heuristic_formula_for_X_n} more precisely. To do this we will need to introduce some further notation. For $z\in \mathsf{N}$ define
	\begin{equation*}
		\mathsf{R}_z^\sharp = \bigg\{ \Big(1-\frac{K}{n} + \frac{2K}{n} \cdot\frac{a+s}{M_1}\Big) e\big(\pi (b+t)/M_2\big) \, : \, s,t\in[-1/2+n^{-3\beta/4},1/2-n^{-3\beta/4}] \bigg\} \cap \mathbb{H} \, ,
	\end{equation*} 
	and note that $\mathsf{R}_z^\circ \subset \mathsf{R}_z^\sharp \subset \mathsf{R}_z$. Furthermore, for a finite union of intervals $U\subset \bR_{\ge0}$, we set
	\begin{equation}
		\label{eq:blow_up_and_blow_down_of_U}
		U^+ = \Big\{x\in \bR_{\ge 0} \, : \, \text{dist}(x,U) \le \frac{1}{\log n} \Big\}\, ,  \quad U^- = \Big\{x\in U \, : \, \text{dist}(x,U^c) \ge \frac{1}{\log n} \Big\} \, .
	\end{equation}
	Finally, we set 
	\begin{equation}
		\label{eq:def_of_A_z_U_+}
		A_z^+(U) = \bigg\{\frac{f_n(z)}{f_n^\prime(z)} - z \in \mathsf{R}_z^{\sharp} \, , \ \frac{2|f_n^\prime(z)|}{|f_n^{\prime\prime}(z)|} \in n^{-5/4} U^+ \, , \ |f_n^\prime(z)| \ge \frac{n^{5/4}}{\log n} \bigg\} \, , 
	\end{equation}
	and 
	\begin{equation}
		\label{eq:def_of_A_z_U_-}
		A_z^-(U) = \bigg\{\frac{f_n(z)}{f_n^\prime(z)} - z \in \mathsf{R}_z^{\circ} \, , \ \frac{2|f_n^\prime(z)|}{|f_n^{\prime\prime}(z)|} \in n^{-5/4} U^- \, , \ |f_n^\prime(z)| \ge \frac{n^{5/4}}{\log n} \bigg\} \, .
	\end{equation}
	Note that $A_z^-(U)\subset A_z^+(U)$. This leads us to the definition
	\begin{equation}
		\label{eq:def_of_X_n_pm}
		X_n^\pm (U) = \sum_{\{z,w\} \in \binom{\mathsf{N}_{\tt s}}{2}} \mathbf{1}_{A_z^\pm(U)\cap A_w^\pm(U)} \,  \mathbf{1}_{|z-w| \le \log n/n^{5/4}} \, , 
	\end{equation}
	where we denote by $\binom{S}{2}$ the set of all unordered pairs from a finite set $S$.
	\begin{proposition}
		\label{prop:two_sided_inequality_for_X_n}
		On the event $\mathcal{G}\cap \mathcal{B}^c$ we have
		\[
		X_n^-(U) \le X_n(U) \le X_n^+(U) \, .
		\]
	\end{proposition}
	For $x\in \bR$ and $m\ge 1$ recall that the falling factorial is given by
	\[
	(x)_m = x(x-1)\cdots(x-m+1) \, .
	\]
	We will show that both $X_n^+(U)$ and $X_n^-(U)$ have limiting moments matching those of a Poisson random variable, with the intensity $\la_{K,U}$ given by~\eqref{eq:def_of_la_K_U}. This is stated in the next proposition, the proof of which is postponed to Section~\ref{sec:poisson_limit_for_sum_over_the_smooth_net} below. 
	\begin{proposition}
		\label{prop:poisson_moments_for_X_n_pm}
		Let $X_n^+(U)$ and $X_n^-(U)$ be given by~\eqref{eq:def_of_X_n_pm}. Then for all $m\ge 1$ we have
		\[
		\limsup_{n\to \infty} \bE\Big[\big(X_n^+(U)\big)_{m}\Big] \le \big(\la_{K,U}\big)^m \, ,
		\]
		and 
		\[
		\liminf_{n\to \infty}  \bE\Big[\big(X_n^-(U)\big)_{m}\Big] \ge \big(\la_{K,U}\big)^m \, .
		\]
	\end{proposition}
	Assuming for the moment Propositions~\ref{prop:two_sided_inequality_for_X_n} and~\ref{prop:poisson_moments_for_X_n_pm}, we can provide the proof of Theorem~\ref{thm:poisson_convergence_near_unit_circle}. 
	\begin{proof}[Proof of Theorem~\ref{thm:poisson_convergence_near_unit_circle}]
		Since $X_n^-(U) \le X_n^+(U)$ deterministically, Proposition~\ref{prop:poisson_moments_for_X_n_pm} implies that 
		\[
		\lim_{n\to \infty} \bE\Big[\big(X_n^+(U)\big)_{m}\Big] = \lim_{n\to \infty} \bE\Big[\big(X_n^-(U)\big)_{m}\Big] = \big(\la_{K,U}\big)^m \, .
		\]
		In particular, by the method of moments, we conclude that both $X_n^+(U)$ and $X_n^-(U)$ converge in distribution, as $n\to \infty$, to a Poisson random variable with parameter $\la_{K,U}$. By combing Proposition~\ref{prop:two_sided_inequality_for_X_n} with Lemma~\ref{lemma:in_annulus_good_events_are_likely}, we conclude that for all $s\in \bR$ we have
		\[
		\bP\big[X_n^+(U) \ge s\big] - o(1) \le \bP\big[X_n(U) \ge s\big] \le \bP\big[X_n^-(U) \ge s\big] + o(1)
		\] 
		as $n\to \infty$, which implies the desired convergence for $X_n(U)$ as well.
	\end{proof}
	The remaining of this section is devoted to the proof of Lemma~\ref{lemma:in_annulus_good_events_are_likely} and Proposition~\ref{prop:two_sided_inequality_for_X_n}.
	\subsection{Upper bound}
	Before proving the upper bound in Proposition~\ref{prop:two_sided_inequality_for_X_n}, we require the following application of Rouch\'e's theorem.
	\begin{claim}
		\label{claim:if_there_is_a_root_linear_approx_sees_it_same_with_second_root}
		For all $\beta>0$ and for each $K\ge 1$, $U\subset \bR_{\ge 0}$ the following holds for all large enough $n$. On the event $\mathcal{B}_1^c\cap \mathcal{G}$, if $\alpha\in R_z^\circ$ with $f_n(\alpha) = 0$ then
		\[
		z -\frac{f_n(z)}{f_n^\prime(z)} \in R_z^\sharp \, .
		\]
		Furthermore, if there exists some $\alpha^\prime$ such that $f_n(\alpha^\prime) = 0$ and $|\alpha-\alpha^\prime| \in n^{-5/4}U$, then 
		\[
		\frac{2|f_n^\prime(z)|}{|f_n^{\prime\prime}(z)|} \in U^+ \, .
		\] 
	\end{claim}
	\begin{proof}
		On the event $\mathcal{G}$, we can Taylor expand around $\alpha$ and observe that
		\begin{align*}
			\alpha - \Big(z-\frac{f_n(z)}{f_n^\prime(z)}\Big) &= \alpha-z + \frac{f_n(\alpha) + (z-\alpha) f_n^\prime(\alpha) + O(n^{-2\beta} \log^2n)}{f_n^\prime(\alpha) + O(n^{-\beta} \log^2 n)} \\ &\stackrel{\text{on } \mathcal{B}_1^c}{=}  (\alpha - z) - (\alpha - z)  + O(n^{-5/4-2\beta} \log^3 n) = O(n^{-5/4-2\beta} \log^3 n) \, , 
		\end{align*}
		which proves the first assertion. Another Taylor expansion shows also that
		\[
		f_n(\alpha^\prime) = f_n(\alpha) + (\alpha^\prime-\alpha)f_n^\prime(\alpha) + (\alpha^\prime - \alpha)^2 \frac{f_n^{\prime\prime}(\alpha)}{2} + O(|\alpha-\alpha^\prime|^3 n^{7/2} \log^2n) \, .
		\]
		Since $f_n(\alpha^\prime) = f_n(\alpha) = 0$, we can rearrange the above and get that 
		\begin{equation}
			\label{eq:proof_of_claim:if_there_is_a_root_linear_approx_sees_it_same_with_second_root}
			0 = f_n^\prime(\alpha) + (\alpha^\prime - \alpha) \frac{f_n^{\prime\prime}(\alpha)}{2} + O(|\alpha-\alpha^\prime|^2 n^{7/2} \log^2n) \, .
		\end{equation}
		Now, since $\mathcal{B}_1^c$ holds and since $U$ lies in a compact subset of $(0,\infty)$,~\eqref{eq:proof_of_claim:if_there_is_a_root_linear_approx_sees_it_same_with_second_root} implies that $|f_n^{\prime\prime}(\alpha)| \gtrsim n^{5/2}/\log n$, which by~\eqref{eq:proof_of_claim:if_there_is_a_root_linear_approx_sees_it_same_with_second_root} shows that
		\[
		\Big|\frac{2f_n^\prime(\alpha)}{f_n^{\prime\prime}(\alpha)}\Big| = |\alpha - \alpha^\prime| + O(n^{-3/2} \log^3 n) \, .
		\]
		On the event $\mathcal{G}$ we also have
		\[
		\frac{2f_n^\prime(z)}{f_n^{\prime\prime}(z)} = \frac{2f_n^\prime(\alpha)}{f_n^{\prime\prime}(\alpha)} + O(n^{-5/4-\beta} \log^4 n)
		\]
		which, together with the above, implies that $\frac{2|f_n^\prime(z)|}{|f_n^{\prime\prime}(z)|} \in U^+$, as desired.
	\end{proof}
	\begin{proof}[Proof of Proposition~\ref{prop:two_sided_inequality_for_X_n}: the upper bound]
		We want to show here that on the event $\mathcal{G} \cap \mathcal{B}^c$ we have $X_n(U) \le X_n^+(U)$. Indeed, denote by $\mathcal{Z}_n$ the set of roots of $f_n$, and note that
		\[
		X_n(U) = \sum_{\{\alpha,\alpha^\prime\} \in \binom{\mathcal{Z} \cap \Omega_K}{2}} \mathbf{1}_{\{|\alpha - \alpha^\prime| \in n^{-5/4} U\}} \, .
		\] 
		Set $\Omega_{\tt s}^\circ = \Omega_K\cap\big(\bigcup_{z\in \mathsf{N}_{\tt s}} \mathsf{R}_z^\circ\big)$. On the event $\mathcal{B}_2^c\cap \mathcal{B}_3^c$, we have
		\[
		X_n(U) = \sum_{\{\alpha,\alpha^\prime\} \in \binom{\mathcal{Z} \cap \Omega_{\tt s}^\circ}{2}} \mathbf{1}_{\{|\alpha - \alpha^\prime| \in n^{-5/4} U\}} \, .
		\]
		Next, we argue that on the event $\mathcal{G} \cap \mathcal{B}_1^c$, there is no pair of roots $\{\alpha,\alpha^\prime\}\in \binom{\mathcal{Z} \cap \Omega_{\tt s}^\circ}{2}$ so what $\alpha,\alpha^\prime\in \mathsf{R}_z^\circ$ for some $z\in \mathsf{N}_{\tt s}$. Indeed, if there was such a pair, then Taylor's theorem would imply that
		\[
		0= f_n(\alpha^\prime) - f_n(\alpha) = (\alpha^\prime - \alpha) f_n^\prime(\alpha) + O(|\alpha^\prime - \alpha|^2 \, n^{5/2} \log^2n) \, ,
		\]
		which in turn shows that $|f_n^\prime(\alpha)| = O(n^{5/4 - \beta} \log^2 n)$, which cannot hold on $\mathcal{B}_1^c$. Combining the above, we see that on the event $\mathcal{G}\cap \mathcal{B}^c$ we can write 
		\[
		X_n(U) = \sum_{\{z,w\} \in \binom{\mathsf{N}_{\tt s}}{2}} \mathbf{1}_{\big\{\exists \alpha\in \mathsf{R}_z^\circ, \, \exists \alpha^\prime\in \mathsf{R}_w^\circ \, : \,  f_n(\alpha) = f_n(\alpha^\prime) = 0 \ \text{and} \ |\alpha - \alpha^\prime| \in n^{-5/4} U  \big\}} \, .
		\] 
		Claim~\ref{claim:if_there_is_a_root_linear_approx_sees_it_same_with_second_root} shows that
		\[
		\mathbf{1}_{\big\{\exists \alpha\in \mathsf{R}_z^\circ, \, \exists \alpha^\prime\in \mathsf{R}_w^\circ \, : \,  f_n(\alpha) = f_n(\alpha^\prime) = 0 \ \text{and} \ |\alpha - \alpha^\prime| \in n^{-5/4} U  \big\}} \le \mathbf{1}_{A_z^+(U)\cap A_w^+(U)} \,  \mathbf{1}_{|z-w| \le \log n/n^{5/4}}
		\] 
		on the event $\mathcal{G}\cap \mathcal{B}^c$, which in view of~\eqref{eq:def_of_X_n_pm} implies the upper bound.
	\end{proof}
	\subsection{Lower bound}
	We start working towards the lower bound in Proposition~\ref{prop:two_sided_inequality_for_X_n}. First, we will show that the event $A_z(U)$ given by~\eqref{eq:def_of_event_A_z_U_for_net_point} indicates there is a root of $f_n$ nearby due to a linear approximation around $z$. 
	\begin{claim}
		\label{claim:linear_approximation_predicts_a_root_in_R_z}
		For all $\beta>0$ and for each $K\ge 1$, $U\subset \bR_{\ge 0}$ the following holds for all large enough $n$. If $z\in \mathsf{N}$ is such that
		\[
		|f_n^\prime(z)| \ge \frac{n^{5/4}}{\log n} \, , \qquad z-\frac{f_n(z)}{f_n^\prime(z)} \in \mathsf{R}_z \, ,
		\]
		and the event $\mathcal{G}$ holds, then there exists some $\alpha$ with $f_n(\alpha)=0$ and
		\[
		\Big|\alpha - \Big(z-\frac{f_n(z)}{f_n^\prime(z)}\Big)\Big| \le n^{-5/4-2\beta} \log^4 n\, .
		\]
	\end{claim}
	\begin{proof}
		Consider the linear approximation of $f_n$ around $z$, that is the linear function $$L_z(w) = f_n(z) + (w-z) f_n^\prime(z)$$ and denote by $\zeta = z - f_n(z)/f_n^\prime(z)$ the unique root of $L_z$. Note that by the definition of $\mathsf{R}_z$ we have $|z-\zeta| \le 2n^{-5/4 -\beta}$. We will apply Rouch\'e's theorem to show that $f_n$ and $L_z$ have the same number of roots inside a disk of radius $r=n^{-5/4-2\beta} (\log n)^4$ around $\zeta$. First note that
		\[
		\min_{|w|=r} |L_z(w+\zeta)| = r|f_n^\prime(z)| \ge n^{{-2\beta}} (\log n)^3 \, .
		\]
		On the other hand, on the event $\mathcal{G}$ we have
		\begin{align*}
			\max_{|w|=r}|f_n(w+\zeta) - L_z(w+\zeta)| &\lesssim |w+\zeta - z|^2 n^{5/2} \log^2 n \\ & \lesssim \big(|w - z|^2 + r^2\big) n^{5/2} \log^2 n \\ &\lesssim \big(n^{-5/2 -4\beta} (\log n)^8 + n^{-5/2-2\beta} \big) n^{5/2} \log^2 n \lesssim n^{-2\beta} (\log n)^{2} \, .
		\end{align*}
		and by Rouch\'e we know that $f_n$ has a root $\alpha\in \bD(\zeta,r)$, as desired.
	\end{proof}
	Next, we show that the quadratic approximation provides us with the second root.
	\begin{lemma}
		\label{lemma:quadratic_approximation_predicts_second_root}
		For all $\beta>0$ and for each $K\ge 1$, $U\subset \bR_{\ge 0}$ the following holds for all large enough $n$. If $\mathcal{G}$ holds and $z\in \mathsf{N}$ is such that
		\[
		|f_n^\prime(z)| \ge \frac{n^{5/4}}{\log n} \, , \qquad z-\frac{f_n(z)}{f_n^\prime(z)} \in \mathsf{R}_z \, , \qquad \frac{2|f_n^\prime(z)|}{|f_n^{\prime\prime}(z)|} \le \frac{\log n}{n^{5/4}} \, ,
		\]
		then there exists some $\alpha^\prime$ with $f_n(\alpha^\prime)=0$ and
		\[
		\Big|\alpha^\prime - \Big(z+\frac{f_n(z)}{f_n^\prime(z)} - \frac{2 f_n^\prime(z)}{f_n^{\prime\prime}(z)}\Big)\Big| \le n^{-5/4-2\beta} \log^5 n\, .
		\]
	\end{lemma}
	\begin{proof}
		As before, we denote by $r= n^{-5/4-2\beta} (\log n)^4$. By Claim~\ref{claim:linear_approximation_predicts_a_root_in_R_z}, there exists $\alpha\in \bD\Big(z-\frac{f_n(z)}{f_n^\prime(z)},r\Big)$ so that $f_n(\alpha) = 0$. Set
		\[
		T_\alpha (w) = (w-\alpha) f_n^\prime(\alpha) + \frac{(w-\alpha)^2}{2}f_n^{\prime\prime}(\alpha) 
		\]
		and note that $T_\alpha$ has exactly two roots; one at $\alpha$ and another at 
		\[
		\zeta = \alpha - \frac{2f_n^\prime(\alpha)}{f_n^{\prime\prime}(\alpha)} \, .
		\]
		On the event $\mathcal{G}$, the assumptions of the lemma also implies that
		\begin{equation*}
			|\zeta - \alpha| = 2\frac{|f_n^\prime(\alpha)|}{|f_n^{\prime\prime}(\alpha)|} \ge 2\frac{|f_n^\prime(z)| + O\big(|z-\alpha| n^{5/2} \log^2 n\big)}{n^{5/2} \log^2 n} \gtrsim \frac{n^{-5/4}}{(\log n)^3} \, .
		\end{equation*}
		Since $z-\frac{f_n(z)}{f_n^\prime(z)} \in \mathsf{R}_z$, we also have $|z-\alpha| \le 2n^{-5/4-\beta}$. As $|f_n^{\prime\prime}(z)| \ge n^{5/2} /\log^2 n$, we get that
		\begin{align}
			\label{eq:proof_of_lemma:quadratic_approximation_predicts_second_root}
			\nonumber 
			\frac{f_n^\prime(\alpha)}{f_n^{\prime\prime}(\alpha)} &= \frac{f_n^\prime(z)}{f_n^{\prime\prime}(\alpha)} + (\alpha-z) \frac{f_n^{\prime\prime}(z)}{f_n^{\prime\prime}(\alpha)} + O\big(n^{-3/2-2\beta} (\log n)^4\big) \\ \nonumber &= \frac{1}{1+O\big(n^{-1/4-\beta} (\log n)^4\big)}\bigg(\frac{f_n^\prime(z)}{f_n^{\prime\prime}(z)} + \alpha- z\bigg) + O\big(n^{-3/2-2\beta} (\log n)^4\big) \\ &= \frac{f_n^\prime(z)}{f_n^{\prime\prime}(z)} + \alpha- z + O\big(n^{-3/2-\beta} (\log n)^5\big) \, .
		\end{align}
		Therefore
		\begin{equation*}
			|\zeta - \alpha| \stackrel{\eqref{eq:proof_of_lemma:quadratic_approximation_predicts_second_root}}{\lesssim} \Big| \frac{f_n^\prime(z)}{f_n^{\prime\prime}(z)} + \alpha- z \Big| \le \Big| \frac{f_n(z)}{f_n^{\prime}(z)} \Big| + \Big| \frac{f_n^\prime(z)}{f_n^{\prime\prime}(z)} \Big| + r \lesssim n^{-5/4} \log n \, .
		\end{equation*}
		We will apply Rouch\'e's theorem to show that $T_\alpha$ and $f_n$ both have exactly one root in the disk $\bD(\zeta, r)$. Indeed, first note that
		\begin{align*}
			\min_{|w|=r} \big|T_\alpha(\zeta +w)\big| &= \min_{|w|=r} \Big|\Big(w-\frac{2f_n^\prime(\alpha)}{f_n^{\prime\prime}(\alpha)}\Big)f_n^\prime(\alpha) + \Big(w-\frac{2f_n^\prime(\alpha)}{f_n^{\prime\prime}(\alpha)}\Big)^2 \frac{f_n^{\prime\prime}(\alpha)}{2} \Big| \\ &= \min_{|w|=r} \Big|-wf_n^\prime(\alpha) + w^2 \frac{f_n^{\prime\prime}(\alpha)}{2} \Big| \\ &\ge r\frac{n^{5/4}}{2\log n} - r^2 n^{5/2}\log^2n  \gtrsim n^{-2\beta} (\log n)^3 \, .
		\end{align*}
		On the other hand, we have the upper bound
		\begin{align*}
			\max_{|w|=r} \big|T_\alpha(\zeta +w) - f_n(\zeta+w)\big| &\le |\zeta + w - \alpha|^3 n^{7/2} \log^2 n \\ & \lesssim \Big(n^{-15/4} (\log n)^3 + r^3\Big) n^{7/2} \log^2 n \lesssim n^{-1/4} (\log n)^5 \, .
		\end{align*}
		Therefore, we conclude that for all $n$ large enough Rouch\'e applies, and hence $f_n$ has a root $\alpha^\prime$ in the disk $\bD(\zeta,r)$. It remains to observe that
		\begin{align*}
			\zeta &= \alpha - \frac{2f_n^\prime(\alpha)}{f_n^{\prime\prime}(\alpha)} \\ &=  z-\frac{f_n(z)}{f_n^\prime(z)} - \frac{2f_n^\prime(\alpha)}{f_n^{\prime\prime}(\alpha)} + O\big(n^{-5/4-2\beta} (\log n)^4\big) \\ & \stackrel{\eqref{eq:proof_of_lemma:quadratic_approximation_predicts_second_root}}{=} z-\frac{f_n(z)}{f_n^\prime(z)} - 2\Big(\frac{f_n^\prime(z)}{f_n^{\prime\prime}(z)} + \alpha- z\Big) + O\big(n^{-5/4-2\beta} (\log n)^4\big) \\ &= z +\frac{f_n(z)}{f_n^\prime(z)} - \frac{2f_n^\prime(z)}{f_n^{\prime\prime}(z)}  + O\big(n^{-5/4-2\beta} (\log n)^4\big) 
		\end{align*}
		and hence 
		\[
		\Big|\alpha^\prime - \Big(z+\frac{f_n(z)}{f_n^\prime(z)} - \frac{2 f_n^\prime(z)}{f_n^{\prime\prime}(z)}\Big)\Big| \lesssim n^{-5/4-2\beta} \log^4 n \, ,
		\]
		and we are done.
	\end{proof}
	Finally, the next lemma shows that on the typical event of having two close roots, all other roots are separated by an almost macroscopic amount. This fact will be useful in Section~\ref{sec:poisson_limit_for_sum_over_the_smooth_net}, when we actually need to compute the relevant moments.
	\begin{lemma}
		\label{lemma:existence_of_two_roots_implies_macroscopic_separation}
		For all $\beta>0$ and for each $K\ge 1$, $U\subset \bR_{\ge 0}$ the following holds for all large enough $n$. If $A_z(U) \cap \mathcal{G}$ holds for some $z\in \Omega_K$, then $f_n$ has exactly two roots in a ball of radius $\frac{1}{n\log^4 n}$ centered at $z$. The same statement continues to hold if $A_z(U)$ is replaced by $A_z^+(U)$ or by $A_z^-(U)$. 
	\end{lemma}
	\begin{proof}
		We will provide the proof only for the event $A_z(U)$, as the cases $A_z^+(U),A_z^-(U)$ are similar. Define the polynomial
		\[
		P_z(w) = f_n(z) + (w-z) f_n^\prime(z) + \frac{(w-z)^2}{2} f_n^{\prime\prime}(z)
		\]
		and set $t = \frac{1}{n (\log n)^4}$. We will apply Rouch\'e's theorem to show that $f_n$ and $P_z$ have the same number of roots in the disk $\bD(z,t)$. By Claim~\ref{claim:linear_approximation_predicts_a_root_in_R_z} and Lemma~\ref{lemma:quadratic_approximation_predicts_second_root}, we already know that $f_n$ has at least two roots inside $\bD(z,t)$, and since $P_z$ has at most two roots, the proof will follow. On the event $A_z(U)\cap \mathcal{G}$, we have
		\[
		|f_n(z)| = O(n^{-\beta} \log^2 n) \, , \qquad |f_n^\prime(z)| = O(n^{-5/4} \log^2 n) \, , \qquad |f_n^{\prime\prime}| \gtrsim n^{5/2} /\log n \, .
		\]
		From the above, the lower bound follows since
		\[
		\min_{|w| = t} |P_z(w+z)| \ge \frac{t^2}{2} |f_n^{\prime\prime}(z)| - |f_n(z)| - t|f_n^\prime(z)| \gtrsim \frac{n^{1/2}}{(\log n)^9} \, .
		\]
		The event $\mathcal{G}$ also gives an upper bound, namely
		\[
		\max_{|w| = t} |P_z(w+z) - f_n(w+z)|  \lesssim t^3 n^{7/2} \log^2 n \lesssim \frac{n^{1/2}}{(\log n)^{10}} \, .
		\]
		By Rouch\'e, the two displayed equations above imply that $P_z$ and $f_n$ have the same number of roots in $\bD(z,t)$, completing the proof of the lemma.
	\end{proof}
	We are finally ready to provide the lower bound in Proposition~\ref{prop:two_sided_inequality_for_X_n}.
	\begin{proof}[Proof of Proposition~\ref{prop:two_sided_inequality_for_X_n}: the lower bound]
		As in the proof of the upper bound, on the event $\mathcal{G} \cap \mathcal{B}^c$ we can write
		\[
		X_n(U) = \sum_{\{z,w\} \in \binom{\mathsf{N}_{\tt s}}{2}} \mathbf{1}_{\big\{\exists \alpha\in \mathsf{R}_z^\circ, \, \exists \alpha^\prime\in \mathsf{R}_w^\circ \, : \,  f_n(\alpha) = f_n(\alpha^\prime) = 0 \ \text{and} \ |\alpha - \alpha^\prime| \in n^{-5/4} U  \big\}} \, .
		\]
		On the event $\mathcal{G}$, we can apply Claim~\ref{claim:linear_approximation_predicts_a_root_in_R_z} and Lemma~\ref{lemma:quadratic_approximation_predicts_second_root} to conclude that for each $z\in \mathsf{N}_{\tt s}$ for which $A_z^-(U)$ holds, there exist two distinct roots $\alpha_1,\alpha_1^\prime$ with $\alpha_1\in \mathsf{R}_z$ and $|\alpha_1-\alpha_1^\prime| \in n^{-5/4} U$. Similarly, if $w\in \mathsf{N}_{\tt s} \setminus \{z\}$ is another point for which $A_w^-(U)$ holds and $|z-w| \le \log n/ n^{5/4}$, then $f_n$ has two distinct roots $\alpha_2,\alpha_2^\prime$ where $\alpha_2\in \mathsf{R}_{w}$ and $|\alpha_2-\alpha_2^\prime| \in n^{-5/4} U$. We must have $\alpha_1 \not= \alpha_2$ since $\mathsf{R}_z\cap \mathsf{R}_w = \emptyset$. Furthermore, by Lemma~\ref{lemma:existence_of_two_roots_implies_macroscopic_separation} we have that the set $\{\alpha_1,\alpha_1^\prime,\alpha_2,\alpha_2^\prime\}$ contains exactly two elements, and we conclude that $\alpha_1 = \alpha_2^\prime$ and $\alpha_2 = \alpha_1^\prime$. This shows that for any pair $\{z,w\} \in \binom{\mathsf{N}_{\tt s}}{2}$ we have
		\[
		\mathbf{1}_{A_z^-(U)\cap A_w^-(U)} \,  \mathbf{1}_{|z-w| \le \log n/n^{5/4}} \le \mathbf{1}_{\big\{\exists \alpha\in \mathsf{R}_z^\circ, \, \exists \alpha^\prime\in \mathsf{R}_w^\circ \, : \,  f_n(\alpha) = f_n(\alpha^\prime) = 0 \ \text{and} \ |\alpha - \alpha^\prime| \in n^{-5/4} U  \big\}} 
		\]   
		on the event $\mathcal{G}\cap \mathcal{B}^c$, and the desired lower bound follows.  
	\end{proof}
	\subsection{The ``bad" event $\mathcal{B}$ is rare}
	Now that we concluded the proof of Proposition~\ref{prop:two_sided_inequality_for_X_n}, we use the remainder of the section to prove Lemma~\ref{lemma:in_annulus_good_events_are_likely}. Recall that $\mathsf{N}_{\tt ns} = \mathsf{N} \setminus \mathsf{N}_{\tt s}$ and further decompose as 
	\begin{align}
		\label{eq:further_decomposition_of_non_smooth_in_the_edge}
		\nonumber
		\mathsf{N}_{\tt ns}^{(1)} &= \big\{z\in \mathsf{N}_{\tt ns} \, : \, \arg(z) \in [n^{-1},\pi - n^{-1}]\big\} \, , \\ \mathsf{N}_{\tt ns}^{(2)} &= \big\{z\in \mathsf{N}_{\tt ns} \, : \, \arg(z) \in [n^{-1-\tau/4},n^{-1}] \cup [\pi - n^{-1},\pi - n^{-1-\tau/4}]\big\} \, , \\ \nonumber \mathsf{N}_{\tt ns}^{(3)} &= \big\{z\in \mathsf{N}_{\tt ns} \, : \, \arg(z) \in [0,n^{-1-\tau/4}] \cup [\pi - n^{-1-\tau/4},\pi]\big\} \, . 
	\end{align}
	We note that as $n\to \infty$, we have
	\begin{equation}
		\label{eq:number_of_smooth_points_in_the_net_annulus}
		|\mathsf{N}_{\tt s}| = M_1 M_2 \big(1+o(1)\big) = 16 K n^{3/2 + 2\beta} \big(1+o(1)\big) \, ,
	\end{equation}
	and furthermore
	\begin{align}
		\label{eq:further_decomposition_of_non_smooth_in_the_edge_asymptotics_and_bounds}
		\nonumber
		|\mathsf{N}_{\tt ns}^{(1)}| &\lesssim M_1 M_2 n^{7\tau-1} \lesssim K n^{1/2 + 2\beta + 7\tau} \, , \\ |\mathsf{N}_{\tt ns}^{(2)}| & \lesssim M_1 M_2 n^{-1} \lesssim K n^{1/2 + 2\beta}  \, , \\ \nonumber |\mathsf{N}_{\tt ns}^{(3)}| &\lesssim M_1 M_2 n^{-1-\tau/4} \lesssim K n^{1/2 + 2\beta-\tau/4} \, . 
	\end{align}
	Towards the proof of Lemma~\ref{lemma:in_annulus_good_events_are_likely}, we start by showing that with high probability there is no pair of close roots in the non-smooth part of the annulus $\bigcup_{z\in \mathsf{N}_{\tt ns}} \mathsf{R}_z$. This is shown via a net argument similar to the one presented in Section~\ref{sec:reducing_to_the_annulus}.
	\begin{claim}
		\label{claim:eliminating_non_smooth_points_in_annulus}
		We have $\displaystyle \lim_{n\to \infty} \bP\big[\mathcal{G}\cap \mathcal{B}_1^c \cap \mathcal{B}_2\big] = 0$. 
	\end{claim}
	\begin{proof}
		By the union bound, we have
		\begin{equation*}
			\bP\big[\mathcal{G}\cap \mathcal{B}_1^c \cap \mathcal{B}_2\big] \le \sum_{z\in \mathsf{N}_{\tt ns}} \bP\Big[\big\{\exists \alpha \in \mathsf{R}_z  \, ,  \ \text{and } \ \exists \alpha^\prime \in \bD\Big(\alpha,\frac{\log n}{n^{5/4}} \Big) \setminus\{\alpha \}\, : \, f_n(\alpha) = f_n(\alpha^\prime) = 0 \big\} \cap \mathcal{G} \cap \mathcal{B}_1^c \Big] \, . 
		\end{equation*}
		Furthermore, by combining Claim~\ref{claim:close_roots_implies_ratio_of_derivative_and_second_derivative_to_be_small} with Claim~\ref{claim:root_with_small_derivative_implies_smallness_in_net_point} we get the inclusion
		\begin{multline*}
			\big\{\exists \alpha \in \mathsf{R}_z  \, ,  \ \text{and } \ \exists \alpha^\prime \in \bD\Big(\alpha,\frac{\log n}{n^{5/4}} \Big) \setminus\{\alpha \}\, : \, f_n(\alpha) = f_n(\alpha^\prime) = 0 \big\} \cap \mathcal{G} \cap \mathcal{B}_1^c  \\ \subset \big\{|f_n(z)| \le 2 n^{-\beta} (\log n)^3 \, , \ |f_n^{\prime}(z)| \le n^{5/4} (\log n)^3 \big\} \, ,
		\end{multline*}
		and altogether
		\begin{equation}
			\label{eq:proof_of_claim:eliminating_non_smooth_points_in_annulus_after_union_bound}
			\bP\big[\mathcal{G}\cap \mathcal{B}_1^c \cap \mathcal{B}_2\big] \le \sum_{z\in \mathsf{N}_{\tt ns}} \bP\big[|f_n(z)| \le 2 n^{-\beta} (\log n)^3 \, , \ |f_n^{\prime}(z)| \le n^{5/4} (\log n)^3\big] \, .
		\end{equation}
		To bound the sum on the right-hand side of~\eqref{eq:proof_of_claim:eliminating_non_smooth_points_in_annulus_after_union_bound}, we split the sum into $\mathsf{N}_{\tt ns}^{(1)}$, $\mathsf{N}_{\tt ns}^{(2)}$ and $\mathsf{N}_{\tt ns}^{(3)}$ as given by~\eqref{eq:further_decomposition_of_non_smooth_in_the_edge} and bound separately on each part. 
		
		\noindent
		\underline{The sum over $\mathsf{N}_{\tt ns}^{(1)}$:}
		By Claim~\ref{claim:small_ball_away_from_real_axis}, for each $z\in \mathsf{N}_{\tt ns}^{(1)}$ we have the bound
		\begin{multline*}
			\bP\big[|f_n(z)| \le 2 n^{-\beta} \log^3 n \, , \ |f_n^{\prime}(z)| \le n^{5/4} \log^3 n \big] \\ \le \bP\big[|f_n(z)| \le \log n \, , \ |f_n^{\prime}(z)| \le n^{5/4} \log^3 n \big] \lesssim n^{-3/2} \log^{8} n \, .
		\end{multline*}
		Since $|\mathsf{N}_{\tt ns}^{(1)}| \lesssim K n^{1/2 + 2\beta + 7\tau}$ we conclude that 
		\begin{equation}
			\label{eq:proof_of_lemma_Z_n_is_close_of_Z_n_smooth_first_sum}
			\sum_{z\in \mathsf{N}_{\tt ns}^{(1)}} \bP\big[|f_n(z)| \le 2 n^{-\beta} \log^3 n \, , \ |f_n^{\prime}(z)| \le n^{5/4} \log^3 n \big] \lesssim n^{-1 + 2\beta +7\tau} \log^8 n \xrightarrow{n\to \infty} 0 \, .
		\end{equation}
		
		\noindent
		\underline{The sum over $\mathsf{N}_{\tt ns}^{(2)}$:}
		By Claim~\ref{claim:small_ball_points_near_the_real_axis} we see that for all $z\in \mathsf{N}_{\tt ns}^{(1)}$ we have
		\[
		\bP\big[|f_n(z)| \le 2 n^{-\beta} \log^3 n \, , \ |f_n^{\prime}(z)| \le n^{5/4} \log^3 n \big] \lesssim n^{-3/2 + 8\tau} \log^6 n \, .
		\]
		Combining with the fact that $|\mathsf{N}_{\tt ns}^{(2)}| \lesssim K n^{1/2 + 2\beta}$ we get
		\begin{equation}
			\label{eq:proof_of_lemma_Z_n_is_close_of_Z_n_smooth_second_sum}
			\sum_{z\in \mathsf{N}_{\tt ns}^{(2)}} \bP\big[|f_n(z)| \le 2 n^{-\beta} \log^3 n \, , \ |f_n^{\prime}(z)| \le n^{5/4} \log^3 n \big] \lesssim n^{-1 + 2\beta + 8\tau} \log^6 n \xrightarrow{n\to \infty} 0 \, .
		\end{equation}
		
		\noindent
		\underline{The sum over $\mathsf{N}_{\tt ns}^{(3)}$:} 
		For the range $z\in \mathsf{N}_{\tt ns}^{(3)}$ of point very close to the real axis, we apply Claim~\ref{claim:small_ball_real_points} to conclude that
		\begin{equation}
			\label{eq:proof_of_lemma_Z_n_is_close_of_Z_n_smooth_third_sum}
			\sum_{z\in \mathsf{N}_{\tt ns}^{(3)}} \bP\big[|f_n(z)| \le 2 n^{-\beta} \log^3 n \, , \ |f_n^{\prime}(z)| \le n^{5/4} \log^3 n \big] \lesssim |\mathsf{N}_{\tt ns}^{(3)}| n^{-1/2} \lesssim n^{-\tau/4 + 2\beta} \xrightarrow{n\to \infty} 0 \, , 
		\end{equation}
		for all $\beta>0$ small enough.
		Combining~\eqref{eq:proof_of_lemma_Z_n_is_close_of_Z_n_smooth_first_sum}, \eqref{eq:proof_of_lemma_Z_n_is_close_of_Z_n_smooth_second_sum}, and~\eqref{eq:proof_of_lemma_Z_n_is_close_of_Z_n_smooth_third_sum} we see that the sum on the right-hand side of~\eqref{eq:proof_of_claim:eliminating_non_smooth_points_in_annulus_after_union_bound} tends to zero as $n\to \infty$, which is what we wanted to prove.
	\end{proof}
	Next, we show how that we can eliminate roots which are near the boundaries of smooth rectangles. 
	\begin{claim}
		\label{claim:eliminating_roots_near_boundaries}
		We have $\displaystyle \lim_{n\to \infty} \bP\big[\mathcal{G} \cap \mathcal{B}_3\big] = 0$.
	\end{claim}
	\begin{proof}
		By the union bound we have
		\[
		\bP\big[\mathcal{G} \cap \mathcal{B}_3\big] \le \sum_{z\in \mathsf{N}_{\tt s}} \bP\Big[ \exists \alpha \in \mathsf{R}_z\setminus \mathsf{R}_z^\circ \, , \ \text{and } \ \exists \alpha^\prime \in \bD\Big(\alpha,\frac{\log n}{n^{5/4}} \Big) \setminus\{\alpha \}\, : \, f_n(\alpha) = f_n(\alpha^\prime) = 0 \, , \mathcal{G} \Big] \, .
		\]
		For $z\in \mathsf{N}_{\tt s}$ we cover the set $\mathsf{R}_z \setminus \mathsf{R}_z^{\circ}$ with disks of radius $\widetilde \delta = \delta n^{-\beta/2}$ so that the number of disks is $\lesssim n^{\beta/2}$ (this is possible, as the area of $\mathsf{R}_z \setminus \mathsf{R}_z^{\circ}$ is $\simeq \delta^2 n^{-\beta/2}$). Denote for a moment by $\widetilde z$ the center of an arbitrary disk $\bD(\widetilde z, \widetilde \delta)$ in the covering. By combining Claim~\ref{claim:close_roots_implies_ratio_of_derivative_and_second_derivative_to_be_small} with Claim~\ref{claim:root_with_small_derivative_implies_smallness_in_net_point}, we conclude that on the event $\mathcal{G}$, if there exists $\alpha\in \bD(\widetilde z, \widetilde \delta)$ and $\alpha\prime  \in \bD\big(\alpha,\frac{\log n}{n^{5/4}} \big) \setminus\{\alpha \}$ such that $f_n(\alpha) = f_n(\alpha^\prime) = 0$ then we must have 
		\[
		|f_n(\widetilde z)| \le 2\widetilde \delta |f_n^\prime(\widetilde z)| \, \, \qquad \text{and } \qquad |f_n^\prime(\widetilde z)| \le 2n^{5/4} (\log n)^3 \, .
		\]
		Furthermore, as $z\in \mathsf{N}_{\tt s}$ we know that any $\widetilde z$ in the new $\widetilde \delta$-net is $n^{5\tau}$-smooth, in the sense of Definition~\ref{def:smooth_angle}. Therefore, by Lemma~\ref{lemma:small_ball_probability_smooth_points_bound_single_point} we have
		\begin{multline*}
			\bP\Big[|f_n(\widetilde z)| \le 2\widetilde \delta |f_n^\prime(\widetilde z)| \, , |f_n^\prime(\widetilde z)| \le 2n^{5/4} (\log n)^3 \Big] \\ \le \bP\Big[|f_n(\widetilde z)| \le 4n^{-3\beta/2} (\log n)^3 \, , |f_n^\prime(\widetilde z)| \le 2n^{5/4} (\log n)^3\Big] \lesssim n^{-3/2 -3 \beta} (\log n)^8 \, .
		\end{multline*}
		Hence, a union bound implies that
		\begin{multline*}
			\bP\Big[ \exists \alpha \in \mathsf{R}_z\setminus \mathsf{R}_z^\circ \, , \ \text{and } \ \exists \alpha^\prime \in \bD\Big(\alpha,\frac{\log n}{n^{5/4}} \Big) \setminus\{\alpha \}\, : \, f_n(\alpha) = f_n(\alpha^\prime) = 0 \, , \mathcal{G} \Big] \\ \lesssim n^{\beta/2} n^{-3/2-3\beta} (\log n )^8= n^{-3/2 - 2\beta -\beta/2} (\log n )^8\, .  
		\end{multline*}
		Since $|\mathsf{N}_{\tt s}| \lesssim Kn^{3/2+2\beta}$, we conclude that
		\[
		\bP \big[\mathcal{G} \cap \mathcal{B}_3\big] \lesssim n^{-\beta/2} (\log n )^8
		\]
		and we are done.
	\end{proof}
	Putting everything together, we finally give the proof of Lemma~\ref{lemma:in_annulus_good_events_are_likely}.
	\begin{proof}[Proof of Lemma~\ref{lemma:in_annulus_good_events_are_likely}]
		By Lemma~\ref{lemma:control_on_maximum_of_polynomial_and_derivatives} we know that $\bP\big[\mathcal{G}^c\big] = o(1)$ as $n\to \infty$, and by Lemma~\ref{lemma:with_high_probability_no_roots_with_too_small_derivative} we also know that
		\[
		\lim_{n\to \infty} \bP\big[\mathcal{G}\cap \mathcal{B}_1^c\big] = 0 \, .
		\]
		Combining the above facts with Claim~\ref{claim:eliminating_non_smooth_points_in_annulus} and Claim~\ref{claim:eliminating_roots_near_boundaries} completes the lemma.
	\end{proof}
	\section{Poisson limit for the sum over the net}
	\label{sec:poisson_limit_for_sum_over_the_smooth_net}
	In this section we prove Proposition~\ref{prop:poisson_moments_for_X_n_pm}. We first recall some relevant notations from the previous section and give a general outline. Recall the ``good" event
	\begin{equation*}
		\mathcal{G} = \bigg\{ \max_{0\le j \le 3} \max_{z\in \Omega_{2K}} \frac{|f_n^{(j)}(z)|}{n^{j+1/2}} \le \log^2 n \bigg\} \, ,
	\end{equation*}
	By Lemma~\ref{lemma:control_on_maximum_of_polynomial_and_derivatives}, we have $\bP\big[\mathcal{G}^c\big] \le \exp\big(-c\log^2n\big)$. Furthermore, for $z\in \mathsf{N}_{\tt s}$ and $U\subset \bR_{\ge 0}$ we recall the events 
	\begin{align}
		\label{eq:def_of_A_z_U_with_pm_section_poisson_limit}
		\nonumber
		A_z(U) &= \bigg\{\frac{f_n(z)}{f_n^\prime(z)} - z \in \mathsf{R}_z \, , \ \frac{2|f_n^\prime(z)|}{|f_n^{\prime\prime}(z)|} \in n^{-5/4} U \, , \ |f_n^\prime(z)| \ge \frac{n^{5/4}}{\log n} \bigg\} \, , \\ A_z^+(U) &= \bigg\{\frac{f_n(z)}{f_n^\prime(z)} - z \in \mathsf{R}_z^{\sharp} \, , \ \frac{2|f_n^\prime(z)|}{|f_n^{\prime\prime}(z)|} \in n^{-5/4} U^+ \, , \ |f_n^\prime(z)| \ge \frac{n^{5/4}}{\log n} \bigg\} \, , \\ \nonumber A_z^-(U) &= \bigg\{\frac{f_n(z)}{f_n^\prime(z)} - z \in \mathsf{R}_z^{\circ} \, , \ \frac{2|f_n^\prime(z)|}{|f_n^{\prime\prime}(z)|} \in n^{-5/4} U^- \, , \ |f_n^\prime(z)| \ge \frac{n^{5/4}}{\log n} \bigg\} \, ,
	\end{align}
	where $U^\pm$ are given by~\eqref{eq:blow_up_and_blow_down_of_U}. For a reminder on what  the polar rectangles $\mathsf{R}_z,\mathsf{R}_z^\sharp,\mathsf{R}_z^\circ$ are, we refer the reader to Definition~\ref{definition:net_in_the_main_annulus}. The events above led us to consider
	\begin{equation*}
		\label{eq:def_of_X_n_pm-reminder}
		X_n^\pm (U) = \sum_{\{z,w\} \in \binom{\mathsf{N}_{\tt s}}{2}} \mathbf{1}_{A_z^\pm(U)\cap A_w^\pm(U)} \,  \mathbf{1}_{|z-w| \le \log n/n^{5/4}} \, . 
	\end{equation*}
	The goal of this section is to prove matching upper and lower bounds for the factorial moments of $X_n^+(U)$ and $X_n^-(U)$, respectively. Below we will show that the limiting moments as $n\to \infty$ match those of a Poisson random variable with parameter $\la_{K,U} = \mathfrak{c}_\ast(K) \int_{U} t^{3} {\rm d}t$, where 
	\begin{equation}
		\label{eq:intensity_constant_for_K_annulus}
		\mathfrak{c}_\ast(K) = \frac{1}{4}\int_{-K}^{K} \mathfrak{F}(x) \, {\rm d}x \, ,
	\end{equation}
	and $\mathfrak{F}$ is given by~\eqref{eq:def_of_limiting_intensity_function_F} below. We will also use this opportunity to identify the limiting intensity $\mathfrak{c}_\ast t^3 {\rm d}t$ for the Poisson process described in our main result Theorem~\ref{thm:poisson_limit_for_close_roots}. Indeed, since $X_n(U)$ given by~\eqref{eq:def_of_X_n_U} is asymptotically a Poisson random variable with parameter $\la_{K,U}$, and since $X_n(U)$ and the set of roots at distance $n^{-5/4}U$ do not differ with high probability for large values of $K>0$ (see the proof of Theorem~\ref{thm:poisson_limit_for_close_roots} in Section~\ref{sec:reducing_main_result_to_a_net_argument}), we set
	\begin{equation}
		\label{eq:limiting_constant_for_intensity}
		\mathfrak{c}_\ast = \lim_{K\to \infty}\mathfrak{c}_\ast(K) = \frac{1}{4}\int_{-\infty}^{\infty} \mathfrak{F}(x) \, {\rm d}x \, .
	\end{equation}
	Indeed, by Claim~\ref{claim:integral_properties_of_intensity_mathfrak_F} below we know in particular that the limit in~\eqref{eq:limiting_constant_for_intensity} exists.
	
	\subsection{Factoring probabilities over spread tuples}
	We denote by $\bP_{{\sf G}}$ the probability distribution of the random polynomial $f_n$, in the case where the random coefficients are i.i.d.\ standard Gaussian random variables. Recall Definition~\ref{def:spread_tuples} of $\gamma$-spread tuples. We will need to know that for macroscopically spread tuples, the probabilities of seeing close roots approximately factors, and is the same as in the case of Gaussian coefficients. The exact formulation is given by the following lemma, the proof of which we postpone to Section~\ref{sec:gaussian_comparison_for_tuples}. 
	\begin{lemma}
		\label{lemma:for_spread_tuples_probabilities_of_smooth_points_factor}
		For $m\ge 1$ fixed let $\mathbf{z} = (z_1,\ldots,z_m) \in (\mathsf{N}_{\tt s})^m$ be a $n^\beta$-spread tuple. Then, as $n\to \infty$, we have
		\[
		\bP\big[A_{z_1}(U) \cap \ldots \cap A_{z_m}(U)\big] = \big(1+o(1)\big) \prod_{j=1}^{m} \bP_{{\sf G}}\big[A_{z_{j}}(U)\big]\, ,
		\]
		uniformly in $K>0$ and $U\subset \bR_{\ge 0}$ in a compact set. Furthermore, the above remains true if the events on the left-hand side are replaced by $A_z^+(U)$ or by $A_z^-(U)$.
	\end{lemma}
	We will also need the following bound, which shows that the probabilities for almost-macroscopically spread tuples is comparable to those obtained from a Gaussian computations, up to a tolerable poly-logarithmic lost.
	\begin{lemma}
		\label{lemma:for_almost_spread_tuples_probabilities_are_bounded_by_what_you_want}
		For $m\ge 1$ fixed let $\mathbf{z} = (z_1,\ldots,z_m) \in (\mathsf{N}_{\tt s})^m$ be a $\frac{1}{(\log n)^5}$-spread tuple. We have
		\[
		\bP\big[A_{z_1}^+(U) \cap \ldots \cap A_{z_m}^+(U)\big] \lesssim_{m,K,U} \Big(\frac{1}{n^{3/2+2\beta}}\Big)^m \log^{O_m(1)} n \, .
		\]
	\end{lemma} 
	\begin{proof}
		Recall that on the event $A_z^+(U) \cap \mathcal{G}$ we have that
		\[
		|f_n(z)| \le 2n^{-\beta} \log^2 n \, , \qquad |f_n^\prime(z)| \le n^{5/4} \log^2 n \, .
		\]
		Therefore, we have the bound
		\begin{equation}
			\label{eq:proof_of_lemma:for_almost_spread_tuples_probabilities_are_bounded_by_what_you_want}
			\bP\big[A_{z_1}^+(U) \cap \ldots \cap A_{z_m}^+(U)\big] \le \bP\Big[\bigcap_{j=1}^{m} \big\{|f_n(z_j)| \le 2n^{-\beta} \log^2 n \, , \, |f_n^\prime(z_j)| \le n^{5/4} \log^2 n \big\}\Big] + \bP\big[\mathcal{G}^c\big] \, .
		\end{equation}
		Clearly, the 4-dimensional Lebesgue measure of the set
		\[
		\bD\big(0,n^{-1/2-\beta} \log^2 n\big) \times \bD\big(0,n^{-1/4} \log^2 n\big)
		\]
		is $\lesssim n^{-3/2-2\beta} \log^{O(1)} n$. Since we can cover the above set by disks of radius $n^{-2}$ so that the sum of their measure is also $\lesssim n^{-3/2-2\beta} \log^{O(1)} n$, and since the tuple $\mathbf{z}=(z_1,\ldots,z_m)$ is $\frac{1}{(\log n)^5}$-spread and $n^{7\tau}$-smooth, we can apply Proposition~\ref{prop:small_ball_estimate_for_tuples_of_smooth} and get that
		\[
		\bP\Big[\bigcap_{j=1}^{m} \big\{|f_n(z_j)| \le 2n^{-\beta} \log^2 n \, , \, |f_n^\prime(z_j)| \le n^{5/4} \log^2 n \big\}\Big] \lesssim \Big(\frac{1}{n^{3/2+2\beta}}\Big)^m \log^{O(1)} n \, .
		\]
		Since $\bP\big[\mathcal{G}^c\big] \le \exp\big(-c\log^2 n\big)$, in view of~\eqref{eq:proof_of_lemma:for_almost_spread_tuples_probabilities_are_bounded_by_what_you_want} we are done. 
	\end{proof}
	We will also need another small-ball estimate, showing that on the event $A_z(U)\cap \mathcal{G}$, it is likely that at least one of $A_w(U)$ occur, where $|z-w| \in n^{-5/4} U$. The following lemma will be used in the proof of the lower bound of Proposition~\ref{prop:poisson_moments_for_X_n_pm}. For $z\in \mathsf{N}_{\tt s}$ consider the event
	\begin{equation}
		\label{eq:def_of_event_upsilon}
		\Upsilon_z = A_z^+(U) \cap  \Bigg\{ z+\frac{f_n(z)}{f_n^\prime(z)} - \frac{2 f_n^\prime(z)}{f_n^{\prime\prime}(z)} \in \bigcup_{\substack{w\in \mathsf{N}_{\tt s} \\ 0<|z-w| \le \frac{\log n}{n^{5/4}}}}( \mathsf{R}_w \setminus \mathsf{R}_w^\circ) \bigg\} \, ,
	\end{equation}
	where we recall that $z+\frac{f_n(z)}{f_n^\prime(z)} - \frac{2 f_n^\prime(z)}{f_n^{\prime\prime}(z)}$ is the ``quadratic prediction" for the location of the close root, see Lemma~\ref{lemma:quadratic_approximation_predicts_second_root} for the precise statement. 
	\begin{lemma}
		\label{lemma:bound_on_probability_of_tuple_with_upsilon}
		For $m\ge 1$ fixed let $\mathbf{z} = (z_1,\ldots,z_m) \in (\mathsf{N}_{\tt s})^m$ be a $n^\beta$-spread tuple. We have
		\[
		\bP\big[\Upsilon_{z_1} \cap A_{z_2}^+(U) \cap \ldots \cap A_{z_m}^+(U)\big] \lesssim_{m,K,U} \Big(\frac{1}{n^{3/2+2\beta}}\Big)^m \cdot n^{-\beta/10} \, . 
		\]
	\end{lemma}
	\begin{proof}
		We start by bounding the Lebesgue measure of the set
		\begin{multline*}
			\label{eq:def_of_set_relevant_for_upsilon}
			\Xi_{n} = \Big\{(X,X^\prime,X^{\prime\prime}) \in \bC^3 \, : \, |X| \le \frac{\log^2 n}{n^{1/2+\beta}} \, , \\  \frac{1}{n^{1/4}\log n } \le |X^\prime| \le \frac{\log^2n}{n^{1/4}} \, , \ |X^{\prime\prime}|\le \log^2 n\, ,  \ \frac{X}{nX^\prime} - \frac{2 X^\prime}{nX^{\prime\prime}} \in \bigcup_{\substack{w\in \mathsf{N}_{\tt s} \\  |z-w| \in n^{-5/4} U }}( \mathsf{R}_w \setminus \mathsf{R}_w^\circ - z)\Big\} \, .
		\end{multline*}
		Indeed, we will show that
		\begin{equation}
			\label{eq:bound_on_measure_of_xi_n}
			m(\Xi_n) \lesssim n^{-3/2-2\beta - \beta/10} \, . 
		\end{equation}
		The set $\Xi_n$ is relevant to our computation since for all $z\in \mathsf{N}_{\tt}$ we have the
		\begin{equation}
			\label{eq:why_is_the_set_xi_relevant}
			\Upsilon_z \cap \mathcal{G} \subset \Big\{ \Big(\frac{f_n(z)}{n^{1/2}},\frac{f_n^\prime(z)}{n^{3/2}},\frac{f_n^{\prime\prime}(z)}{n^{5/2}} \Big) \in \Xi_n  \Big\} \, ,
		\end{equation}
		where $\Upsilon_z$ is given by~\eqref{eq:def_of_event_upsilon}. After the bound~\eqref{eq:bound_on_measure_of_xi_n} is established, we will apply Proposition~\ref{prop:small_ball_estimate_for_tuples_of_smooth} and argue as in the proof of Lemma~\ref{lemma:for_spread_tuples_probabilities_of_smooth_points_factor}. 
		We first note that on $\Xi_n$ we have
		\[
		\Big|\frac{X}{nX^\prime}\Big| \le \frac{(\log n)^3}{n^{5/4 + \beta}} \, ,
		\]
		and since $\mathsf{R}_z^\circ \subset \mathsf{R}_z^\sharp$ with $m\big(\mathsf{R}_z^\sharp \setminus \mathsf{R}_z^\circ\big) \lesssim n^{-5/4 - 3\beta/2}$ we get by the triangle inequality that
		\begin{equation*}
			\Xi_n \subset \Big\{|X| \le \frac{\log^2 n}{n^{1/2+\beta}} \, , \   \frac{1}{n^{1/4}\log n } \le |X^\prime| \le \frac{\log^2n}{n^{1/4}} \, , \ |X^{\prime\prime}| \le \log^2n \, , \ \frac{2 X^\prime}{nX^{\prime\prime}} \in \bigcup_{\substack{w\in \mathsf{N}_{\tt s} \\  |z-w| \in n^{-5/4} U }}( \mathsf{R}_w \setminus \mathsf{R}_w^\sharp - z)\Big\} \, .
		\end{equation*}
		Furthermore, we have that
		\[
		\frac{n^{5/4}}{\log^2 n} \le \frac{2n}{|X^\prime|} \le 2n^{5/4} \log n \, .
		\]
		Together with the simple bound
		\[
		m( \mathsf{R}_w \setminus \mathsf{R}_w^\sharp ) \lesssim n^{-5/4 - 5\beta/2} \, ,
		\]
		we get the following constraint on the variable $X^{\prime\prime}\in \bC$: it must lie in the set
		\begin{equation}
			\label{eq:set_containing_X_prime+prime}
			\Bigg(\bigg[ (2n^{5/4} \log^2 n )\bigcup_{\substack{w\in \mathsf{N}_{\tt s} \\  |z-w| \in n^{-5/4} U }}( \mathsf{R}_w \setminus \mathsf{R}_w^\sharp - z) \bigg]\setminus \bD\Big(0,\frac{1}{\log^2 n}\Big)\Bigg)^{-1}\, .
		\end{equation}
		Here, by the inverse of a set $A\subset \bC$ we just mean $A^{-1} = \{z^{-1} \, : \, z\in A\}$. The set~\eqref{eq:set_containing_X_prime+prime} might look complicated, but we really just need to estimate its measure, which is fairly straightforward. By the previous estimates we have
		\[
		m\bigg((2n^{5/4} \log^2 n )\bigcup_{\substack{w\in \mathsf{N}_{\tt s} \\  |z-w| \in n^{-5/4} U }}( \mathsf{R}_w \setminus \mathsf{R}_w^\sharp - z)  \bigg) \lesssim n^{-\beta/4} \log^{O(1)} n \, .
		\]
		Moreover, for a bounded closet set $A\subset \bC$ such that $0\not\in A$ we can change variables and observe that
		\[
		m(A^{-1}) = \int_{A^{-1}} \, {\rm d}m \lesssim \int_{A} \frac{{\rm d}m(z)}{|z|^2} \lesssim \frac{m(A)}{\text{dist}(0,A)^2}  \, .
		\]
		We conclude that the measure of the set~\eqref{eq:set_containing_X_prime+prime} is $\lesssim n^{-\beta/4} \log^{O(1)}n$, and by bounding the constraints on $X,X^\prime$ trivially we obtain
		\[
		m(\Xi_n) \lesssim m\Big(\bD\Big(0,\frac{\log^2 n}{n^{1/2+\beta}}\Big)\Big)\times m\Big(\bD\Big(0,\frac{\log^2n}{n^{1/4}}\Big)\Big)\times n^{-\beta/5} \lesssim  n^{-1-2\beta} n^{-1/2} n^{-\beta/6} \, ,
		\]
		which yields~\eqref{eq:bound_on_measure_of_xi_n} for $n$ large enough. Since $\Xi_n$ is a Lipschitz domain, we can cover the set $\Xi_n$ with balls of radius $n^{-2}$ so that the sum of their measure is $\lesssim n^{-3/2-2\beta - \beta/10}$. By combining Proposition~\ref{prop:small_ball_estimate_for_tuples_of_smooth} with the observation~\eqref{eq:why_is_the_set_xi_relevant} yields that
		\begin{align*}
			\bP&\big[\Upsilon_{z_1} \cap A_{z_2}^+(U) \cap \ldots \cap A_{z_m}^+(U)\big] \\ & \le \bP\bigg[ \Big\{ \Big(\frac{f_n(z_1)}{n^{1/2}},\frac{f_n^\prime(z_1)}{n^{3/2}},\frac{f_n^{\prime\prime}(z_1)}{n^{5/2}} \Big) \in \Xi_n  \Big\}\cap \bigcap_{j=2}^{m} \big\{|f_n(z_j)| \le 2n^{-\beta} \log^2 n \, , \, |f_n^\prime(z_j)| \le n^{5/4} \log^2 n \big\}\bigg] +\bP\big[\mathcal{G}^c\big] \\ &\lesssim \Big(\frac{1}{n^{3/2+2\beta}}\Big)^m \cdot n^{-\beta/10} + \exp\big(-c\log^2 n\big)
		\end{align*}
		and we are done. 
	\end{proof}
	\subsection{Convergence of Riemann sums}
	For $x\in \bR$ and $j\in \{0,\ldots,4\}$ we set 
	\begin{equation}
		\label{eq:def_of_a_j}
		a_j(x) = \int_{0}^{1} t^j e^{2xt} {\rm d}t = \begin{cases}
			\frac{e^{2x}-1}{2x} & j=0 \, , \\ \frac{e^{2 x} (2 x -1)+1}{4 x^2}  & j=1 \, , \\ \frac{e^{2 x} (1 + 2x(x-1)) -1}{4 x^3} & j=2 \, , \\  \frac{e^{2x}(4x^3 - 6x^2 + 6x -3) +3}{8 x^4}  & j=3 \, , \\ \frac{e^{2x}(2x^4 -4x^3+6x^2 -6x + 3) -3}{4 x^5} & j=4 \, .
		\end{cases}
	\end{equation}
	We will further denote by
	\begin{equation}
		\label{eq:def_of_Delta_functions}
		\Delta_1(x) = a_1(x) a_3(x) - a_2(x)^2 \qquad \text{and} \qquad \Delta_2(x) = a_2(x)a_4(x) -a_3(x)^2 \, ,
	\end{equation}
	and note that both functions as non-negative on the real line, as they can be written as a determinant of a Gram matrix. Finally, we set
	\begin{equation}
		\label{eq:def_of_eta_function}
		\eta(x) = a_0(x)a_2(x)a_4(x)-a_0(x)a_3(x)^2-a_1(x)^2a_4(x)+2a_1(x)a_2(x)a_3(x)-a_2(x)^3 \, ,
	\end{equation}
	and apply some algebra with~\eqref{eq:def_of_a_j} to write
	\[
	\eta(x) = \frac{e^{3x}\big(2 x^3 \cosh(x) - x^2 (3 + x^2) \sinh(x) + \sinh(x)^3\big)}{16x^9}\, .
	\] 
	Finally, we set
	\begin{equation}
		\label{eq:def_of_limiting_intensity_function_F}
		\mathfrak{F}(x) =  \frac{1}{2\pi}\frac{a_4(x)^5}{\eta(x)} \cdot \bigg(1+\frac{1}{\Delta_2(x)} \Big(\frac{a_4(x)^2\Delta_1(x)^2}{\eta(x)} + a_4(x) a_3(x)^2\Big)\bigg)^{-3} 
	\end{equation} 
	and recall that $M_1 = \lceil 4K/(\delta n) \rceil$, $M_2 = \lceil 4/\delta \rceil$ are the parameters in Definition~\ref{definition:net_in_the_main_annulus} that determine the size of the $\delta$-net $\mathsf{N}$.
	\begin{lemma}
		\label{lemma:gaussian_asymptotic_for_a_single_probability}
		Let $z\in \Omega_K$ be such that $\arg(z) \in [n^{-1+4\tau},\pi - n^{-1+4\tau}]$, and let $A_z(U)$ be given by~\eqref{eq:def_of_event_A_z_U_for_net_point}. Then, as $n\to \infty$,
		\[
		\bP_{{\sf G}}\big[A_z(U)\big] = \frac{K}{M_1 M_2}  \mathfrak{F}(x) \Big( \int_{U} t^3{\rm d}t \Big) \big(1+o(1)\big)
		\] 
		where $x = n\log|z|$, and the error term is uniform in $K>0$ and $U\subset \bR_{\ge 0}$ in a compact set. 
	\end{lemma}
	The proof of Lemma~\ref{lemma:gaussian_asymptotic_for_a_single_probability} is merely a Gaussian computation, and we postpone it to Section~\ref{sec:limiting_intensity_for_gaussian}. The next claim lists some basic properties of the limiting intensity function $\mathfrak{F}$.
	\begin{claim}
		\label{claim:integral_properties_of_intensity_mathfrak_F}
		Let $\mathfrak{F}$ be given by~\eqref{eq:def_of_limiting_intensity_function_F}. Then $\mathfrak{F}$ is a non-negative smooth function on $\bR$ with 
		\begin{equation}
			\label{eq:value_of_intensity_at_x=0}
			\mathfrak{F}(0) = \frac{1}{4\pi}  \times \frac{27}{10^4} 
		\end{equation}
		Furthermore, there exists an absolute constant $C>0$ so that
		\[
		\sup_{x\ge 1} \big(e^x \, \mathfrak{F}(x) \big) \le C  \, , \quad \text{and} \quad  \sup_{x<-1} \big(x^{16} \, \mathfrak{F}(x) \big) \le C \, .
		\]
	\end{claim}
	\begin{proof}
		We already pointed out that $\Delta_1,\Delta_2$ are non-negative functions, while the non-negativity of $a_0,\ldots,a_4$ is obvious from the definition~\eqref{eq:def_of_a_j}. Hence, to prove that $\mathfrak{F}$ is non-negative, we only need to show that $\eta$ given by~\eqref{eq:def_of_eta_function} is a positive function. Abbreviating the variable $x$ for the moment, this comes from the fact that
		\begin{equation}
			\label{eq:eta_as_a_schur_complement}
			\eta = \Delta_2\bigg(a_0 - \frac{a_1^2a_4 -2 a_1 a_2a_3 + a_2^3}{a_2a_4-a_3^2}\bigg) = \Delta_2 \bigg(a_0 - \frac{1}{a_2a_4 - a_3^2}\begin{pmatrix}
				a_1 & a_2
			\end{pmatrix} \cdot \begin{pmatrix}
				a_4 & - a_3 \\ -a_3 & a_2
			\end{pmatrix}\cdot \begin{pmatrix}
				a_1 \\ a_2
			\end{pmatrix} \bigg)
		\end{equation}
		as can be seen from a direct inspection. Hence, $\eta/\Delta_2$ is the Schur complement of the positive-definite matrix given by
		\begin{equation*}
			\begin{pmatrix}
				a_0 & a_1 & a_2 \\  a_1 & a_2 & a_3 \\ a_2 & a_3 & a_4
			\end{pmatrix} \, ,
		\end{equation*}
		and hence, $\eta$ is also pointwise positive. To find the value at $x=0$, a direct computation shows that
		\begin{equation*}
			a_j(0) = \frac{1}{j+1}\, , \quad \Delta_1(0) = \frac{1}{72} \, , \quad \Delta_2(0) = \frac{1}{240} \, , \quad \eta(0) = \frac{1}{2160} \, ,
		\end{equation*}
		and plugging the above into~\eqref{eq:def_of_limiting_intensity_function_F} and some simplifications gives~\eqref{eq:value_of_intensity_at_x=0}. Finally, to get the asymptotic bounds as $|x|\to \infty$, we first note that
		\begin{align*}
			a_j(x) \sim \frac{e^{2x}}{2x} \, , \qquad  \text{for } \ j\in\{0,\ldots,4\} \, 
		\end{align*}
	and $$\Delta_1(x) \sim \frac{e^{4x}}{16 x^4}\,, \quad \Delta_2(x) \sim \frac{e^{4x}}{16 x^4}\,,\quad \eta(x) \sim \frac{e^{6x}}{128 x^{13}} $$ 
		as $x\to +\infty$, which in turn implies that, as $x\to \infty$,
		\[
		\log \Delta_1(x)\sim  4x \, , \quad \log \Delta_2(x)\sim 4x \, , \quad \log \eta(x) \sim 6x \, .
		\]
		Plugging this asymptotic into~\eqref{eq:def_of_limiting_intensity_function_F} gives that
		$
		\log \mathfrak{F}(x) \sim -2x 
		$
		as $x\to +\infty$,
		and so the right-tail bound follows. As for the left-tail bound, we note from~\eqref{eq:def_of_a_j} that
		\[
		a_j(x) \simeq |x|^{-j-1} \, , \qquad \text{as }\ x\to -\infty \, , 
		\]
		and furthermore
		\[
		\Delta_1(x) \simeq x^{-6} \, , \quad \Delta_2(x) \simeq x^{-8} \, , \quad \eta(x)  \simeq (-x)^{-9} \, .
		\]
		Plugging once more into~\eqref{eq:def_of_limiting_intensity_function_F} and keeping track of the exponents shows that
		\[
		\mathfrak{F}(x) \simeq x^{-16} \, , \quad \text{as } \ x\to -\infty \, ,
		\]
		which proves the desired bound for the left-tail as well.
	\end{proof}
	In view of the moment computation ahead, we show that summing the success probabilities over the net points give us the limiting intensity. 
	\begin{claim}
		\label{claim:riemann_sum_converge_to_intensity}
		We have
		\[
		\lim_{n\to \infty} \sum_{z\in \mathsf{N}} \bP_{{\sf G}}\big[A_z(U)\big] = 2 \, \mathfrak{c}_\ast(K) \int_{U} t^3 {\rm d}t \, .
		\]
	\end{claim}
	\begin{proof}
		By Lemma~\ref{lemma:gaussian_asymptotic_for_a_single_probability} the claim will follow once we show that
		\begin{equation}
			\label{eq:proof_of_claim:riemann_sum_converge_to_intensity_what_we_want}
			\lim_{n\to\infty} \frac{K}{M_1 M_2} \sum_{z\in \mathsf{N}} \mathfrak{F} (n\log|z|) = 2\, \mathfrak{c}_\ast(K) \, ,
		\end{equation}
		where $\mathfrak{c}_\ast(K)$ is given by~\eqref{eq:intensity_constant_for_K_annulus}. 
		Recall that the $\delta$-net $\mathsf{N}$ is consists of points
		\[
		z = \Big(1-\frac{K}{n} + \frac{2K}{n} \cdot\frac{a}{M_1}\Big) e\big(\pi b/M_2\big) \, , 
		\]
		with $a\in\{0,\ldots,M_1\}$ and $b\in \{ 1,\ldots,M_2\}$. Hence, we see that $\log|z|$ does not depend on the variable $b$, and we get that
		\begin{equation}
			\label{eq:proof_of_Y_n_U_has_poisson_moments_reduction_to_sum_over_E_sharp_lossing_the_b_variable}
			\frac{K}{M_1 M_2}\sum_{z\in \mathsf{N}} \mathfrak{F} (n\log|z|) = \frac{K}{M_1} \sum_{a=0}^{M_1} \mathfrak{F} \Big(n \log\Big|1-\frac{K}{n} + \frac{2K}{n}\cdot \frac{a}{M_1}\Big|\Big) \, .
		\end{equation}
		Next, we claim that the sum on the right-hand side of the above display is a Riemann sum. Indeed, the point set
		\[
		n \log\Big|1-\frac{K}{n} + \frac{2K}{n}\cdot \frac{a}{M_1}\Big| \, , \quad 0\le a\le M_1\, ,
		\] 
		is a partition of the interval $[-K,K]$, with the spacing between two points given by
		\[
		n\Big(\frac{2K}{nM_1} + o\big(\frac{1}{nM_1}\big)\Big) = \frac{2K}{M_1} \Big(1+o(1)\Big)
		\]
		with the error term uniformly bounded for $K>0$ in a compact set. As $M_1$ tends to infinity as $n\to \infty$, we conclude that
		\[
		\lim_{n\to \infty} \frac{K}{M_1} \sum_{a=0}^{M_1} \mathfrak{F} \Big(n \log\Big|1-\frac{K}{n} + \frac{2K}{n}\cdot \frac{a}{M_1}\Big|\Big) = \frac{1}{2} \int_{-K}^{K} \mathfrak{F}(x) \,  {\rm d}x  = 2\, \mathfrak{c}_\ast(K) \, .
		\]
		This proves~\eqref{eq:proof_of_claim:riemann_sum_converge_to_intensity_what_we_want} and we are done. 
	\end{proof} 
	\subsection{Extracting the Poisson limit: proof of Proposition~\ref{prop:poisson_moments_for_X_n_pm}}
	We are finally ready to show the moments on $X_n^\pm(U)$ are Poissonian. For pedagogical reasons, we split the proof into two, showing the desired upper bound and the desired lower bound separately. 
	\begin{proof}[Proof of Proposition~\ref{prop:poisson_moments_for_X_n_pm}: the upper bound]
		To keep notations simple, we will suppress the dependence in $U\subset \bR_{\ge 0}$ until the final calculation. We start with the upper bound, and note that, since $X_n^+$ is a sum of indicators we have
		\begin{equation}
			\label{eq:m_falling_factorial_for_X_+_after_opening_brackets}
			(X_n^+)_m = \sum_{\substack{\{z_1,w_1\},\ldots \{z_m,w_m\} \in \binom{\mathsf{N}_{\tt s}}{2} \\ \{z_j,w_j\} \not= \{z_{j^\prime},w_{j^\prime}\} \ \forall j\not= j^\prime}} \prod_{j=1}^m \mathbf{1}_{A_{z_j}^+} \mathbf{1}_{A_{w_j}^+} \,  \mathbf{1}_{|z_j-w_j| \le \log n/n^{5/4}} \, .		
		\end{equation}
		From Lemma~\ref{lemma:control_on_maximum_of_polynomial_and_derivatives}, we see that
		\[
		\bE\big[(X_n^+)_m \mathbf{1}_{\mathcal{G}^c}\big] \lesssim |\mathsf{N}_{\tt s}|^{2m} \bP\big[\mathcal{G}^c\big] \lesssim  n^{3m+4\beta m} e^{-c \log^2 n} = o(1) \, ,
		\]
		so we can assume that $\mathcal{G}$ holds throughout the computation. By Claim~\ref{claim:linear_approximation_predicts_a_root_in_R_z}, on the event $A_{z_j}^+ \cap \mathcal{G}$, there is a root of $f_n$ in $\mathsf{R}_{z_j}$. We thus conclude from Lemma~\ref{lemma:existence_of_two_roots_implies_macroscopic_separation}, that if the events $A_{z_j}^+ \cap A_{w_j}^+ \cap \mathcal{G}$ occur for all $1\le j \le m$ (where $\mathbf{z}, \mathbf{w}$ are some fixed tuples as in~\eqref{eq:m_falling_factorial_for_X_+_after_opening_brackets}), then we must have $|z_j - z_{j^\prime}| \ge \frac{1}{n (\log n)^5}$ for all $j\not= j^\prime$. From the above reasoning, we conclude that
		\begin{multline*}
			\bE\big[(X_n^+)_m\big] \le \sum_{\substack{\{z_1,w_1\},\ldots \{z_m,w_m\} \in \binom{\mathsf{N}_{\tt s}}{2} \\ \mathbf{z} \text{ is } (\log n)^{-5} \text{-spread}}} \prod_{j=1}^m \mathbf{1}_{A_{z_j}^+} \mathbf{1}_{A_{w_j}^+} \,  \mathbf{1}_{|z_j-w_j| \le \log n/n^{5/4}} + o(1) \\ = 2^{-m} \sum_{\substack{\mathbf{z} \in (\mathsf{N}_{\tt s})^m \\ \mathbf{z} \text{ is } (\log n)^{-5} \text{-spread}}} \bE  \bigg[ \prod_{j=1}^{m} \mathbf{1}_{A_{z_j}^+} \sum_{\substack{ w_j \in \mathsf{N}_{\tt s} \\ 0<|z_j - w_j| \le \frac{\log n}{n^{5/4}} }} \mathbf{1}_{A_{w_j}^+} \mathbf{1}_{\mathcal{G}} \bigg] + o(1) \, , 
		\end{multline*}
		where the above equality follows from summing over ordered pairs in $\mathsf{N}_{\tt s}$ instead of unordered pairs as defined in $\binom{\mathsf{N}_{\tt s}}{2}$. Combining once more Claim~\ref{claim:linear_approximation_predicts_a_root_in_R_z} and Lemma~\ref{lemma:existence_of_two_roots_implies_macroscopic_separation}, we see that at most one of the $w_j$'s in the sum above can occur, and we conclude that
		\begin{equation}
			\label{eq:m_falling_factorial_for_X_+_after_taking_limsup}
			\limsup_{n\to \infty} \bE\big[(X_n^+)_m\big] \le  2^{-m}  \, \limsup_{n\to \infty} \sum_{\substack{\mathbf{z} \in (\mathsf{N}_{\tt s})^m \\ \mathbf{z} \text{ is } (\log n)^{-5} \text{-spread}}} \bP\big[A_{z_1}^+ \cap \ldots \cap A_{z_m}^+\big] \, .
		\end{equation}
		Denote by 
		\begin{equation}
			\label{eq:def_of_n_beta_spread_tuples}
			E_1 = \Big\{ \mathbf{z} \in (\mathsf{N}_{\tt s})^m \, : \, \mathbf{z} \text{ is } n^{\beta} \text{-spread} \Big\}
		\end{equation}
		and by 
		\begin{equation*}
			E_2 = \Big\{ \mathbf{z} \in (\mathsf{N}_{\tt s})^m \, : \, \mathbf{z} \text{ is } (\log n)^{-5} \text{-spread and } \exists j\not = j^\prime \text{ such that } |z_j - z_{j^\prime}| \le n^{\beta - 1} \Big\} \, .
		\end{equation*}
		Clearly, the sum on the right-hand side of~\eqref{eq:m_falling_factorial_for_X_+_after_taking_limsup} breaks into two separate sums over $E_1$ and $E_2$, respectively. We first show that the sum over $E_2$ is negligible. Indeed, we note that $|E_2| \lesssim |\mathsf{N}_{\tt s}|^m n^{\beta - 1}$, since there are $|\mathsf{N}_{\tt s}|n^{\beta - 1}$ possible choices for the close points with all other kept fixed. By Lemma~\ref{lemma:for_almost_spread_tuples_probabilities_are_bounded_by_what_you_want} we see that
		\[
		\sum_{\mathbf{z} \in E_2} \bP\big[A_{z_1}^+ \cap \ldots \cap A_{z_m}^+\big] \lesssim |E_2|  \Big(\frac{1}{n^{3/2+2\beta}}\Big)^m \log^{O(1)} n \lesssim n^{\beta - 1} \log^{O(1)} n 
		\]
		and since $\beta>0$ is small, we see that this sum does not contribute to~\eqref{eq:m_falling_factorial_for_X_+_after_taking_limsup}. On the other hand, we have by Lemma~\ref{lemma:for_spread_tuples_probabilities_of_smooth_points_factor} that 
		\begin{align*}
			\sum_{\mathbf{z} \in E_1} \bP\big[A_{z_1}^+ \cap \ldots \cap A_{z_m}^+\big] &= \big(1+o(1)\big) \bigg(\sum_{\mathbf{z} \in  E_1} \prod_{j=1}^{m} \bP_{{\sf G}}\big[A_{z_{j}}\big] \bigg) \\ 
			&= \big(1+o(1)\big) \bigg(\sum_{\mathbf{z} \in |\mathsf{N}|^m } \prod_{j=1}^{m} \bP_{{\sf G}}\big[A_{z_{j}}\big] \bigg) = \big(1+o(1)\big)\bigg(\sum_{z\in \mathsf{N}} \bP_{{\sf G}}\big[A_{z}(U)\big] \bigg)^m \, .
		\end{align*}
		In the above, we used the simple fact that $|E_1| \sim |\mathsf{N}|^m$. Claim~\ref{claim:riemann_sum_converge_to_intensity} now implies that
		\begin{equation}
			\label{eq:taking_limit_in_sum_over_m_tuples_with_probabilities}
			\lim_{n\to \infty} \sum_{\mathbf{z} \in E_1} \bP\big[A_{z_1}^+ \cap \ldots \cap A_{z_m}^+\big]  = \Big(2\, \mathfrak{c}_\ast(K) \int_{U} t^3 {\rm d}t\Big)^m 
		\end{equation}
		which, in view of~\eqref{eq:m_falling_factorial_for_X_+_after_taking_limsup}, completes the proof for the upper bound.
	\end{proof}
	We conclude the section with the proof of the lower bound. 
	\begin{proof}[Proof of Proposition~\ref{prop:poisson_moments_for_X_n_pm}: the lower bound]
		As in the proof of the upper bound, our starting point is the formula
		\begin{equation*}
			\label{eq:m_falling_factorial_for_X_-_after_opening_brackets}
			(X_n^-)_m = \sum_{\substack{\{z_1,w_1\},\ldots \{z_m,w_m\} \in \binom{\mathsf{N}_{\tt s}}{2} \\ \{z_j,w_j\} \not= \{z_{j^\prime},w_{j^\prime}\} \ \forall j\not= j^\prime}} \prod_{j=1}^m \mathbf{1}_{A_{z_j}^-} \mathbf{1}_{A_{w_j}^-} \,  \mathbf{1}_{|z_j-w_j| \le \log n/n^{5/4}} \, .		
		\end{equation*}
		Recall the set of $n^\beta$-spread tuples $E_1$ given by~\eqref{eq:def_of_n_beta_spread_tuples}. We clearly have the lower bound
		\begin{align}
			\label{eq:m_factorial_moment_X_minus_lower_bound} \nonumber
			(X_n^-)_m &\ge 2^{-m} \sum_{\mathbf{z} \in E_1} \bigg(\sum_{\substack{\mathbf{w} \in (\mathsf{N}_{\tt s})^m \\ 0<|z_j - w_j| \le \frac{\log n}{n^{5/4}} }} \prod_{j=1}^{m} \mathbf{1}_{A_{z_j}^-} \mathbf{1}_{A_{w_j}^-} \mathbf{1}_{\mathcal{G}} \bigg) \\ &=  2^{-m} \sum_{\mathbf{z} \in E_1} \bigg[\prod_{j=1}^{m} \bigg(\mathbf{1}_{A_{z_j}^-} \cdot \Big( \sum_{ \substack{w_j \in \mathsf{N}_{\tt s} \\ 0<|z_j - w_j| \le \frac{\log n}{n^{5/4}}}}   \mathbf{1}_{A_{w_j}^-} \mathbf{1}_{\mathcal{G}} \Big) \bigg) \bigg]\, .
		\end{align}
		By Claim~\ref{claim:linear_approximation_predicts_a_root_in_R_z} and Lemma~\ref{lemma:quadratic_approximation_predicts_second_root}, on the event $A_{z_j}^- \cap \mathcal{G}$ there are two distinct roots $\alpha,\alpha^\prime$ of $f_n$ present in $\bD(z_j,\frac{\log n}{n^{5/4}})$, with $\alpha\in \mathsf{R}_{z_j}^\sharp$ and 
		\[
		\Big|\alpha^\prime - \Big(z+\frac{f_n(z)}{f_n^\prime(z)} - \frac{2 f_n^\prime(z)}{f_n^{\prime\prime}(z)}\Big)\Big| \le n^{-5/4-2\beta} \log^5 n\, .
		\] 
		Since the same it true with $z_j$ replaced by $w_j$, and since we know from Lemma~\ref{lemma:existence_of_two_roots_implies_macroscopic_separation} there are only two roots, we conclude the inequality
		\[
		\mathbf{1}_{A_{z_j}^-} \cdot \Big( \sum_{ \substack{w_j \in \mathsf{N}_{\tt s} \\ 0<|z_j - w_j| \le \frac{\log n}{n^{5/4}}}}   \mathbf{1}_{A_{w_j}^-} \mathbf{1}_{\mathcal{G}} \Big) \ge \mathbf{1}_{A_{z_j}^-}\mathbf{1}_{\mathcal{G}} - \mathbf{1}_{\Upsilon_{z_j}}
		\]
		where $\Upsilon_{z_j}$ is given by~\eqref{eq:def_of_event_upsilon}. Plugging this observation into~\eqref{eq:m_factorial_moment_X_minus_lower_bound}, we see that
		\begin{equation}
			\label{eq:m_factorial_moment_X_minus_lower_bound_after_taking_expectation}
			\bE\big[(X_n^-)_m\big] \ge 2^{-m} \sum_{\mathbf{z} \in E_1} \bP\big[A_{z_1}^- \cap \ldots \cap A_{z_m}^- \cap \mathcal{G} \big] - O\bigg(\sum_{\mathbf{z} \in E_1} \bP\big[\Upsilon_{z_1} \cap A_{z_2}^+ \cap \ldots A_{z_m}^+\big]\bigg) \, .
		\end{equation} 
		We start by estimating the error terms. Indeed, Lemma~\ref{lemma:bound_on_probability_of_tuple_with_upsilon} gives that
		\[
		\sum_{\mathbf{z} \in E_1} \bP\big[\Upsilon_{z_1} \cap A_{z_2}^+ \cap \ldots A_{z_m}^+\big] \lesssim |E_1|\Big(\frac{1}{n^{3/2+2\beta}}\Big)^m \cdot n^{-\beta/10} \lesssim n^{-\beta/100} \, . 
		\]
		Furthermore, Lemma~\ref{lemma:control_on_maximum_of_polynomial_and_derivatives} implies that $|E_1| \bP[\mathcal{G}^c] = o(1)$ as $n\to \infty$, and we conclude from~\eqref{eq:m_factorial_moment_X_minus_lower_bound_after_taking_expectation} that
		\[
		\liminf_{n\to \infty} \bE\big[(X_n^-)_m\big] \ge 2^{-m} \lim_{n\to \infty} \sum_{\mathbf{z} \in E_1} \bP\big[A_{z_1}^- \cap \ldots \cap A_{z_m}^- \big] =  \Big(\mathfrak{c}_\ast(K) \int_{U} t^3 {\rm d}t\Big)^m
		\]
		The last equality is derived exactly the same as in the proof of~\eqref{eq:taking_limit_in_sum_over_m_tuples_with_probabilities} in the upper bound part. First, we apply Lemma~\ref{lemma:for_spread_tuples_probabilities_of_smooth_points_factor} to show that the probabilities asymptotically factor (we can do that since the tuples from $E_1$ are $n^\beta$-spread), and then apply Claim~\ref{claim:riemann_sum_converge_to_intensity} to show that the Riemann sums converge to the desired integral. With that we conclude the proof of Proposition~\ref{prop:poisson_moments_for_X_n_pm}.
	\end{proof}
	\section{Small ball probability bounds for general points}
	\label{sec:small_ball_for_general_points}
	Recall that $d = \min\{n,(1-r)^{-1}\}$ is the effective degree of the random polynomial $f_n$ at the point $z=re^{i\theta}\in \bD$. We denote by $V_n(z)$ the covariance matrix of the random vector  
	\begin{align}
		\nonumber
		\label{eq:pointwise_gaussian_vector}
		d^{-1/2}\big(&f_n(z), d^{-1} z f_n^{\prime}(z)\big) \\ &= d^{-1/2} \Big(\sum_{k=0}^{n} \xi_k r^k\cos(k\theta) , \sum_{k=0}^{n} \xi_k r^k\sin(k\theta) , \sum_{k=0}^{n} \xi_k \tfrac{k}{d} r^k\cos(k\theta) , \sum_{k=0}^{n} \xi_k \tfrac{k}{d} r^k\cos(k\theta) \Big)
	\end{align}
	As we always identify $\bC\simeq \bR^2$, $V_n(z)$ is a $4\times4$ real, symmetric, positive-definite matrix given by
	\begin{equation}
		\label{eq:pointwise_covariance_matrix}
		V_n(z) = \frac{1}{d}\sum_{k=0}^{n} r^{2k}
		\begin{pmatrix}
			\cos^2(k\theta) &  \cos(k\theta)\sin(k\theta) &  \tfrac{k}{d} \cos^2(k\theta) & \tfrac{k}{d} \cos(k\theta)\sin(k\theta)  \\ \cos(k\theta)\sin(k\theta) & \sin^2(k\theta) &  \tfrac{k}{d} \cos(k\theta)\sin(k\theta) & \tfrac{k}{d}\sin^2(k\theta) \\ \tfrac{k}{d} \cos^2(k\theta) & \tfrac{k}{d} \cos(k\theta)\sin(k\theta)  &  \tfrac{k^2}{d^2} \cos^2(k\theta) & \tfrac{k^2}{d^2} \cos(k\theta)\sin(k\theta) \\ \tfrac{k}{d}\cos(k\theta) \sin(k\theta) & \tfrac{k}{d}\cos^2(k\theta) & \tfrac{k^2}{d^2} \cos(k\theta) \sin(k\theta) & \tfrac{k^2}{d^2} \sin^2(k\theta) 
		\end{pmatrix} \, .
	\end{equation}
	
	\subsection{Non-singularity of covariance matrix}
	We start our argument with a simple lemma that shows that if $z$ is separated from the real axis then the correlations between the random polynomial and its derivative are non-trivial. 
	\begin{lemma}
		\label{lemma:singularity_properties_of_covariance_pointwise}
		Let $V=V_n(z)$ be the covariance matrix~\eqref{eq:pointwise_covariance_matrix}, and denote by $\la = \la_n(z)$ its smallest eigenvalue. There exists a constant $c>0$ so that for all $d^{-1-\tau}\le \theta \le \pi - d^{-1-\tau}$ we have $$\la \ge c \min\{1,(d\theta)^7,(d(\pi - \theta))^7\} \, .$$
	\end{lemma}
	\begin{remark}
		\label{remark:m=1_lemma_is_baby_case_for_general_m}
		In Section~\ref{sec:gaussian_comparison_for_tuples}, we will prove a similar non-singularity result for the covariance matrix of tuples of spread points of the form~\eqref{eq:def_of_S_n_z}, see Lemma~\ref{lemma:non-singularity_for_covariance_of_spread_tuple} therein. Formally, the case where $|\theta|\ge d^{-1}$ and $d=n$ in Lemma~\ref{lemma:singularity_properties_of_covariance_pointwise} follows from this case. Besides the fact that we need know the non-singularity of the covariance for different degrees $d=d_n(z)$, we decided to include a separate proof of Lemma~\ref{lemma:singularity_properties_of_covariance_pointwise} for two reasons: (1) we need the qualitative lower bound for $\la$ as $d\theta$ becomes small, and (2) it serves as a ``warm-up" for the more technical proof of Lemma~\ref{lemma:non-singularity_for_covariance_of_spread_tuple}, where a similar (but more complicated) method is applied.
	\end{remark} 
	\begin{proof}
		Note that the matrix $V$ is invariant under the transformation $\theta \mapsto \pi - \theta$, and so we can assume without loss of generality that $\theta\le \pi/2$. Throughout the proof we will focus on the more difficult case where $\theta \le d^{-1}$, and will explain briefly in end what to do in the complementary case $\pi/2 \ge  \theta >d^{-1}$. Let $W=W_n(z)$ denote the $(n+1)\times 4$ matrix with rows
		\[
		\Big( r^k\cos(k\theta) ,  r^k\sin(k\theta) ,  \tfrac{k}{d} r^k\cos(k\theta) ,  \tfrac{k}{d} r^k\cos(k\theta) \Big) \, , \qquad k=0,\ldots,n \, .
		\] 
		Then $V= d^{-1} W^{\sf T} W$ and to prove the lemma it will be enough to show
		\begin{equation}
			\label{eq:lower_bound_on_singular_value_of_W_pointwise_case}
			\min_{u\in \bS^3} \|Wu\|^2 \ge c d  (d\theta)^{7} \, .
		\end{equation}
		Denote by $u=(u_1,\ldots,u_4)\in \bS^3$, and note that for each $k\in[0,n]$ we have
		\begin{align*}
			(Wu)_k &=  r^k\Big( \cos(k\theta) \big(u_1 + \tfrac{k}{d}u_3\big) + \sin(k\theta) \big(u_2 + \tfrac{k}{d}u_4\big)\Big) \\ &= \frac{r^k}{2} \Big(e(k\theta) \big(u_1 + \tfrac{k}{d} u_3 -iu_2 -i\tfrac{k}{d} u_4 \big) + e(-k\theta) \big(u_1 + \tfrac{k}{d} u_3 +iu_2 +i\tfrac{k}{d} u_4\big)\Big) \, .
		\end{align*} 
		We will denote, just for this proof,  $x=u_1+iu_2$ and $y=u_3 + iu_4$; the above display then reads
		\begin{equation}
			\label{eq:evaluation_of_W_on_u_pointwise}
			(Wu)_k= \frac{1}{2} \Big(z^k \big(\overline{x} + \tfrac{k}{d} \overline y \big) + \overline{z}^k \big(x + \tfrac{k}{d}y)\Big) \, .
		\end{equation}
		Take some small $\eps>0$ that we will choose later and set $L=\lfloor \eps d\rfloor$. Consider the arithmetic progression $P=\{j\cdot L : j\in \bZ \}$ and denote by $W_P$ the sub-matrix of $W$ with rows in the progression $P$. We will first show that
		\begin{equation}
			\label{eq:lower_bound_on_singular_value_of_W_pointwise_case_on_progression}
			\min_{u\in \bS^3} \|W_P u\|^2 \gtrsim_\eps (d\theta)^6. 
		\end{equation}
		For $\zeta \in \bC$ and a sequence $f:\bZ\to \bC$ we define the (twisted) second order discrete differential by
		\begin{equation}
			\label{eq:definition_of_second_order_twisted_differntial}
			(D_\zeta f) (k) = \sum_{t=0}^2 \binom{2}{t} (-1)^{t} \zeta^{-tL} f(k+tL).
		\end{equation} 
		Later on, in Section~\ref{sec:gaussian_comparison_for_tuples}, we will need to consider a similar differential of order three, as we will have to deal with the second derivative of $f_n$ as well, but for now~\eqref{eq:definition_of_second_order_twisted_differntial} will do. For the sequences
		\[
		a_z(k) = z^k \, , \qquad b_z(k) = \frac{k}{d} z^k \, ,
		\]
		a simple computation shows that
		\begin{align*}
			& (D_\zeta a_z)(k) = z^k \Big(1-(z/\zeta)^L\Big)^2 \\ & (D_\zeta b_z)(k) = \frac{z^k}{d} \Big(k\big(1-2(z/\zeta)^L + (z/\zeta)^{2L}\big) + 2L \big((z/\zeta)^{2L} - (z/\zeta)^L\big) \Big)  \, .
		\end{align*}
		Plugging in $\zeta = z$ and denoting by $D$ the matrix associated to the linear operator $D_{\zeta}$ on $\bC^P$, we see from~\eqref{eq:evaluation_of_W_on_u_pointwise} that
		\[
		(D W_P u)_k = \frac{\overline{z}^k}{2}\Big(x\big(1-e(-2L\theta)\big)^2 + y \frac{k\big(1-e(-2L\theta)\big)^2 + 2L e(-2L\theta)\big(e(-2L\theta) - 1\big)}{d}\Big) \, .
		\]
		In particular, for all $k\in P$ such that $k+2L\in P$ we can take the modulus and sum to get
		\begin{align}
			\label{eq:lower_bound_sum_over_arithmetic_progression_pointwise_case}
			\nonumber
			\sum_{k\in P : k+2L\in P} \|(D W_P u)_k\|^2 &\gtrsim  \big|1-e(-2L\theta)\big|^4 \sum_{\substack{k\in P : k+2L\in P \\ k \le d}} \Big|x + y \frac{k}{d} - y\frac{2L e(-2L\theta)}{d\big(e(-2L\theta) - 1\big)}\Big|^2 \\ & \gtrsim  (L\theta)^4  \sum_{\substack{k\in P : k+2L\in P \\ k \le d}} \Big|x + y \frac{k}{d} - y\frac{2L e(-2L\theta)}{d\big(e(-2L\theta) - 1\big)}\Big|^2 \, .
		\end{align}
        The next claims gives a lower bound for the sum in \eqref{eq:lower_bound_sum_over_arithmetic_progression_pointwise_case}.

        \begin{claim}\label{cl:sum-AP-bounded}
            We have $$\sum_{\substack{k\in P : k+2L\in P \\ k \le d}} \Big|x + y \frac{k}{d} - y\frac{2L e(-2L\theta)}{d\big(e(-2L\theta) - 1\big)}\Big|^2\gtrsim_\eps 1\,.$$
        \end{claim}
        \noindent \textit{Proof of Claim \ref{cl:sum-AP-bounded}.} 
            Recalling that $x=u_1+iu_2$ and $y=u_3+iu_4$ are such that $|x|^2 + |y|^2 = 1$, to further lower bound the right-hand side of~\eqref{eq:lower_bound_sum_over_arithmetic_progression_pointwise_case} we split into three cases:
		\begin{equation*}
			\text{(I)} \ |y| \ge 4 d\theta ; \qquad \text{(II)} \ |y| \le \tfrac{1}{4} d\theta ; \qquad \text{and } \quad \text{(III)} \ \tfrac{1}{4} d\theta < |y| \le 4d\theta.
		\end{equation*}
		First note that the number of summands in~\eqref{eq:lower_bound_sum_over_arithmetic_progression_pointwise_case} is at least of size $|P\cap [0,d]| \asymp \eps^{-1}$. In case (I), we have that 
		\[
		\Big|y\frac{2L e(-2L\theta)}{d\big(e(-2L\theta) - 1\big)}\Big| \ge |y| \frac{1}{d\theta} \ge 4 \, ,
		\]
		and since $|x+ \tfrac{k}{d} y| \le 2$ we get that
		\begin{equation}
			\label{eq:lower_bound_sum_over_arithmetic_progression_pointwise_case_when_the_summand_is_large}
			\sum_{\substack{k\in P : k+2L\in P \\ k \le d}} \Big|x + y \frac{k}{d} - y\frac{2L e(-2L\theta)}{d\big(e(-2L\theta) - 1\big)}\Big|^2  \gtrsim_\eps 1 \, .
		\end{equation}
		In case (II), we have $|x| \ge \tfrac{3}{4}$ and since
		\[
		\Big|y \frac{k}{d} - y\frac{2L e(-2L\theta)}{d\big(e(-2L\theta) - 1\big)} \Big|\le \frac{d\theta}{4} \big(1+ 1.5 (d\theta)^{-1} \big) < \frac{3}{4} 
		\]
		we get that~\eqref{eq:lower_bound_sum_over_arithmetic_progression_pointwise_case_when_the_summand_is_large} holds in this case as well. It remains to deal with case (III). Denote by
		\[
		{\sf p}_k = \Big| x + y \frac{k}{d} - y\frac{2L e(-2L\theta)}{d\big(e(-2L\theta) - 1\big)} \Big| \, , 
		\]
		and note that $\max\{{\sf p}_0 , {\sf p}_d\} \ge |y|/2 \ge d\theta/4$. Furthermore, all consecutive elements from $\{{\sf p}_k\}_{k\in P}$ have
		\[
		\big|{\sf p}_{k+L} - {\sf p}_k\big| \le |y| \frac{L}{d} \le 4\eps \,  d\theta  
		\]
		where $\eps>0$ is a sufficiently small constant. Therefore, at least $10^{-2} \eps^{-1}$ elements $k\in P$ must have ${\sf p}_k \gtrsim d\theta$, and we conclude that
		\[
		\sum_{\substack{k\in P : k+2L\in P \\ k \le d}} \Big|x + y \frac{k}{d} - y\frac{2L e(-2L\theta)}{d\big(e(-2L\theta) - 1\big)}\Big|^2  \gtrsim_\eps (d\theta)^2 \, .
		\]  
        \qed
        
		\noindent
		Combining Claim \ref{cl:sum-AP-bounded} with~\eqref{eq:lower_bound_sum_over_arithmetic_progression_pointwise_case} we arrive at 
		\[
		\sum_{k\in P : k+2L\in P} \|(D W_P u)_k\|^2 \gtrsim_{\eps} (d\theta)^6 \, .
		\]  
		Since the operator norm of $D:\ell_2(P) \to \ell_2(P)$ is $O(1)$, the display above shows that uniformly in $u\in \bS^3$ we have
		\[
		\sum_{k\in P : k+2L\in P} \|(W_P u)_k \|^2 \gtrsim_{\eps}  (d\theta)^6 \, ,
		\]
		which gives~\eqref{eq:lower_bound_on_singular_value_of_W_pointwise_case_on_progression}. It remains to prove~\eqref{eq:lower_bound_on_singular_value_of_W_pointwise_case}. Consider the sub-matrices $W_P,W_{1+P}, \ldots,W_{d_0+P}$ composed of rows from the shifted progressions $P,1+P,\ldots,d_0 + P$, respectively. If $d_0 \le L$ then these sub-matrices are all disjoint. Moreover, denoting by
		\[
		F = \begin{pmatrix}
			r^kH & 0 \\ 0 & r^kH
		\end{pmatrix} \, \quad \text{with} \quad H = \begin{pmatrix}
			\cos \theta & - \sin \theta \\ \sin \theta & \cos \theta 
		\end{pmatrix}\, ,
		\] 
		we see that $W_{k+P}$ and $W_P H^k$ differ only by a matrix of norm $O(d_0/d)$. Indeed, the matrix $H$ corresponds to the shift in the phase $(\cos k\theta , \sin k\theta) \mapsto (\cos (k+1)\theta , \sin (k+1)\theta)$, so the matrices $W_{k+P}$ and $W_P F^k$ differ only in the $k/d$ terms in the last two variables. Since the matrix $H$ is unitary, we have for $k\in\{1,\ldots,d\}$ that 
		\[
		\sigma_4(W_P F^k) \gtrsim \sigma_4(W_P) \stackrel{\eqref{eq:lower_bound_on_singular_value_of_W_pointwise_case_on_progression}}{\gtrsim} (d \theta)^{6} \, .
		\]
		By choosing $d_0 = \lfloor\eps d^2 \theta \rfloor \le L$ with $\eps>0$ sufficiently small we have by the triangle inequality that $\sigma_4(W_{k+P}) \gtrsim  (d \theta)^{6}$ for all $1\le k \le d_0$. As those arithmetic progressions are disjoint, we can sum up over all of them and obtain
		\begin{equation*}
			\|W u\|^2 \ge \sum_{k=0}^{d_0} \| W_{k+P} u \|^2 \gtrsim d_0  (d \theta)^{6} \gtrsim_\eps d  (d \theta)^{7} \, .
		\end{equation*}
		which yields~\eqref{eq:lower_bound_on_singular_value_of_W_pointwise_case} and concludes the lemma in the case where $d\theta \le 1$. To deal with the complementary case $\theta > d^{-1}$, one can repeat the argument above with $L = \lfloor \tfrac{1}{100} d \rfloor$ (say) and note that the corresponding lower bound that we get from~\eqref{eq:lower_bound_sum_over_arithmetic_progression_pointwise_case} now reads 
		\[
		\sum_{k\in P : k+2L\in P} \|(D W_P u)_k\|^2 \gtrsim 1 \, .
		\]
		Summing this over the $\Omega(d)$-disjoint arithmetic progressions (as we did before) yields the desired lower bound also for this case, and we are done. 
	\end{proof}
	\subsection{Comparison of characteristic function to a Gaussian}
	To handle the small-ball bounds we will apply the `exponentially-tilted' Esseen inequality (Proposition~\ref{propsition:exponentially_tilded_esseen_ineq}) and compare the characteristic function of~\eqref{eq:pointwise_gaussian_vector} with the characteristic function of the limiting Gaussian random vector.  The desired comparison is given by the following lemma.
	\begin{lemma}
		\label{lemma:CLT_on_fourier_side_bounded_third_moment}
		Let ${\bf X}_1,\ldots,{\bf X}_n$ be independent random vectors taking values in $\bR^k$, having distributions $G_1,\ldots,G_n$ respectively. Assume that the ${\bf X}_j$'s have mean zero and a uniformly bounded third moment, that is
		\begin{equation}
			\label{eq:bounded_third_moment_for_X_j}
			\max_{1\le j\le n}\bE\big[|{\bf X}_j|^3\big] \le M\, .
		\end{equation}
		Assume further that the average covariance matrix $V$ is non-singular, and denote by $\la$ its smallest eigenvalue. Let $B$ be the symmetric positive-definite matrix so that $V^{-1} = B^2$. Then there exists constants $C,c>0$ depending only on $M$ so that
		\begin{equation*}
			\sup_{|t|\le c n^{1/2} \la^{3/2}} \bigg| \prod_{j=1}^{n} \widehat{G}_j\Big(\frac{Bt}{\sqrt{n}}\Big) - \exp\Big(-\frac{|t|^2}{2}\Big)\bigg| \le C n^{-1/2} \la^{-3/2} |t|^3 \exp(-c|t|^2) \, ,
		\end{equation*}
		where $\widehat{G}_j(\xi) = \bE[\exp(i \langle \xi, {\bf X}_j \rangle)]$ is the characteristic functions of the random vector ${\bf X}_j$.
	\end{lemma}
	\begin{proof}
		By~\eqref{eq:bounded_third_moment_for_X_j}, we know that for all $\xi\in \bR^k$
		\begin{equation*}
			\Big|\widehat{G}_j(\xi) - 1 + \frac{1}{2} \langle \xi , V_j \xi \rangle \Big| \le \frac{1}{6} \, M |\xi|^3
		\end{equation*}
		where $V_j$ is the covariance matrix of ${\bf X}_j$, see for instance~\cite[Corollary~8.2]{Bhattacharya-Rao}. Hence, for $\xi$ small enough, we have that
		\begin{align*}
			\prod_{j=1}^{n}\widehat{G}_j(\xi) &= \exp\Bigg\{ \sum_{j=1}^{n} \log\Big(1+ \big(\widehat{G}_j(\xi) -1 \big)\Big)\Bigg\} \\ &=  \exp\Bigg\{ - \frac{1}{2}\sum_{j=1}^{n} \langle \xi , V_j \xi \rangle + O\big(nM|\xi|^3\big)\Bigg\}  = \exp\Bigg\{ - \frac{n}{2}\langle \xi , V \xi \rangle + O\big(nM|\xi|^3\big)\Bigg\}\, ,
		\end{align*}
		where $\log$ is some branch of the complex logarithm defined in the vicinity of $1$. Plugging $\xi = Bt/\sqrt{n}$ and using the identity $B^2 = V^{-1}$ we get that
		\begin{equation*}
			\prod_{j=1}^{n} \widehat{G}_j\Big(\frac{Bt}{\sqrt{n}}\Big) = \exp\bigg\{ - |t|^2/2 + O\big(n^{-1/2} \la^{-3/2} M |t|^3 \big)\bigg\} \, .
		\end{equation*} 
		It remains to apply the elementary inequality $|e^x - 1| \le |x|e^{x}$ to get that
		\begin{align*}
			\bigg| \prod_{j=1}^{n} \widehat{G}_j\Big(\frac{Bt}{\sqrt{n}}\Big) - \exp\Big(-\frac{|t|^2}{2}\Big)\bigg| &= \exp\Big(-\frac{|t|^2}{2}\Big) \bigg|\exp\Big(O(n^{-1/2} \la^{-3/2} M |t|^3)\Big)-1\bigg| \\ &\lesssim n^{-1/2} \la^{-3/2} |t|^3 \exp\Big(-\frac{|t|^2}{2} + O(n^{-1/2} \la^{-3/2} M |t|^3) \Big) \, ,
		\end{align*}
		and since $|t| \le c \sqrt{n}\la^{3/2}$ for some small $c>0$ the lemma follows.
	\end{proof}
	\subsection{Proving the ``worst-case'' small-ball bounds}
	We are now ready to prove the small-ball bounds stated in Section~\ref{sec:small_ball_probability_general_statements}. 
	\begin{proof}[Proof of Claim~\ref{claim:small_ball_away_from_real_axis}]
		We may assume by a simple covering argument that $a=b=\delta \ge (\log d)/d^{1/2}$. By Proposition~\ref{propsition:exponentially_tilded_esseen_ineq} we have that
		\begin{multline}
			\label{eq:proof_of_small_ball_away_from_real_axis_after_esseen}
			\max_{\alpha,\beta\in \bC} \bP \Big[|f_n(z)-\alpha| \le a d^{1/2} \, , \, |f_n^\prime(z)-\beta| \le b d^{3/2} \, , \, |f_n^{\prime \prime}(z)| \geq t d^{5/2} \Big] \\ \leq C e^{-c t^2} (ab)^2 \int_{\bR^4} \Bigg(\prod_{k=0}^{n} \varphi_{\eta_0} \Big( d^{-1/2} \Big\langle \xi , \begin{pmatrix}
				z^k \\ \tfrac{k}{d} z^k
			\end{pmatrix} \Big\rangle \Big) \Bigg)^{1/4} e^{-c\delta^2|\xi|^2} \, {\rm d}m(\xi) \, .
		\end{multline}
		Therefore, the claim will follow once we prove uniform boundedness of the integral in~\eqref{eq:proof_of_small_ball_away_from_real_axis_after_esseen}. Indeed, we will split the integral into two domains ${\rm I} = \{|\xi| \le \eps_0d^{1/2}\}$ and ${\rm II} = \{|\xi| > \eps_0 d^{1/2}\}$ for some sufficiently small $\eps_0>0$ (depending only on the distribution of $\xi_0$). Since $\varphi_{\eta_0} (\cdot) \le 1$, we have
		\begin{align*}
			\int_{{\rm II}} \Bigg(\prod_{k=0}^{n} \varphi_{\eta_0}\Big( d^{-1/2} \Big\langle \xi , \begin{pmatrix}
				z^k \\ \tfrac{k}{d} z^k
			\end{pmatrix} \Big\rangle \Big) \Bigg)^{1/4} e^{-\delta^2|\xi|^2/2} \, {\rm d}m(\xi) \le \int_{{\rm II}}e^{-c\delta^2|\xi|^2} \, {\rm d}m(\xi) = O(d^{O(1)} e^{-c(\log d)^2}) = o(1) \, .
		\end{align*}
		To deal with the integral over ${\rm I}$, we apply Lemma~\ref{lemma:CLT_on_fourier_side_bounded_third_moment} with the bounded random vectors 
		\[
		{\bf X}_k = \eta_k \begin{pmatrix}
			z^k \\ \tfrac{k}{d} z^k
		\end{pmatrix} \, .
		\]
		In other words, we deal with the ``symmetrized" version of our random polynomial $g_n$ given by~\eqref{eq:symmetrized_version_of_f_n}, instead of the original $f_n$. We get that
		\begin{multline}
			\label{eq:proof_of_small_ball_away_from_real_axis_bound_on_Gaussian_part}
			\int_{{\rm I}} \Bigg(\prod_{k=0}^{n} \varphi_{\eta_0}\Big( d^{-1/2} \Big\langle \xi , \begin{pmatrix}
				z^k \\ \tfrac{k}{d} z^k
			\end{pmatrix} \Big\rangle \Big) \Bigg)^{1/4} \, {\rm d}m(\xi) \\ \lesssim \int_{\bR^4} e^{-c|\xi|^2} {\rm d}m(\xi) + \frac{1}{\sqrt{d}} \int_{\bR^4} |\xi|^3 e^{-c|\xi|^2} {\rm d}m(\xi) = O(1) \, .
		\end{multline}
		We note that the uniform non-singularity of $V=V_n(z)$ in this range is guaranteed by Lemma~\ref{lemma:singularity_properties_of_covariance_pointwise}, which is the reason we can use the Gaussian comparison for the characteristic function here all the way to $\eps_0 d^{1/2}$ for some $\eps_0>0$ small. Combining both estimates gives a uniform upper bound on the integral in~\eqref{eq:proof_of_small_ball_away_from_real_axis_after_esseen}, and the claim follows. 
	\end{proof}
	\begin{proof}[Proof of Claim~\ref{claim:small_ball_points_near_the_real_axis}]
		The proof here is similar to the proof of Claim~\ref{claim:small_ball_away_from_real_axis}, only that here we cannot get the Gaussian comparison for the characteristic function all the way up to $c\sqrt{d}$ because of the quantitative singularity of $V_n(z)$ in this range. Still, as this non-singularity is controlled via Lemma~\ref{lemma:singularity_properties_of_covariance_pointwise}, we can get some bound. Arguing similarly as in the proof of Claim~\ref{claim:small_ball_away_from_real_axis}, from the inequality~\eqref{eq:proof_of_small_ball_away_from_real_axis_after_esseen} we split the integral into two domains:
		\[
		{\rm I}^\prime = \{|\xi| \le d^{1/2-2\tau}\} \qquad {\rm II}^{\prime} = \{|\xi| > d^{1/2-2\tau}\} \, .
		\]
		For the domain ${\rm II}^{\prime\prime}$ we have the bound
		\[
		\int_{{\rm II}^\prime} \Bigg(\prod_{k=0}^{n} \varphi_{\eta_0}\Big( d^{-1/2} \Big\langle \xi , \begin{pmatrix}
			z^k \\ \tfrac{k}{d} z^k
		\end{pmatrix} \Big\rangle \Big) \Bigg)^{1/4} e^{-\delta^2|\xi|^2/2} \, {\rm d}m(\xi) \le \int_{{\rm II}^\prime} e^{-\delta^2|\xi|^2/2} \, {\rm d}m(\xi) = O(d^{O(1)} e^{-d^{2\tau}}) = o(1) \, .
		\]
		To deal with ${\rm I}^\prime$, we note that Lemma~\ref{lemma:singularity_properties_of_covariance_pointwise} implies that $\la \ge d^{-\tau}$ for $z$ in our range, and therefore Lemma~\ref{lemma:CLT_on_fourier_side_bounded_third_moment} gives that
		\[
		\prod_{k=0}^{d} \varphi_{\eta_0}\Big( d^{-1/2} \Big\langle \xi , \begin{pmatrix}
			z^k \\ \tfrac{k}{d} z^k
		\end{pmatrix} \Big\rangle \Big) \lesssim \exp\Big(-\langle \xi, V \xi\rangle\Big) + \frac{d^{2\tau}}{\sqrt{d}} |\xi|^{3/2} \exp\Big(-c\langle \xi, V \xi\rangle\Big)  
		\] 
		for $\xi\in {\rm I}^\prime$. This yields that
		\[
		\int_{{\rm II}^\prime} \Bigg(\prod_{k=0}^{n} \varphi_{\eta_0}\Big( d^{-1/2} \Big\langle \xi , \begin{pmatrix}
			z^k \\ \tfrac{k}{d} z^k
		\end{pmatrix} \Big\rangle \Big) \Bigg)^{1/4} \, {\rm d}m(\xi) \lesssim \frac{1}{\sqrt{\det V}} \le \la^{-2} \le d^{2\tau}
		\]
		and we get what we wanted.
	\end{proof}

	\section{Small-ball probability bounds for smooth points}
	\label{sec:small_ball_bounds_for_smooth_points}
	Recall that a point in $z=re^{i\theta}\in \bD \cap \mathbb{H}$ is called $A$-smooth if 
	\[
	\|p\theta/\pi \|_{\bR/\bZ} \ge A/d_n(z) \qquad \forall p\in [1,A+1] \cap \bZ \, ,
	\]
	where as before $d=d_n(z) = \min\{n,(1-r)^{-1}\}$. In this section we shall prove Lemma~\ref{lemma:small_ball_probability_smooth_points_bound_single_point}, which states that for a $d^{7\tau}$-smooth point $z$, the optimal small-ball probability bound for the random vector $d^{-1/2} \big(f_n(z),d^{-1} f_n^\prime(z)\big)$ is true up to the scale $d^{-3/2+o(1)}$. Our treatment here is largely inspired by~\cite[Section~4]{Konyagin-Schlag}, but we have adapted it to also cover the case $r<1$. 
	\subsection{Moving to the Fourier side}
	The starting point for the proof is the ``exponentially tilted" Esseen inequality (Proposition~\ref{propsition:exponentially_tilded_esseen_ineq}), which we use exactly as we used in~\eqref{eq:proof_of_small_ball_away_from_real_axis_after_esseen}. To recall, it gives the bound
	\begin{multline*}
		\max_{\alpha,\beta\in \bC} \bP \Big[|f_n(z)-\alpha| \le a d^{1/2} \, , \, |f_n^\prime(z)-\beta| \le b d^{3/2} \, , \, |f_n^{\prime \prime}(z)| \geq t d^{5/2} \Big] \\ \leq C e^{-c t^2} (ab)^2 \int_{\bR^4} \Bigg(\prod_{k=0}^{d} \varphi_{\eta_0} \Big( d^{-1/2} \Big\langle \xi , \begin{pmatrix}
			z^k \\ \tfrac{k}{d} z^k
		\end{pmatrix} \Big\rangle \Big) \Bigg)^{1/4} e^{-c\delta^2|\xi|^2} \, {\rm d}m(\xi)  
	\end{multline*}
	where now $\delta \ge d^{-3/2 + 11\tau}$. Our goal now is to show that the integral on the right-hand side of the above display is uniformly bounded for $d$ large. To this end, we fix some small constant $\eps_0>0$ (depending only on the distribution of $\xi_0$) and split the integral into three different regimes
	\begin{align*}
		I_1 &= \int_{\{|\xi| \le \eps_0 \sqrt{d} \}}\Bigg(\prod_{k=0}^{d} \varphi_{\eta_0} \Big( d^{-1/2} \Big\langle \xi , \begin{pmatrix}
			z^k \\ \tfrac{k}{d} z^k
		\end{pmatrix} \Big\rangle \Big) \Bigg)^{1/4} e^{-c\delta^2|\xi|^2} \, {\rm d}m(\xi) \, , \\ I_2 &= \int_{\{\eps_0\sqrt{d} \le |\xi| \le d^{3/2-10\tau} \}}\Bigg(\prod_{k=0}^{d} \varphi_{\eta_0} \Big( d^{-1/2} \Big\langle \xi , \begin{pmatrix}
			z^k \\ \tfrac{k}{d} z^k
		\end{pmatrix} \Big\rangle \Big) \Bigg)^{1/4} e^{-c\delta^2|\xi|^2} \, {\rm d}m(\xi) \, , \\ I_3 &= \int_{\{ |\xi| \ge  d^{3/2-10\tau} \}}\Bigg(\prod_{k=0}^{d} \varphi_{\eta_0} \Big( d^{-1/2} \Big\langle \xi , \begin{pmatrix}
			z^k \\ \tfrac{k}{d} z^k
		\end{pmatrix} \Big\rangle \Big) \Bigg)^{1/4} e^{-c\delta^2|\xi|^2} \, {\rm d}m(\xi) \, .
	\end{align*}
	We can handle the integral $I_1$ in the exact same way as we did in the proof of Claim~\ref{claim:small_ball_away_from_real_axis} (specifically, we use the bound~\eqref{eq:proof_of_small_ball_away_from_real_axis_bound_on_Gaussian_part}), where we note that any smooth point $z$ must lie at distance at least $d^{-1}$ away from the real axis. Handling $I_3$ is also easy, and follows simply from our lower bound on $\delta$ as
	\begin{equation}
		\label{eq:small_ball_smooth_point_bound_on_I_3}
		I_3 \le \int_{|\xi| \ge d^{3/2-10\tau}} e^{-c|\xi|^2 \delta^2} \, {\rm d}m(\xi) \lesssim d^{O(1)} \exp\Big(-cd^{2\tau}\Big) = o(1) \, .
	\end{equation}
	Therefore, to prove Lemma~\ref{lemma:small_ball_probability_smooth_points_bound_single_point}, it remains to deal with the integral $I_2$, which is the main technical part of this section. For $u=(u_1,u_2)\in \bC^2$, we denote by
	\begin{equation}
		\label{eq:def_of_psi_k_u}
		\psi_k(u) = \text{Re} \Big(z^k \big(\overline u_1 + \tfrac{k}{d} \overline u_2\big)\Big) \, ,
	\end{equation}
	and note that
	\[
	d^{-1/2} \Big\langle \xi , \begin{pmatrix}
		z^k \\ \tfrac{k}{d} z^k
	\end{pmatrix} \Big\rangle = \psi_k(u)
	\] 
	with $u= \xi/\sqrt{d}$. The following lemma is proved in Section~\ref{section:finite_diffrentiation_single_point}. 
	\begin{lemma}
		\label{lemma:sum_of_norm_psi_k_intermidiate_range_is_large_single_point}
		For all $\eps_0,c_1,c_2>0$ and for any $z\in \bD\cap \mathbb{H}$ which is $d^{7\tau}$-smooth the following holds. For any $u\in \bC^2$ with $\eps_0\le |u|\le d^{1-10\tau}$, we have that
		\[
		\sum_{k=0}^{d} \| s \psi_k(u) \|_{\bR/\bZ}^2 \ge d^{\tau} \, ,
		\]
		for all $s\in[c_1,c_2]$ and for all $d\ge d_0(\eps_0,c_1,c_2)$ large enough.
	\end{lemma}
	Assuming Lemma~\ref{lemma:sum_of_norm_psi_k_intermidiate_range_is_large_single_point} for the moment, we can conclude the proof of Lemma~\ref{lemma:small_ball_probability_smooth_points_bound_single_point}.
	\begin{proof}[Proof of Lemma~\ref{lemma:small_ball_probability_smooth_points_bound_single_point}]
		In view of the discussion above, to conclude the lemma we only need to show that the integral $I_2$ is uniformly bounded for large $d$. Indeed, as $\eta_0$ is a mean-zero non-degenerate symmetric random variable, there exists constant $c_1,c_2>0$ so that
		\begin{equation}
			\label{eq:bound_of_char_of_symmetric_random_variable_distance_to_integers}
			\varphi_{\eta_0}(t) = \bE\big[\cos (t\eta_0)\big]  \le \exp\big(-c\inf_{c_1\le s\le c_2} \|st\|_{\bR/\bZ}^2\big) \, , \qquad \forall t\in \bR \, .
		\end{equation}
		See for instance~\cite[Eq.~(9.2)]{Cook-Nguyen}. Plugging this into the definition of $I_2$, we see that the proof will follow once we show that
		\begin{equation}
			\label{eq:small_ball_smooth_point_I_2_bound_reduction}
			\int_{\{\eps_0\sqrt{d} \le |\xi| \le d^{3/2-10\tau} \}} \exp\left(-c\inf_{c_1\le s\le c_2} \sum_{k=0}^{d} \|s\psi_k(\xi/\sqrt{d})\|_{\bR/\bZ}^2 \right) \, {\rm d}m(\xi)  = O(1) \, .
		\end{equation}
		As we assume $z$ is $d^{7\tau}$-smooth, we conclude from Lemma~\ref{lemma:sum_of_norm_psi_k_intermidiate_range_is_large_single_point} that the integral in~\eqref{eq:small_ball_smooth_point_I_2_bound_reduction} is bounded from above by
		\[
		\int_{\{\eps_0\sqrt{d} \le |\xi| \le d^{3/2-10\tau} \}} \exp\Big(-cd^{\tau}\Big) \lesssim d^{O(1)} e^{-cd^{\tau}} = o(1)
		\]
		which in particular proves~\eqref{eq:small_ball_smooth_point_I_2_bound_reduction} and with that the lemma. 
	\end{proof}
	\subsection{Finite differences}
	\label{section:finite_diffrentiation_single_point}	
	As $u$ will remain fixed throughout this section, we will sometimes write $\psi_k = s\psi_k(u)$ to lighten the notation. Roughly speaking, to prove Lemma~\ref{lemma:sum_of_norm_psi_k_intermidiate_range_is_large_single_point} we will show that if the sequence $\{\psi_k\}$ is close to integers, then in fact many of these integers must be equal to 0. This observation, originating in~\cite{Konyagin-Schlag}, 
	relies heavily on the assumption that $z=re^{i\theta}$ is smooth, as the desired conclusion is false if $\theta$ lies close to rationals with small denominators. Recall that by Dirichlet's approximation theorem~\cite[Chapter~1]{Cassels}, there exists $p\in \bZ$ and $t\in\bR$ such that
	\begin{equation}
		\label{eq:dirichlet_principle}
		p\theta/2\pi - t\in \bZ \, , \qquad 1\le p\le \lfloor d^{\tau} \rfloor \, , \ |t|\le d^{-\tau} \, .
	\end{equation}
	Since $z$ is assumed to be $d^{7\tau}$-smooth (see Definition~\ref{def:smooth_angle}), this implies that $|t|\ge d^{-1+7\tau}$. We break down the proof of Lemma~\ref{lemma:sum_of_norm_psi_k_intermidiate_range_is_large_single_point} into three claims.
	
	\begin{claim}
		\label{claim:finite_differentiation_of_psi_k}
		Assume there exists an interval $\mathcal{J}\subset[0,d]$ with $|\mathcal{J}|\ge \tfrac{1}{2} d^{1-5\tau}$ such that
		\[
		\sup_{k\in \mathcal{J}}\|\psi_k\|_{\bR/\bZ} \le d^{-\tau} \, .
		\]
		Then for any arithmetic progression $\mathcal{P} = \{ k_0 + pj \}_{j\ge 0}$ with $k_0\in \cJ$ and $p\ge 1$ given by~\eqref{eq:dirichlet_principle} such that
		$
		\# \big(\mathcal{P} \cap \cJ\big) \ge cd^{1-5\tau},
		$ 
		there exists $m\in \bZ$ such that 
		\begin{equation*}
			|\psi_{s}-m|= \|\psi_{s}\|_{\bR/\bZ} \qquad \text{for all } s\in \mathcal{P}\cap \cJ \, .
		\end{equation*}
	\end{claim}
	\begin{proof}
		For a sequence $g:\bZ\to \bC$ and an integer $h\ge0$ we set
		\begin{equation}
			\label{eq:def_on_finite_difference}
			\big(\Delta_p^h g\big)(\ell) = \sum_{s=0}^h \binom{h}{s} (-1)^sg(\ell + sp) \, ,
		\end{equation}
		the discrete differential of order $h$, with step size $p$. We denote by $m_k$ the closest integer to $\psi_k$, then by our assumption
		\begin{equation}
			\label{eq:integer_clostest_to_psi_k}
			\sup_{k\in \cJ} |m_k - \psi_k| \le d^{-\tau} \, .
		\end{equation}
		We would like to bound the discrete differential along the arithmetic progression $\mathcal{P} = \{k_0 + jp\}_{j\ge 0}$. Recall that by writing $u=(u_1,u_2) \in \bC$ we have
		\begin{equation}
			\label{eq:psi_k_in_complex_coordinates}
			\psi_k = \text{Re}\Big(z^k \big(\overline{u}_1 + \frac{k}{d} \overline{u}_2\big)\Big) \, .
		\end{equation}
		Thus, for all $j\ge 0$, we have 
		\begin{align*}
			\left|\big(\Delta_p^h \psi\big)(k_0 + jp)\right| &= \left|\sum_{s=0}^h \binom{h}{s} (-1)^s \psi_{k_0 + jp +sp}\right|  \\ 
			&= \left|\text{Re} \, \left( z^{k_0+jp} \sum_{s=0}^h \binom{h}{s} (-1)^s z^{sp} \left(\overline{u_1} + \frac{k_0 + jp + sp}{d}\overline{u}_2\right)\right)\right| \\ & \le \left| \sum_{s=0}^h \binom{h}{s} (-1)^s z^{sp} \Big(\overline{u_1} + \frac{k_0 + jp + sp}{d}\overline{u}_2\Big) \right| \\ & \le |u_1| \Big|\sum_{s=0}^h \binom{h}{s} (-1)^s z^{sp} \Big| + |u_2| \Big|\sum_{s=0}^h \binom{h}{s} (-1)^s z^{sp} \frac{k_0 + jp + sp}{d} \Big| \, .
		\end{align*}
		Since 
		\begin{align*}
			&\sum_{s=0}^h \binom{h}{s} (-1)^s z^{sp} = (1-z^p)^h \qquad \text{and} \\ &\sum_{s=0}^h \binom{h}{s} (-1)^s z^{sp}sp = z^{-1} \partial_z\Big(z(1-z^p)^h\Big) = z^{-1}(1-z^p)^h -hp(1-z^p)^{h-1}z^{p-1} \, ,
		\end{align*}
		we get that
		\begin{equation*}
			|\big(\Delta_p^h \psi\big)(k_0 + jp)| \lesssim |u| h p |1-z^p|^h  \, .
		\end{equation*}
		On the other hand, by~\eqref{eq:dirichlet_principle} we see that
		\[
		|1-z^p| \le 1-r^p + r^p|1-e(2\pi t)| \lesssim |t|  + d^{-1+\tau} \lesssim d^{-\tau} 
		\]
		and since $|u| \le d^{1-10\tau}$, we arrive at
		\begin{equation}
			\label{eq:bound_on_differences_of_psi_order_h}
			|\big(\Delta_p^h \psi\big)(k_0 + jp)| \lesssim h 2^h  d^{1-9\tau-h\tau}  \, .
		\end{equation}
		On the other hand,~\eqref{eq:integer_clostest_to_psi_k} implies the trivial bound
		\begin{equation*}
			\big(\Delta_p^h (\psi-m) \big)(k_0 + jp) \le 2^h d^{-\tau} 
		\end{equation*}
		for all those $j\ge 0$ for which we have $[k_0+jp,k_0+jp+hp]\subset \cJ$. Combining with~\eqref{eq:bound_on_differences_of_psi_order_h}, we get that
		\begin{equation*}
			|\big(\Delta_p^h m \big)(k_0 + jp)| \le 2^h d^{-\tau} + C h 2^h  d^{1-9\tau-h\tau}  \, .
		\end{equation*}
		We can now choose $h=\lfloor \tau (\log d )/2 \rfloor$ and $d\ge d_0$ sufficiently large so that the right-hand side of the above is less that $\frac{1}{2}$, which implies that for all admissible $j\ge 0$
		\begin{equation*}
			\big(\Delta_p^h m \big)(k_0 + jp) = 0 \, .
		\end{equation*}
		That is, there exist a polynomial $Q$ with degree $\le h-1$ such that 
		\begin{equation}
			\label{eq:polynomial_equal_to_closest_integer}
			Q(j) = m_{k_0 + jp} \, .
		\end{equation}
		We will be done once we show that $Q$ is the constant polynomial. Suppose it is not, then $Q^\prime$ has at most $h-2$ zeros, and hence $Q$ must be strictly monotone on an interval of length $|\cJ|/(ph) \ge cd^{1-5\tau-o(1)}$. For all $k\in[0,d]$ we have
		\[
		|m_k| \le |\psi_k| + 1 \stackrel{\eqref{eq:psi_k_in_complex_coordinates}}{\le} |u| + 1\le d^{1-10\tau} + 1
		\]
		and hence we get a contradiction to strict monotonicity of $Q$. That is, $Q$ is a constant polynomial and in view of~\eqref{eq:polynomial_equal_to_closest_integer} we are done.
	\end{proof}
	\begin{claim}
		\label{claim:sign_changes_of_psi}
		Assume that for some $k_0\ge 0$ we have $|\psi_{k_0 + jp}(u)| \ge c|u|$ for all $0\le j \le cd^{1-5\tau}$. Then there exist $j,j^\prime \in[0,cd^{1-5\tau}]$ such that
		\[
		\psi_{k_0+jp} \cdot \psi_{k_0+j^\prime p} < 0 \, .
		\]
	\end{claim}
	\begin{proof}
		By scaling, we may assume without loss of generality that $|u|=1$. By~\eqref{eq:psi_k_in_complex_coordinates}, we have that
		\[
		\psi_{k_0+jp}(u) = \text{Re}\big(z^{k_0+jp}(\overline{u}_1 + \tfrac{k_0+jp}{d}\overline{u}_2)\big) = \text{Re}\big(z^{jp}(z^{k_0}\overline{u}_1 + z^{k_0}\tfrac{k_0}{d}\overline{u}_2)\big) + O(d^{-5\tau}) \, .
		\]
		Hence, assuming by contradiction there are no sign changes of $\{\psi_{k_0+jp}\}$, we get that for some $c>0$ and for all $0\le j\le cd^{1-5\tau}$
		\begin{equation}
			\label{eq:sign_changes_of_psi_lower_bound_on_inner_product}
			\text{Re}\big(z^{jp}(z^{k_0}\overline{u}_1 + z^{k_0}\tfrac{k_0}{d}\overline{u}_2)\big) > c \, .
		\end{equation}
		By~\eqref{eq:dirichlet_principle} we have $z^{jp} = r^{jp} e(2\pi jt)$, and because $z$ is smooth $|t|\ge d^{-1+6\tau}$. Therefore, the sequence $\{z^{jp}\}_{0\le j \le cd^{1-5\tau}}$ visits all four quadrants in the plane, and we get a contradiction to~\eqref{eq:sign_changes_of_psi_lower_bound_on_inner_product}.  
	\end{proof}
	\begin{claim}
		\label{claim:for_smooth_points_psi_are_not_aligned}
		Suppose that $\#\big\{ k\in[0,d] \cap \bZ \, : \, \|\psi_k \| \ge d^{-\tau} \big\} \le d^{3\tau}.$ Then there exists an interval $\cJ\subset[0,d]$ such that		
		\begin{align*}
			\sum_{k\in \mathcal{J}} |\psi_k|^2 \ge c d^{1-5\tau} |u|^2 \, \qquad \text{ and } \qquad  \sup_{k\in \mathcal{J}} |\psi_k| \le d^{-\tau}\, .
		\end{align*}
	\end{claim}
	\begin{proof}
		We decompose the interval $[0,d]$ into $\lfloor d^{5\tau}\rfloor$ non-overlapping intervals of lengths in $[\tfrac{1}{2}d^{1-5\tau}, d^{1-5\tau}]$. Denote by $J_\nu$ the $\nu$'th interval, $\nu\in\{1,\ldots,\lfloor d^{5\tau}\rfloor\}$ and by
		\[
		s_\nu = \sum_{k\in J_\nu} |\psi_k|^2 \, .
		\]
		By Lemma~\ref{lemma:singularity_properties_of_covariance_pointwise} we know that
		\[
		\sum_{\nu=1}^{\lfloor d^{4\tau}\rfloor} s_\nu \ge 10^{-4} \, |u|^2 d\, \, .
		\]
		We conclude there exists at least $d^{4\tau}$ intervals so that $s_\nu\ge 10^{-5} |u^2| d^{1-5\tau}$, as otherwise we would have 
		\[
		\sum_{\nu=1}^{\lfloor d^{4\tau}\rfloor} s_\nu \le d^{4\tau} d^{1-5\tau} |u|^2 + 10^{-5} |u|^2 d^{1-5\tau} (d^{5\tau} - d^{4\tau}) < 10^{-4}|u|^2 d
		\]
		which is a contradiction. By the pigeonhole principle, there exists an interval $\cJ\subset[0,d]$ of length at least $\tfrac{1}{2}d^{1-5\tau}$ such that
		\begin{align}
			\label{eq:for_smooth_points_psi_are_not_aligned_lower_bound_on_interval}
			\sum_{k\in \mathcal{J}} |\psi_k|^2 \ge 10^{-5} d^{1-5\tau} |u|^2 \qquad \text{ and } \qquad \sup_{k\in \mathcal{J}} \|\psi_k\| \le d^{-\tau}\, .
		\end{align}
		We want to show that $\cJ$ is the desired interval, i.e. that $\|\psi_k\| = |\psi_k|$ for $k\in \cJ$. To show this, cover $\cJ$ with $p$ arithmetic progressions $\mathcal{P}_1,\ldots,\mathcal{P}_{p}$ of step size $p\ge 1$ given by~\eqref{eq:dirichlet_principle}. Each such arithmetic progression is of length $\ge c d^{1-5\tau}$, and hence Claim~\ref{claim:finite_differentiation_of_psi_k} implies there exist integers $m_1,\ldots,m_{p}$ so that
		\[
		\sup_{k\in \mathcal{P}_{j}} |\psi_k - m_j| \le d^{-\tau} \, , \qquad 1\le j\le p.
		\]
		We claim that $m_j = 0$ for all $1\le j\le p$. Assume by contradiction this is not the case, and denote by $j^\prime$ the index for which $|m_{j^\prime}| \ge 1$ is maximal. By~\eqref{eq:for_smooth_points_psi_are_not_aligned_lower_bound_on_interval} we have
		\begin{equation*}
			d^{1-5\tau} |m_{j^\prime}|^2\ge \frac{1}{2}\sum_{k\in \mathcal{P}_{j^\prime}} |\psi_k|^2  \ge \frac{1}{2}\sum_{j=1}^{p} \sum_{k\in \mathcal{P}_j} |\psi_k|^2 \ge \frac{1}{2}10^{-5} d^{1-5\tau} |u|^2
		\end{equation*}
		and hence $|m_{j^\prime}|\ge c|u|^2$. In particular, this implies that $|\psi_{k}| \ge c|u|$ for all $k\in \mathcal{P}_{j^\prime}$, and hence by Claim~\ref{claim:sign_changes_of_psi} $\{\psi_k\}$ must change signs somewhere along the arithmetic progression $\mathcal{P}_{j^\prime}$. This is a contradiction to the fact that $|m_{j^\prime}| \ge 1$, and hence $m_1=\ldots=m_p = 0$, as desired.
	\end{proof}
	We are ready to prove the main result of this section.
	\begin{proof}[Proof of Lemma~\ref{lemma:sum_of_norm_psi_k_intermidiate_range_is_large_single_point}]
		The proof splits into two cases. In case $\#\big\{ k\in[0,d] \cap \bZ \, : \|\psi_k \| \ge d^{-\tau} \big\} > d^{3\tau}$, we clearly have
		\[
		\sum_{k=0}^{d} \|  \psi_k  \|^2 \ge d^{-2\tau} d^{3\tau}  = d^{\tau}.
		\]
		The other case is when $\#\big\{ k\in[0,d] \cap \bZ \, :  \|\psi_k \| \ge d^{-\tau} \big\} \le d^{3\tau}$, which by Claim~\ref{claim:for_smooth_points_psi_are_not_aligned} implies there exists an interval $\cJ\subset[0,d]$ such that
		\begin{align*}
			\sum_{k\in \cJ} |\psi_k|^2 \ge cd^{1-5\tau}|u|^2 \ge cd^{1-5\tau} \, \qquad \text{ and }\qquad \sup_{k\in \cJ} |\psi_k| \le d^{-\tau}\, .
		\end{align*}
		As $\tau>0$ is sufficiently small, we get that
		$$
		\sum_{k=0}^{d} \|  \psi_k  \|^2 \ge \sum_{k\in \cJ} |\psi_k|^2 \ge cd^{1-5\tau} \ge d^{\tau}
		$$
		for all $d$ large enough, and we are done.
	\end{proof}
	
	\section{Gaussian comparison for smooth separated tuples}
	\label{sec:gaussian_comparison_for_tuples}
	The goal of this section is to prove Theorem~\ref{thm:small_ball_comparison_to_Gaussian_for_tuples}, that is, to show that for smooth, separated tuples of points in $\Omega_K$ the probabilities of small-ball events at scale $\gtrsim n^{-M}$ (where $M>0$ is arbitrary but fixed), are the same as in the case of Gaussian coefficients. Our derivation here is largely inspired by~\cite{Cook-Nguyen}, which proved a similar result (see in particular~\cite[Theorem~3.2]{Cook-Nguyen}) for net points exactly on the unit circle (which corresponds to taking $K=0$, instead of large but fixed constant) and considering the polynomial and its first derivative, whereas in our derivation we need also to consider the second derivative. However, these extensions are rather of technical nature, and the core of the proof remains the same. 
	
	Recall that for a fixed $m\ge 1$ and $\mathbf{z} = (z_1,\ldots,z_m) \in \Omega_K^m$ we consider the random vector $S_n(\mathbf{z}) \in \bC^{3m} \simeq \bR^{6m}$ as given by~\eqref{eq:def_of_S_n_z}. First, we derive a convenient ``random walk" representation of this random vector. For a tuple $\mathbf{z} = (z_1,\ldots,z_m) \in \Omega_K^m$ and $0\le k\le n$ we set
	\begin{equation*}
		\mathbf{z}^k = \big(z_1^k , \ldots, z_m^{k}\big) \in \bC^m \, ,
	\end{equation*}
	and 
	\begin{equation}
		\label{eq:def_of_w_k}
		w_k = w_k(\mathbf{z}) = \frac{1}{\sqrt{n}}\Big(\mathbf{z}^k,\frac{k}{n}\mathbf{z}^{k},\frac{k(k-1)}{n^2}\mathbf{z}^{k}\Big) \in \bC^{3m}.
	\end{equation}
	Let $W_n=W_n(\mathbf{z})$ be a $(n+1)\times (3m)$ matrix with rows $w_k$ for $0\le k\le n$. We have
	\begin{equation}
		\label{eq:S_n_z_as_a random_walk}
		S_n(\mathbf{z}) = \sum_{k=0}^{n} \xi_k w_k(\mathbf{z}) = W_n^{{\sf T}} {\pmb \xi} 
	\end{equation} 
	where ${\pmb \xi} = (\xi_0,\ldots,\xi_n)$ is the vector of random coefficients. We note that the covariance matrix of $S_n(\mathbf{z})$ is given by $W_n^{{\sf T}} W_n$, as shown clearly from the representation~\eqref{eq:S_n_z_as_a random_walk}.
	\subsection{Non-singularity of covariance matrix for tuples}
	As a first step towards the proof of Theorem~\ref{thm:small_ball_comparison_to_Gaussian_for_tuples} we need to prove that $S_n(\mathbf{z}) \in \bC^{3m} \simeq \bR^{6m}$ is genuinely $6m$-dimensional. This amounts to showing that the minimal eigenvalue of the covariance matrix $W_n^{{\sf T}} W_n$ is uniformly bounded away from zero, under the necessary assumption that the tuple $\mathbf{z} = (z_1,\ldots,z_m)$ is spread, in the sense of Definition~\ref{def:spread_tuples}.
	\begin{lemma}
		\label{lemma:non-singularity_for_covariance_of_spread_tuple}
		For any $m\ge 1$, $\kappa>0$ and for any tuple $\mathbf{z} \in \Omega_K^m$ of $n^{\kappa}$-smooth and $\gamma$-spread points we have
		\[
		\min_{\substack{u_1,u_2\in \bC^{3m} \\[0.1em] \|u_1\|^2 + \|u_2\|^2=1}} \|W_n(\mathbf{z}) u_1 + \overline{W_n(\mathbf{z})} u_2  \|^2 \gtrsim_{m} \min\{\gamma,1\}^{8m-4} 
		\]
	\end{lemma}
	For the proof of Lemma~\ref{lemma:non-singularity_for_covariance_of_spread_tuple}, it will be slightly more convenient to work on a scale on which the points $(z_1,\ldots,z_m)$ are of order one separation. The next claim shows that such a rescaling is possible.  
	\begin{claim}
		\label{claim:the_L_rescaling_for_tuples}
		Let $z_1,\ldots,z_m\in \Omega_K^n$ be $\gamma$-spread for some $\gamma>0$, and let $\{a_n\}$ be any positive sequence of integers with $a_n = o(n)$ as $n\to \infty$. Then for all $n$ large there exists an integer $L \asymp n/a_n$ so that 
		\[
		\qquad \qquad |z_j^L - z_{j^\prime}^L| \gtrsim_{m} \gamma/a_n \, , \qquad \forall 1 \le j < j^\prime \le m \, .
		\]
	\end{claim}
	\begin{proof}
		To prove this claim, it will be more convenient to parameterize $z_j = e(w_j/n)$, where $w_j = x_j + i y_j$ and $x_j \in [0,n\pi]$ and $y_j \in [-K,K]$. Our assumption that $(z_1,\ldots,z_m)$ is $\gamma$-spread implies that for each $j<j^\prime$ we have
		\[
		|x_j - x_{j^\prime}| \ge \gamma/2 \qquad \text{or } \qquad |y_j - y_{j^\prime}| \ge \gamma/2 \, .
		\]
		We call a pair $(j,j^\prime)\in [m]^2$ \emph{good} if $|x_j - x_{j^\prime}| \ge \gamma/2$, and otherwise we say that the pair $(j,j^\prime)$ is \emph{bad}. By~\cite[Lemma~2.8]{Cook-Nguyen} there exists $L \asymp n/a_n$ so that
		\[
		\Big| \frac{L(x_j - x_{j^\prime})}{n} \Big| \gtrsim_m \gamma/a_n
		\]  
		for all good pairs $(j,j^\prime)$. In particular, for these pairs we have 
		\begin{align*}
			|z_j^L - z_{j^\prime}^L|  &= | e(w_j L/n) - e(w_{j^\prime} L/n ) | \\ &\ge \min\big\{e^{-y_jL/n}, e^{-y_{j^\prime} L/n} \big\} \, \big|e(x_j L/n) - e(x_{j^\prime} L/n)\big| \ge e^{-K/a_n}  \Big| \frac{L(x_j - x_{j^\prime})}{n} \Big| \gtrsim_m \gamma/a_n \, .
		\end{align*}
		To conclude the claim, we note that for all bad pairs $(j,j^\prime)$ we must have $|y_j - y_{j^\prime}| \ge \gamma/2$, which in turn implies the lower bound
		\begin{equation*}
			|z_j^L - z_{j^\prime}^L| \ge \big| |z_j|^L - |z_{j^\prime}|^L \big| \ge L \frac{|y_j - y_{j^\prime}|}{n} \gtrsim \gamma/a_n
		\end{equation*}
		and we are done.
	\end{proof}
	\begin{proof}[Proof of Lemma~\ref{lemma:non-singularity_for_covariance_of_spread_tuple}]
		Without loss of generality we can assume that $\gamma\in(0,1)$. Fix $u_1,u_2\in \bC^{3m}$ and write $u_j = (u_{j,1},\ldots,u_{j,3m})$ for $j=1,2$. The $k$'th entry of $W_n(\mathbf{z}) u_1$ is given by
		\[
		\big(W_n(\mathbf{z}) u_1\big)_k = \langle w_k(\mathbf{z}), u_1 \rangle = \frac{1}{\sqrt{n}}\sum_{j=1}^{m} z_j^k \Big(\overline{u}_{1,j} + \frac{k}{n} \overline{u}_{1,j+m} + \frac{k(k-1)}{n^2} \overline{u}_{1,j+2m}\Big) \, . 
		\]
		Therefore, we are after the lower bound
		\begin{multline}
			\label{eq:proof_of_lemma:non-singularity_for_covariance_of_spread_tuple_what_we_want}
			\|W_n(\mathbf{z}) u_1 + \overline{W_n(\mathbf{z})} u_2  \|^2 \\ = \frac{1}{n}\sum_{k=0}^{n} \bigg|\sum_{j=1}^m z_j^k \Big(\overline{u}_{1,j} + \frac{k}{n}\overline{u}_{1,j+m} + \frac{k(k-1)}{n^2} \overline{u}_{1,j+2m}\Big) + \overline{z}_j^k \Big(\overline{u}_{2,j} + \frac{k}{n}\overline{u}_{2,j+m} + \frac{k(k-1)}{n^2} \overline{u}_{2,j+2m}\Big)\bigg|^2 \\ \gtrsim \gamma^{8m-4} \, ,
		\end{multline}
		uniformly over all possible choices of $u_1,u_2\in \bC^{3m}$ which satisfy $|u_1|^2 + |u_2|^2 = 1$. 
		By Claim~\ref{claim:the_L_rescaling_for_tuples} (applied to the $\gamma$-spread tuple $(z_1,\ldots,z_m,+1,-1)$), there exists $L = c(m) n$ with $c(m)>0$ being a small constant so that 
		\begin{equation}
			\label{eq:proof_of_lemma:non-singularity_for_covariance_of_spread_tuple_application_of_rescaling}
			\qquad \qquad |z_j^L - z_{j^\prime}^L| \gtrsim_{m} \gamma \, , \qquad \forall 1 \le j < j^\prime \le m \, ,
		\end{equation}
		and
		\[
		\qquad \qquad \min\{|z_j^{L} + 1| , |z_j^{L} - 1|\} \gtrsim_{m} \gamma \, , \qquad \forall 1\le j\le m \, .
		\]
		We denote by $P$ the intersection of $[0,n]$ with the arithmetic progression $\{jL \, : \, j\in \bZ\}$, and note that $|P|\asymp_m 1$. With a slight abuse of notation, let $W_P$ be the sub-matrix of $W_n$ with rows indexed by $P$. Towards proving~\eqref{eq:proof_of_lemma:non-singularity_for_covariance_of_spread_tuple_what_we_want}, we first show that for all admissible choices $u_1,u_2\in \bC^{3m}$ we have
		\begin{equation}
			\label{eq:proof_of_lemma:non-singularity_for_covariance_of_spread_tuple_what_we_want_with_arithmetic_progression}
			\|W_P(\mathbf{z}) u_1 + \overline{W_P(\mathbf{z})} u_2  \|^2 \gtrsim_{m} \frac{\gamma^{6m-3}}{n} \, .
		\end{equation}
		The method of proof is similar to that of the proof of Lemma~\ref{lemma:singularity_properties_of_covariance_pointwise}. Indeed, for $\zeta\in \bC$ we introduce the third-order twisted discrete differential operator
		\begin{equation}
			\label{eq:def_third_order_twisted_diff_operator}
			(D_\zeta f)(k) = \sum_{t=0}^{3} \binom{3}{t} (-1)^t \zeta^{-tL} f(k+tL) \, , 
		\end{equation}
		(not to be confused with the second-order difference given by~\eqref{eq:definition_of_second_order_twisted_differntial}) acting on sequences $f:P \to \bC$. For the sequences
		\[
		a_z(k) = z^k \, , \qquad b_z(k) = \frac{k}{n}z^{k} \, , \qquad c_z(k) = \frac{k(k-1)}{n^2} z^k \, ,
		\] 
		a simple computation shows that
		\begin{align}
			\label{eq:computaion_of_third_order_twissted_differential_for_sequences_a_b_c} \nonumber
			(D_\zeta a_z)(k) &= \big(1-(z/\zeta)^L\big)^3 a_z(k) \, , \\ (D_\zeta b_z)(k) &= \big(1-(z/\zeta)^L\big)^3 \Big[b_z(k) + a_z(k)\alpha_L(z/\zeta)\Big]\, , \\ \nonumber (D_\zeta c_z)(k) &= \big(1-(z/\zeta)^L\big)^3 \Big[c_z(k) + 2b_z(k) \alpha_L(z/\zeta) + a_z(k) \Big(\beta_L(z/\zeta) - \frac{\alpha_L(z/\zeta)}{n}\Big)\Big] \, ,
		\end{align}
		where
		\[
		\alpha_L(s) = -3\frac{L}{n} \frac{s^L}{1-s^L} \qquad \text{and } \qquad \beta_L(s) = 3s^L\frac{L^2}{n^2} \frac{3s^L -1}{(1-s^L)^2} \, .
		\]
		From~\eqref{eq:computaion_of_third_order_twissted_differential_for_sequences_a_b_c}, we see that for all $z\in \bC\setminus\{0\}$ and for all $k\in[0,n]\cap \bZ$ we have
		\[
		(D_z a_z) (k) = (D_z b_z) (k)= (D_z c_z) (k) = 0 \, .
		\] 
		The key point about the factors $(1-(\cdot)^L)^3$, $\alpha_L(\cdot)$ and $\beta_L(\cdot)$ is that they do not depend on $k$, and hence commute with the difference operator $D_\zeta$. 
		
		To prove the lower bound~\eqref{eq:proof_of_lemma:non-singularity_for_covariance_of_spread_tuple_what_we_want_with_arithmetic_progression}, we cover the complex sphere in $\bC^{6m}$ by $2m$ pieces as
		\[
		\Big\{(u_1,u_2) \in \bC^{6m} \, : \, |u_{j,r}|^2 + |u_{j,r+m}|^2 + |u_{j,r+2m}|^2 \ge 1/2m \Big\} \, , \qquad j\in\{1,2\}, \ 1\le r \le m\, ,
		\]
		and prove the desired lower-bound on each piece separately. We will assume from now on that
		\begin{equation*}
			|u_{1,m}|^2 + |u_{1,2m}|^2 + |u_{1,3m}|^2 \ge 1/2m \, .
		\end{equation*}
		as all other cases will follow from the same argument. Denote by $D:\bC^{|P|} \to \bC^{|P|}$ the matrix associated to the linear operator
		\[
		D = D_{z_1} \circ \ldots \circ D_{z_{m-1}} \circ D_{\overline z_1} \circ \ldots \circ D_{\overline z_{m-1}} \circ D_{\overline z_m}\, .
		\]
		By~\eqref{eq:computaion_of_third_order_twissted_differential_for_sequences_a_b_c}, for all $k\in P$ and for all $1\le r \le m-1$ we have that
		\[
		(D a_{z_r}) (k) = (D b_{z_r}) (k) = (D c_{z_r}) (k)= (D a_{\overline z_r}) (k) = (D b_{\overline z_r}) (k) = (D c_{\overline z_r}) (k)= 0 \, . 
		\]
		For $r=m$,~\eqref{eq:computaion_of_third_order_twissted_differential_for_sequences_a_b_c} implies that 
		\[
		(D a_{\overline z_m}) (k) = (D b_{\overline z_m}) (k) = (D c_{\overline z_m}) (k)= 0 \, ,
		\]
		and
		\begin{align*}
			(D a_{z_m})(k) &= G\cdot a_{z_m}(k) \, , \\ (D b_{z_m})(k) &= G \cdot \big[b_{z_m}(k) + A \cdot a_{z_m}(k)\big]\, , \\ \nonumber (D c_{z_m})(k) &= G \cdot \big[c_{z_m}(k) + 2A \cdot b_{z_m}(k) + (B- A/n) \cdot a_{z_m}(k) \big] \, ,
		\end{align*}
		where
		\begin{align*}
			G &= \big(1-(z_m/\overline z_m)^L\big)^3 \prod_{r=1}^{m-1}\Big[\big(1-(z_m/z_r)^L\big)^3 \big(1-(z_m/\overline z_r)^L\big)^3  \Big] \, , \\ A &= \alpha_L(z_m/\overline z_m) + \sum_{r=1}^{m-1} \big[\alpha_L(z_m/z_r) + \alpha_L(z_m/\overline z_r)\big] \, , \\ B &= \beta_L(z_m/\overline z_m) + \sum_{r=1}^{m-1} \big[\beta_L(z_m/z_r) + \beta_L(z_m/\overline z_r)\big] \, . 
		\end{align*}
		We now want to use the above computations to lower bound the minimal singular value of the linear operator $D\circ(W_P + \overline{W}_P) : \bC^{6m} \to \bC^{|P|}$. For $k\in P$ such that $k+3mL \in P$, we have
		\begin{align*}
			(D W_P& u_1 + D \overline W_P u_2)_k \\ &= \frac{G}{\sqrt{n}} \cdot\Big(a_{z_m}(k) \overline u_{1,m} + \big(b_{z_m}(k) + A a_{z_m}(k)\big) \overline u_{2m} + \big(c_{z_m}(k) + 2Ab_{z_m}(k) + (B-A/n) a_{z_m}(k)\big)\overline u_{3m} \Big) \\ &= z_m^{k} \, \frac{G}{\sqrt{n}} \cdot \Big(\overline u_m + \frac{k}{n} \overline u_{2m} + \frac{k(k-1)}{n^2} \overline u_{3m} + A\big(\overline u_{2m} + \frac{2k-1}{n} \overline u_{3m}\big) + B \overline{u}_{3m}\Big) \, .
		\end{align*}
		Taking the modulus on both sides and summing, while using the fact that $z_m\in \Omega_K$, we see that
		\begin{align}
			\label{eq:proof_of_lemma:non-singularity_for_covariance_of_spread_tuple_what_we_want_with_arithmetic_progression_first_bound}
			\sum_{\substack{k\in P \\ k + 3mL\in P}} |(D W_P& u_1 + D \overline W_P u_2)_k|^2 \\ \nonumber &\ge e^{-K} \frac{|G|^2}{n} \sum_{\substack{k\in P \\ k + 3mL\in P}} \Big|\overline u_m + \frac{k}{n} \overline u_{2m} + \frac{k(k-1)}{n^2} \overline u_{3m} + A\big(\overline u_{2m} + \frac{2k-1}{n} \overline u_{3m}\big) + B \overline{u}_{3m}\Big|^2 \, .
		\end{align}
		Since $z_m$ is $n^\kappa$-smooth, it is separated from the real axis and our choice of $L\asymp_m n$ given by~\eqref{eq:proof_of_lemma:non-singularity_for_covariance_of_spread_tuple_application_of_rescaling} implies that $|G|^2 \gtrsim_m \gamma^{6m-3}$, and furthermore that
		\[
		\max\big\{|A|,|B|\big\}\lesssim m/\gamma \, .
		\]
		Since $|u_m|^2 + |u_{2m}|^2 + |u_{3m}|^2 \ge 1/2m$, similar considerations as in the proof of Claim \ref{cl:sum-AP-bounded} in the proof of Lemma~\ref{lemma:singularity_properties_of_covariance_pointwise} imply that the sum on the right-hand side of~\eqref{eq:proof_of_lemma:non-singularity_for_covariance_of_spread_tuple_what_we_want_with_arithmetic_progression_first_bound} is of size at least $|P| \asymp_m 1$, and hence
		\begin{equation*}
			\sum_{\substack{k\in P \\ k + 3mL\in P}} |(D W_P u_1 + D \overline W_P u_2)_k|^2 \gtrsim_{m} \frac{\gamma^{6m-3}}{n} \, .
		\end{equation*}  
		On the other hand, the operator norm of $D:\ell_2(P) \to \ell_2(P)$ is $O(m)$, so the left-hand side of the above is bounded by $\lesssim_m \|W_P u_1 + \overline W_P u_2  \|^2$, and~\eqref{eq:proof_of_lemma:non-singularity_for_covariance_of_spread_tuple_what_we_want_with_arithmetic_progression} is obtained.
		
		It remains to conclude~\eqref{eq:proof_of_lemma:non-singularity_for_covariance_of_spread_tuple_what_we_want}. Again, the argument is similar to the one presented in the proof of Lemma~\ref{lemma:singularity_properties_of_covariance_pointwise}, so we will keep it brief. Consider the sub-matrices $W_P,W_{1+P},\ldots,W_{n_0 + P}$ composed of rows indexed by the shifted arithmetic progressions $P,1+P,\ldots,n_0+P$. If $n_0 \le L$, then these matrices are all disjoint. Letting $F$ be the $3m$-dimensional diagonal matrix with diagonal entries $z_1,\ldots,z_m,z_1,\ldots,z_m,z_1,\ldots,z_m$, we note that $W_{k+P}$ and $W_{P} F^k$ differ by a matrix of norm $\lesssim_m k/n^{3/2}$ (as they only differ by a dilation by $k/n$ and $k(k-1)/n^2$ in some of the columns, and there is a factor $n^{-1/2}$ which multiplies everything). Since $F$ is of full rank, a simple application of the min-max principle shows that
		\[
		\sigma_{3m}(W_P F^k) \ge \sigma_{3m}(F^k) \cdot \sigma_{3m}(W_P)
		\] 
		where $\sigma_{3m}(\cdot)$ denotes here the minimal singular value. Since $z_1,\ldots,z_m\in \Omega_K$, we have $\sigma_{3m}(F^k) \gtrsim \exp(-K)$, and~\eqref{eq:proof_of_lemma:non-singularity_for_covariance_of_spread_tuple_what_we_want_with_arithmetic_progression} gives that
		\[
		\sigma_{3m}(W_P F^k) \gtrsim \frac{\gamma^{3m-3/2}}{\sqrt{n}} \, .
		\] 
		Take $n_0 = c_m \gamma^{2m-1} n$ for some $c_m>0$ sufficiently small, and note that the triangle inequality implies that $\sigma_{3m}(W_{k+P}) \gtrsim \gamma^{3m-3/2}/\sqrt{n}$ for all $1\le k\le n_0$. The exact same argument also applies to $\overline{W_{k+P}}$, and in fact also to the $6m$ block matrix 
		\[
		\begin{pmatrix}
			W_{k+P} & \mathbf{0} \\ \mathbf{0} & \overline W_{k+P}
		\end{pmatrix} \, .
		\]
		Since these progressions are all disjoint, we see that
		\[
		\|W_n(\mathbf{z}) u_1 + \overline{W_n(\mathbf{z})} u_2  \|^2 \ge \sum_{k=0}^{n_0} \|W_{k+P}(\mathbf{z}) u_1 + \overline{W_{k+P}(\mathbf{z})} u_2  \|^2 \gtrsim_{m} n_0 	\frac{\gamma^{6m-3}}{n} \gtrsim_m \gamma^{8m-4} \, ,
		\] 
		which proves~\eqref{eq:proof_of_lemma:non-singularity_for_covariance_of_spread_tuple_what_we_want} and we are done. 
	\end{proof}
	\subsection{Optimal small-ball bounds for tuples}
	The next thing on our agenda is to prove Proposition~\ref{prop:small_ball_estimate_for_tuples_of_smooth}. For that, the key input is Theorem~\ref{thm:decay_of_char_function_tuples_near_circle} below, which is analogous to Lemma~\ref{lemma:sum_of_norm_psi_k_intermidiate_range_is_large_single_point} from the $m=1$ case, and shows that the characteristic function of $S_n(\mathbf{z})$ decays very fast for $\mathbf{z}$ which is smooth and separated. This observation will also be key in the proof of Theorem~\ref{thm:small_ball_comparison_to_Gaussian_for_tuples}. 
	\begin{theorem}
		\label{thm:decay_of_char_function_tuples_near_circle}
		Let $\mathbf{z} = (z_1,\ldots,z_m)$ be a tuple of $n^\kappa$-smooth and $\gamma$-spread for some $\kappa\in(0,1)$ and $\gamma \ge n^{-{1/30m}}$. Then for any $K_\ast>0$ and $u\in \bC^{3m}$ with $n^{1/3} \le \|u\|\le n^{K_\ast}$ we have
		\[
		\Big|\bE\Big[e\big(\text{\normalfont Re}(S_n(\mathbf{z}) \, \overline{u})\big)\Big]\Big| \le \exp\big(-\log^2 n\big)
		\]
		for all sufficiently large $n$ depending on $m,\kappa,K_\ast$ and the distribution of $\xi$.
	\end{theorem} 
	The proof of Theorem~\ref{thm:decay_of_char_function_tuples_near_circle}, which only requires minor modifications from the analogous result~\cite[Theorem~3.1]{Cook-Nguyen}, is provided in Appendix~\ref{sec:decay_of_char_function_tuples} below. 
	\begin{proof}[Proof of Proposition~\ref{prop:small_ball_estimate_for_tuples_of_smooth}]
		By the Esseen inequality~\eqref{eq:classical_esseen_inequality}, we have that
		\begin{equation}
			\label{eq:proof_of_prop:small_ball_estimate_for_tuples_of_smooth_after_essen}
			\sup_{\zeta \in \bC^{3m}} \bP\Big[ S_n(\mathbf{z}) \in B(\zeta,\delta) \Big] \lesssim \delta^{6m} \int_{|u|\le \delta^{-1}} \Big|\bE\Big[e\big(\text{\normalfont Re}(S_n(\mathbf{z}) \, \overline{u})\big)\Big]\Big| \, {\rm d}m(u) \, .
		\end{equation}
		We need to bound the integral on the right-hand side of~\eqref{eq:proof_of_prop:small_ball_estimate_for_tuples_of_smooth_after_essen}. For that, we note that for $|u|\le c \sqrt{n}$ with $c>0$ sufficiently small we have
		\begin{align*}
			\Big|\bE\Big[e\big(\text{\normalfont Re}(S_n(\mathbf{z}) \, \overline{u})\big)\Big]\Big| &= \prod_{k=0}^{n} \Big|\bE\Big[e\big(\xi_k \text{Re}(w_k \, \overline u)\big)\Big]\Big| \\ &= \prod_{k=0}^{n} \Big|\phi\Big(\text{Re}(w_k \, \overline u)\Big)\Big| \\ & \le \exp\Big(-c \sum_{k=0}^{n} \big|\text{Re}(w_k \, \overline u)\big|^2\Big) \le \exp\Big(-c \sum_{k=0}^{n} \big|w_k \, \overline u + \overline{w}_k u\big|^2\Big) \, .
		\end{align*}
		Therefore, by Lemma~\ref{lemma:non-singularity_for_covariance_of_spread_tuple}, we have that
		\begin{equation*}
			\int_{|u|\le c\sqrt{n}} \Big|\bE\Big[e\big(\text{\normalfont Re}(S_n(\mathbf{z}) \, \overline{u})\big)\Big]\Big| \, {\rm d}m(u) \lesssim \int_{|u|\le c\sqrt{n}} \exp\Big(-c|u|^2 \gamma^{8m-4}\Big) {\rm d}m(u) \lesssim \gamma^{12m - 24 m^2} \, .
		\end{equation*}
		On the other hand, by Theorem~\ref{thm:decay_of_char_function_tuples_near_circle} we have
		\[
		\int_{ c\sqrt{n} \le |u| \le \delta^{-1} } \Big|\bE\Big[e\big(\text{\normalfont Re}(S_n(\mathbf{z}) \, \overline{u})\big)\Big]\Big| \, {\rm d}m(u) \lesssim n^{O(1)} \exp\big(-\log^2n\big) = o(1) \, ,
		\]
		which, in view of~\eqref{eq:proof_of_prop:small_ball_estimate_for_tuples_of_smooth_after_essen}, completes the proof of the proposition.
	\end{proof}	
	\subsection{Gaussian asymptotic for small ball probabilities}
	Here we will finally give the proof of Theorem~\ref{thm:small_ball_comparison_to_Gaussian_for_tuples}. For that, we will fix throughout the section a tuple $\mathbf{z} \in \Omega_K^m$ of $1$-spread and $n^\kappa$-smooth points, and a rectangle $Q\subset \bC^{3m} \simeq \bR^{6m}$ as in the statement of Theorem~\ref{thm:small_ball_comparison_to_Gaussian_for_tuples}, which we systematically write as
	\begin{equation}
		\label{eq:writing_Q_in_coordinates}
		Q = [a_1,b_1] \times \cdots \times [a_{3m} , b_{3m}] \, .
	\end{equation}
	We will always assume that the rectangle $Q$ is contained in some fixed compact set. Our goal is to show that the probability that $S_n(\mathbf{z})$ falls within the rectangle $Q$, is asymptotically the same as the probability that a Gaussian vector $\Gamma_n(\mathbf{z})$, with the same covariance structure, falls in $Q$. We will show that by following an idea from~\cite{Konyagin} (see also~\cite{Konyagin-Schlag}), which suggests to compare the characteristic functions of $S_n(\mathbf{z})$ and $\Gamma_n(\mathbf{z})$ in the relevant range. Indeed, we already established such a comparison (given by Lemma~\ref{lemma:CLT_on_fourier_side_bounded_third_moment} above), so we only need to show that all that remains in frequency space is negligible, which we do in what follows. 
	\begin{proof}[Proof of Theorem~\ref{thm:small_ball_comparison_to_Gaussian_for_tuples}]
		Since the tuple $\mathbf{z}$ is kept fixed throughout the proof, we will write for short $S_n =S_n(\mathbf{z})$ and $\Gamma_n = \Gamma_n(\mathbf{z})$. The proof will follow once we show that for all $n$ large enough
		\begin{equation}
			\label{eq:proof_of_thm:small_ball_comparison_to_Gaussian_for_tuples_what_we_want}
			\bigg|\int_{Q} {\rm d} S_n -  \int_{Q} {\rm d} \Gamma_n \bigg| \lesssim \frac{m(Q)}{\sqrt{n}} \, ,
		\end{equation}
		where, for a random vector $\mathbf{X}\in \bR^{6m}$, we denote by ${\rm d}\mathbf{X}$ its law.
		Throughout the proof we denote by $\eps = n^{-K_\ast}$, where $K_\ast = M+2$. For $1\le j\le 6m$, let $\phi_j:\bR\to [0,1]$ be a $C^\infty$-smooth function with the following properties\footnote{To show the existence of such function, simply take $\phi_j = \mathbf{1}_{[a_j-\eps/2,b_j+\eps/2]} \ast \psi_\eps$, where $\psi_\eps$ is a $C^\infty$ bump function supported in $[-\eps/4,\eps/4]$ with $\int_{\bR} \psi_\eps = 1$.}:
		\begin{enumerate}
			\item[(i)] $\phi_j(s) =1 $ for all $s\in[a_j,b_j]$ ; 
			\item[(ii)] $\phi_j(s)  = 0 $ for all $s\in \bR\setminus [a_j - \eps,b_j +\eps]$;
			\item[(iii)] $|\phi_j^{(k)} (s)| \lesssim_k \eps^{-k}$ for all $s\in \bR$ and for all $k\ge 1$ (uniformly in $j$). 
		\end{enumerate}
		Setting $\Phi:\bR^{6m} \to [0,1]$ as
		\begin{equation*}
			\Phi(x_1,\ldots,x_{6m}) = \prod_{j=1}^{6m} \phi_j(x_j) \, ,
		\end{equation*}
		we see that $\Phi$ is a smooth function such that $\Phi \equiv 1$ on $Q$ and $\text{supp}(\Phi) \subset Q^{+\eps}$. Let $\widehat{\Phi}$ be the Fourier transform of $\Phi$, normalized as
		\[
		\widehat \Phi(\xi) = \int_{\bR^{6m}} \Phi(x) \, \overline{e(\langle \xi ,x \rangle)} \, {\rm d}m(x) \, .
		\]
		Then, by combining item~(iii) with integration by parts, we have
		\begin{equation}
			\label{eq:bound_on_decay_of_fourier_transform_of_Phi_depnding_of_eps}
			\sup_{\xi \in \bR^{6m}} |\widehat \Phi(\xi)| \lesssim_{k,m}  \eps^{-k} (1+|\xi|^{k})^{-1}
		\end{equation}
		for all $k\ge 1$. To prove~\eqref{eq:proof_of_thm:small_ball_comparison_to_Gaussian_for_tuples_what_we_want}, it will by sufficient to show that
		\begin{equation}
			\label{eq:proof_of_thm:small_ball_comparison_to_Gaussian_for_tuples_what_we_want_for_smooth_function}
			\bigg|\int_{\bR^{6m}} \Phi \, {\rm d} S_n - \int_{\bR^{6m}}  \Phi\, {\rm d} \Gamma_n \bigg| \lesssim \frac{m(Q)}{\sqrt{n}} \, .
		\end{equation}
		Indeed, let $\mathcal{M}$ be a covering of $Q^{+\eps} \setminus Q$ with balls of radius $\eps$. Then by Proposition~\ref{prop:small_ball_estimate_for_tuples_of_smooth} we have
		\begin{align*}
			0 \le \int_{\bR^{6m}} \Phi \, {\rm d} S_n - \int_{\bR^{6m}} \mathbf{1}_Q \, {\rm d} S_n &\le  \bP\big[S_n(\mathbf{z}) \in \text{supp}(\Phi) \setminus Q\big] \\ & \le \sum_{B\in \mathcal{M}} \bP\big[S_n(\mathbf{z}) \in B \big] \\ &\lesssim  \sum_{B\in \mathcal{M}} m(B) \lesssim m\big(Q^{+\eps} \setminus Q\big) \lesssim \eps \max_{1\le j\le 6m} |b_j - a_j| \lesssim n^{-1/2} m(Q) \, .
		\end{align*}
		A similar bound holds when $S_n$ is replaced by $\Gamma_n$, so indeed~\eqref{eq:proof_of_thm:small_ball_comparison_to_Gaussian_for_tuples_what_we_want_for_smooth_function} implies~\eqref{eq:proof_of_thm:small_ball_comparison_to_Gaussian_for_tuples_what_we_want}.
		
		To prove~\eqref{eq:proof_of_thm:small_ball_comparison_to_Gaussian_for_tuples_what_we_want_for_smooth_function} we use the Parseval formula and see that 
		\[
		\int_{\bR^{6m}} \Phi \, {\rm d} S_n - \int_{\bR^{6m}}  \Phi\, {\rm d} \Gamma_n  = \frac{1}{2\pi} \int_{\bR^{6m}} \widehat{\Phi}(\xi) \Big(\widehat S_n(\xi) - \widehat{\Gamma}_n(\xi) \Big) \, {\rm d}m(\xi) 
		\]
		where $\widehat{S}_n,\widehat{\Ga}_n$ denote the characteristic functions of $S_n$ and $\Ga_n$. Since $\displaystyle |\widehat \Phi (\xi)| \le \int_{\bR^{6m}} \Phi \le 2 m(Q)$ for all $\xi\in \bR^{6m}$, we have that
		\begin{align}
			\label{eq:proof_of_thm:small_ball_comparison_to_Gaussian_for_tuples_what_we_want_for_smooth_function_split_integral}
			\nonumber
			\bigg|&\int_{\bR^{6m}} \Phi \, {\rm d} S_n - \int_{\bR^{6m}}  \Phi\, {\rm d} \Gamma_n \bigg| \\ &\lesssim \int_{\bR^{6m}} |\widehat{\Phi}(\xi)| \, \big|\widehat S_n(\xi) - \widehat{\Gamma}_n(\xi) \big| \, {\rm d}m(\xi) \\ \nonumber & \lesssim m(Q) \int_{|\xi| \le n^{1/3}} \big|\widehat S_n(\xi) - \widehat{\Gamma}_n(\xi) \big| \, {\rm d}m(\xi)   + \int_{|\xi| \ge n^{1/3}} |\widehat\Phi(\xi)||\widehat S_n(\xi)| \, {\rm d}m(\xi) + \int_{|\xi| \ge n^{1/3}} |\widehat\Phi(\xi)||\widehat \Ga_n(\xi)| \, {\rm d}m(\xi) \, .
		\end{align}
		By Lemma~\ref{lemma:non-singularity_for_covariance_of_spread_tuple}, we know that the covariance matrix of $S_n$ is uniformly non-singular (as we assume that the tuple $\mathbf{z}$ is $1$-spread). Thus, we can apply Lemma~\ref{lemma:CLT_on_fourier_side_bounded_third_moment} and get that
		\[
		\int_{|\xi| \le n^{1/3}} \big|\widehat S_n(\xi) - \widehat{\Gamma}_n(\xi) \big| \, {\rm d}m(\xi) \lesssim \frac{1}{\sqrt{n}} \int_{|\xi| \le n^{1/3}} e^{-c|\xi|^2} |\xi|^3 {\rm d}m(\xi) \lesssim \frac{1}{\sqrt{n}} \, .
		\]
		Furthermore, by Theorem~\ref{thm:decay_of_char_function_tuples_near_circle} we have
		\[
		\int_{|\xi| \ge n^{1/3}} |\widehat\Phi(\xi)||\widehat S_n(\xi)| \, {\rm d}m(\xi) \lesssim e^{-\log^2 n} \int_{|\xi| \ge n^{1/3}} |\widehat\Phi(\xi)| \, {\rm d}m(\xi) \stackrel{\eqref{eq:bound_on_decay_of_fourier_transform_of_Phi_depnding_of_eps}}{\lesssim} e^{-\log^2 n} \eps^{-O_{m}(1)}  \, , 
		\]
		and similarly
		\[
		\int_{|\xi| \ge n^{1/3}} |\widehat\Phi(\xi)||\widehat \Ga_n(\xi)| \, {\rm d}m(\xi) \lesssim e^{-cn^{2/3}} \int_{|\xi| \ge n^{1/3}} |\widehat\Phi(\xi)| \, {\rm d}m(\xi) \lesssim e^{-cn^{2/3}} \eps^{-O_m(1)}  \, .
		\]
		Plugging the above bounds into~\eqref{eq:proof_of_thm:small_ball_comparison_to_Gaussian_for_tuples_what_we_want_for_smooth_function_split_integral} gives
		\[
		\bigg|\int_{\bR^{6m}} \Phi \, {\rm d} S_n - \int_{\bR^{6m}}  \Phi\, {\rm d} \Gamma_n \bigg| \lesssim \frac{m(Q)}{\sqrt{n}} + e^{-\log^2 n} n^{O_{m,K_\ast}(1)}
		\]
		which shows that~\eqref{eq:proof_of_thm:small_ball_comparison_to_Gaussian_for_tuples_what_we_want_for_smooth_function} holds for all $n$ large enough, and we are done.
	\end{proof}
	\subsection{Approximating the event $A_z(U)$}
	We conclude this section with the proof of Lemma~\ref{lemma:for_spread_tuples_probabilities_of_smooth_points_factor} from Section~\ref{sec:poisson_limit_for_sum_over_the_smooth_net}. Recall that the events $A_z(U)$ and $A_z^\pm(U)$ are given by~\eqref{eq:def_of_A_z_U_with_pm_section_poisson_limit}. We want to show that for a tuples $\mathbf{z} = (z_1,\ldots,z_m) \in (\Omega_K)^m$ which is $n^{7\tau}$-smooth and $n^\beta$-spread the probability $\bP\big[\bigcap_{j=1}^m A_{z_j}(U)\big]$ asymptotically scales like the corresponding product of Gaussian probabilities. In view of Theorem~\ref{thm:small_ball_comparison_to_Gaussian_for_tuples} and Proposition~\ref{prop:small_ball_estimate_for_tuples_of_smooth} we just proved, this will follow once we show the event $A_z(U)$ can be well approximated by disjoint union of small rectangles, in which we have a local CLT for $S_n(\mathbf{z})$.  
	\begin{proof}[Proof of Lemma~\ref{lemma:for_spread_tuples_probabilities_of_smooth_points_factor}]
		For simplicity we will only consider the case $A_z(U)$; the proof for the events $A_z^+(U)$ and $A_z^-(U)$ follows the exact same steps with minor modifications. For the tuple $\mathbf{z}$ and the finite union of intervals $U\subset \bR_{\ge 0}$, we let $\mathbb{A} \subset \bC^{3m} \simeq \bR^{6m}$ be the domain such that the event
		\begin{equation*}
			\Big(\bigcap_{j=1}^{m} A_{z_j}(U) \Big) \bigcap \Big(\bigcap_{j=1}^{m} \{|f_n^{\prime\prime}(z_j)| \le n^{5/2} \log^2 n\}\Big)
		\end{equation*} 
		occurs if and only if $S_n(\mathbf{z}) \in \mathbb{A}$, where $S_n(\mathbf{z})$ is given by~\eqref{eq:def_of_S_n_z}. Then $\mathbb{A}$ is a compact Lipschitz domain, and we already observed in Section~\ref{sec:poisson_limit_for_sum_over_the_smooth_net} (see the proof of Lemma~\ref{lemma:for_almost_spread_tuples_probabilities_are_bounded_by_what_you_want}, and also the computations in Section~\ref{sec:limiting_intensity_for_gaussian} below) that 
		\[
		\Big(\frac{1}{n^{3/2 + 2\beta}}\Big)^m   \lesssim m_{6m}(\mathbb{A}) \lesssim  \Big(\frac{1}{n^{3/2 + 2\beta}}\Big)^m \log^{O(1)} n
		\]
		where here $m_{6m}$ denotes the Lebesgue measure on $\bR^{6m}$.  Therefore, there exists $M\ge 1$ large enough so that we can cover $\mathbb{A}$ with cubes from $n^{-M} \bZ^{6m}$ and have
		\begin{equation*}
			\sum_{\substack{Q\in n^{-M} \bZ^{6m} \\ Q \cap \partial \mathbb{A} \not= \emptyset}} m_{6m}(Q) \lesssim n^{-1} m_{6m} (\mathbb{A}) .
		\end{equation*}
		The triangle inequality shows that
		\begin{multline*}
			\Big|\bP\big[S_n(\mathbf{z}) \in \mathbb{A}\big] - \bP\big[\Ga_n(\mathbf{z}) \in \mathbb{A}\big] \Big| = \Big|\int_{\mathbb{A}} \big({\rm d}S_n(\mathbf{z}) - {\rm d}\Gamma_n(\mathbf{z})\big) \Big| \\ \le \sum_{\substack{Q\in n^{-M} \bZ^{6m} \\ Q\subset \mathbb{A}}} \Big|\bP\big[S_n(\mathbf{z}) \in Q\big] - \bP\big[\Ga_n(\mathbf{z}) \in Q\big] \Big| + \sum_{\substack{Q\in n^{-M} \bZ^{6m} \\ Q \cap \partial \mathbb{A} \not= \emptyset}}\Big\{ \bP\big[S_n(\mathbf{z}) \in Q\big] + \bP\big[\Ga_n(\mathbf{z}) \in Q\big]\Big\} \, .
		\end{multline*}
		By Theorem~\ref{thm:small_ball_comparison_to_Gaussian_for_tuples}, we know that
		\[
		\sum_{\substack{Q\in n^{-M} \bZ^{6m} \\ Q\subset \mathbb{A}}} \Big|\bP\big[S_n(\mathbf{z}) \in Q\big] - \bP\big[\Ga_n(\mathbf{z}) \in Q\big] \Big| \lesssim \frac{1}{\sqrt{n}} \sum_{\substack{Q\in n^{-M} \bZ^{6m} \\ Q\subset \mathbb{A}}} m_{6m}(Q) \le \frac{m_{6m}(\mathbb{A})}{\sqrt{n}} \, ,
		\]
		whereas by Proposition~\ref{prop:small_ball_estimate_for_tuples_of_smooth} we have
		\[
		\sum_{\substack{Q\in n^{-M} \bZ^{6m} \\ Q \cap \partial \mathbb{A} \not= \emptyset}}\Big\{ \bP\big[S_n(\mathbf{z}) \in Q\big] + \bP\big[\Ga_n(\mathbf{z}) \in Q\big]\Big\} \lesssim \sum_{\substack{Q\in n^{-M} \bZ^{6m} \\ Q \cap \partial \mathbb{A} \not= \emptyset}} m_{6m}(Q) \lesssim n^{-1} m_{6m} (\mathbb{A}) \, .
		\]
		Since we also have that
		\[
		\bP\Big[\Big(\bigcap_{j=1}^{m} \{|f_n^{\prime\prime}(z_j)| \le n^{5/2} \log^2 n\}\Big)^c\Big] \le \exp\Big(-c \log^2 n\Big)
		\]
		(this follows formally from Lemma~\ref{lemma:control_on_maximum_of_polynomial_and_derivatives}, but can also be derived directly from the sub-Gaussian assumption on the coefficients $\{\xi_k\}$), we conclude that 
		\begin{equation}
			\label{eq:proof_of_lemma:for_spread_tuples_probabilities_of_smooth_points_factor_moving_to_gaussian}
			\bP\big[A_{z_1}(U) \cap \ldots \cap A_{z_m}(U)\big] = \big(1+o(1)\big) \bP_{{\sf G}}\big[A_{z_1}(U) \cap \ldots \cap A_{z_m}(U)\big] \, .
		\end{equation}
		To finish the proof of the lemma, we need to further show that
		\begin{equation}
			\label{eq:proof_of_lemma:for_spread_tuples_probabilities_of_smooth_points_factor_gaussian_factorization}
			\bP_{{\sf G}}\big[A_{z_1}(U) \cap \ldots \cap A_{z_m}(U)\big] = \big(1+o(1)\big) \prod_{j=1}^{m} \bP_{{\sf G}}\big[A_{z_{j}}(U)\big]
		\end{equation}
		as $n\to \infty$. Indeed,~\eqref{eq:proof_of_lemma:for_spread_tuples_probabilities_of_smooth_points_factor_gaussian_factorization} follows from a standard argument for decorrelation of Gaussian fields. For every $n^\beta$-separated tuple $\mathbf{z}$, recall the matrix $W_n = W_n(\mathbf{z})$ given by~\eqref{eq:def_of_w_k} and that $W_n^{{\sf T}} W_n$ is the covariance matrix of $\Ga_n(\mathbf{z})$. Lemma~\ref{lemma:non-singularity_for_covariance_of_spread_tuple} implies that  
		\begin{equation}
			\label{eq:proof_of_lemma:for_spread_tuples_probabilities_of_smooth_points_factor_gaussian_factorization_covariance_matrix_factors}
			W_n^{{\sf T}} W_n = \begin{pmatrix}
				\Sigma(z_1) & 0 & \cdots & 0 \\ 0 & \Sigma(z_2) & \cdots & 0 \\ \vdots & \vdots & \ddots & \vdots \\ 0 & \cdots & 0 & \Sigma(z_m) 
			\end{pmatrix} + R 			
		\end{equation}
		where $\Sigma(z)$, which is given explicitly by~\eqref{eq:the_matrix_sigma} below, is the limiting covariance matrix of $\Ga_n(z)$ for $z\in \Omega_K$, and $R$ is a matrix with all entries $O(n^{-\beta})$. Since the matrices $W_n^{{\sf T}} W_n$ and $\{\Sigma(z_j)\}_{j=1}^{m}$ are uniformly non-singular as $n\to \infty$, the equality~\eqref{eq:proof_of_lemma:for_spread_tuples_probabilities_of_smooth_points_factor_gaussian_factorization_covariance_matrix_factors} implies that
		\begin{equation*}
			\big(W_n^{{\sf T}} W_n\big)^{-1} = \begin{pmatrix}
				\Sigma(z_1)^{-1} & 0 & \cdots & 0 \\ 0 & \Sigma(z_2)^{-1} & \cdots & 0 \\ \vdots & \vdots & \ddots & \vdots \\ 0 & \cdots & 0 & \Sigma(z_m)^{-1} 
			\end{pmatrix} + \widetilde{R} \, , 			
		\end{equation*}
		where all entries of $\widetilde{R}$ are $O(n^{-\beta})$. The asymptotic equality~\eqref{eq:proof_of_lemma:for_spread_tuples_probabilities_of_smooth_points_factor_gaussian_factorization} now follows from a routine Gaussian comparison for the densities, see for instance~\cite[Lemma~3.5]{Cook-Nguyen-Yakir-Zeitouni}. In view of~\eqref{eq:proof_of_lemma:for_spread_tuples_probabilities_of_smooth_points_factor_moving_to_gaussian}, the proof of the lemma is complete. 
	\end{proof}
	
	\section{Computing the limiting intensity}
	\label{sec:limiting_intensity_for_gaussian}
	In this section we will assume that the random coefficients of $f_n$ are i.i.d.\ standard Gaussian random variables, without repeating this assumption throughout. Recall that by $Z\sim N_{\bC}(\mu,\sigma^2)$ with $\mu\in\bC$ and $\sigma^2>0$ we mean that $Z$ is a complex-valued random variable with the density
	\begin{equation}
		\label{eq:general_complex_gaussian_density}
		\frac{1}{\pi \sigma^2} \exp\big(-|z-\mu|^2/\sigma^2\big) \, {\rm d}m(z) \, .
	\end{equation}
	The proof of Lemma~\ref{lemma:gaussian_asymptotic_for_a_single_probability} consists of a Gaussian small-ball computation, which involves $f_n(z),f_n^\prime(z),f_n^{\prime\prime}(z)$ evaluated at a point $z\in \Omega_K$. We start by recording some simple computations associated with the complex Gaussian density~\eqref{eq:general_complex_gaussian_density} that will be helpful throughout. By the Lebesgue differentiation theorem, we have 
	\begin{equation}
		\label{eq:lebesgue_diffrentiation_for_gaussian}
		\lim_{\eps \downarrow 0} \frac{\bP\big[Z\in \mathcal{V}_\eps\big]}{m(\mathcal{V}_\eps)}  = \frac{1}{\pi \sigma^2} e^{-|\zeta-\mu|^2/\sigma^2} 
	\end{equation}
	for any family of sets $\{\mathcal{V}_{\eps}\}$ that shrink nicely to a point $\zeta \in \bC$ as $\eps\downarrow 0$, see for instance~\cite[Theorem~3.21]{Folland}. We will also use the fact that, for a mean-zero complex Gaussian $Z$ and for all $c>0$, we have
	\begin{equation}
		\label{eq:computation_of_forth_moment_with_tilte}
		\bE\big[|Z|^4 e^{-c|Z|^2}\big] = \frac{1}{\pi \sigma^2}\int_{\bC} |z|^4 \, e^{-|z|^2(\sigma^{-2} + c)} \, {\rm d}m(z) = \frac{2\sigma^4}{(1+c\sigma^2)^3} \, .
	\end{equation}
	Throughout this section, it will be convenient to apply the change of variables $z = e^{w/n}$ and define
	\begin{equation}
		\label{eq:def_of_F_n_w}
		F_n(w) = \frac{1}{\sqrt{n}} f_n(e^{w/n}). 
	\end{equation}
	As $z\in \Omega_K$, the variable $w$ now parameterizes the rectangle 
	\begin{equation*}
		\widetilde \Omega_K = \Big[n\log\Big(1-\frac{K}{n}\Big),n\log\Big(1+\frac{K}{n}\Big)\Big] \times \big[0,n\pi\big] \subset \bC\, .
	\end{equation*}
	The event $A_z(U)$ given by~\eqref{eq:def_of_event_A_z_U_for_net_point}, for which we want to compute its probability, can be written in the new coordinate system as
	\begin{equation}
		\label{eq:A_z_U_after_change_of_coordinates}
		A_z(u) = \Big\{ \frac{F_n(w)}{F_n^\prime(w)}  - ne^{w/n} \in \widetilde{\mathsf{R}}_w \ , \  \frac{2|F_n^\prime(w)|}{|F_n^{\prime\prime}(w)|} \in n^{-1/4} U \, , \ |F_n^\prime(w)| \ge \frac{1}{n^{1/4} \log n} \Big\}
	\end{equation} 
	where $\widetilde{\mathsf{R}}_w$ is the rectangle $[-\frac{K}{M_1},\frac{K}{M_1}] \times [-\frac{n}{2M_2},\frac{n}{2M_2}]$ around the point $ne^{w/n}$.
	\subsection{Conditional Gaussian random variables}
	Given $w\in \widetilde \Omega_K$ which we write as $w= x+iy$ the random vector 
	\begin{equation}
		\label{eq:def_of_gaussian_vector_F_F'_F''_after_normalization}
		\big(F_n(w),F_n^\prime(w), F_n^{\prime\prime}(w)\big)
	\end{equation} 
	is a mean-zero Gaussian random vector in $\bR^6\simeq \bC^3$. We first recall some basic properties of the multivariate complex Gaussian distribution, as presented in~\cite{Picinbono}. A mean-zero complex Gaussian random vector $\mathbf{Z} = \mathbf{X} + i \mathbf{Y} \in \bC^k$  can be described via 2 parameters: 
	\begin{equation*}
		\Sigma = \bE\big[\mathbf{Z} \cdot \mathbf{Z}^{H}\big] \qquad \text{and } \qquad R = \bE\big[\mathbf{Z} \cdot \mathbf{Z}^{T}\big]
	\end{equation*} 
	where $\mathbf{Z}^T$ is the matrix transpose of $\mathbf{Z}$ and $\mathbf{Z}^H$ denotes the conjugate transpose. The matrix $\Sigma$ is Hermitian and non-negative is referred to as the \emph{complex covariance matrix}, while the symmetric matrix $R$ is referred to as the \emph{relation matrix}. For the complex vector~\eqref{eq:def_of_gaussian_vector_F_F'_F''_after_normalization}, the complex covariance matrix is given as
	\begin{equation}
		\label{eq:the_matrix_sigma}
		\Sigma = \frac{1}{n} \sum_{k=0}^{n} e^{2x\tfrac{k}{n}} \begin{pmatrix}
			1 & \frac{k}{n} & \frac{k^2}{n^2} \\ \frac{k}{n} & \frac{k^2}{n^2} & \frac{k^3}{n^3} \\ \frac{k^2}{n^2} & \frac{k^3}{n^3} & \frac{k^4}{n^4}
		\end{pmatrix} \\ = \int_{0}^{1} e^{2xt} \begin{pmatrix}
			1 & t & t^2 \\ t & t^2 & t^3 \\ t^2 & t^3 & t^4
		\end{pmatrix} {\rm d}t + O(n^{-1}) \, .
	\end{equation}
	Recall the definition~\eqref{eq:def_of_a_j} of the functions $a_j(x)$ from Section~\ref{sec:poisson_limit_for_sum_over_the_smooth_net}, given as
	\begin{equation*}
		a_j = a_j(x) = \int_{0}^{1} t^j e^{2xt} {\rm d}t \, .
	\end{equation*}
	The above takes the form of 
	\begin{equation*}
		\Sigma = \begin{pmatrix}
			a_0 & a_1 & a_2 \\  a_1 & a_2 & a_3 \\ a_2 & a_3 & a_4
		\end{pmatrix} + O(n^{-1})
	\end{equation*}
	uniformly for $w\in \widetilde \Omega_K$. On the other hand, the relation matrix for the complex vector~\eqref{eq:def_of_gaussian_vector_F_F'_F''_after_normalization} tends to zero as $n\to \infty$, uniformly for $z = e^{w/n}$ as in the assumptions of Lemma~\ref{lemma:gaussian_asymptotic_for_a_single_probability}. Indeed, the later remark follows from the simple relation
	\[
	\frac{1}{n} \sum_{k=0}^{n} \frac{k^{\ell}}{n^{\ell}} \, e^{2\tfrac{k}{n}w} = o(1) \, , \qquad \ell \in \{0,\ldots,4\}
	\]
	which in turn follows from our assumption that $\arg(z) \in [n^{-1+4\tau},n^{-1+4\tau}]$. To compute the probability of the event~\eqref{eq:A_z_U_after_change_of_coordinates} we will need to first state two Gaussian conditional laws, as described below. As always, to compute these Gaussian conditional laws we use the Schur complement of the covariance and relation matrices, see~\cite[Section~{\rm IV}]{Picinbono} for the relevant background.
	
	\vspace*{2mm}
	
	\noindent
	(A) $F_n^\prime(w)$ conditional on $F_n^{\prime\prime}(w)$ is a complex Gaussian random variable given by the density~\eqref{eq:general_complex_gaussian_density} with mean
	\[
	\mu = \frac{a_3}{a_4} F_n^{\prime\prime} (w) \, ,
	\]
	and variance
	\[
	\sigma^2 = \Big(a_2 - \frac{a_3^2}{a_4}\Big)\big(1+o(1)\big) \, .
	\]
	
	\vspace*{2mm}
	
	\noindent
	(B) $F_n(w)$ conditional on $\{F_n^\prime(w),F_n^{\prime\prime}(w)\}$ is a complex Gaussian random variable with mean
	\[
	\mu = \frac{1}{a_2a_4 - a_3^2} \begin{pmatrix}
		a_1 & a_2
	\end{pmatrix} \cdot \begin{pmatrix}
		a_4 & -a_3 \\ -a_3 & a_2
	\end{pmatrix} \cdot \begin{pmatrix}
		F_n^{\prime}(w) \\ F_n^{\prime\prime}(w) 
	\end{pmatrix} \, ,
	\]
	and variance
	\[
	\sigma^2 = a_0 - \frac{1}{a_2a_4 - a_3^2}\begin{pmatrix}
		a_1 & a_2
	\end{pmatrix} \cdot \begin{pmatrix}
		a_4 & - a_3 \\ -a_3 & a_2
	\end{pmatrix}\cdot \begin{pmatrix}
		a_1 \\ a_2
	\end{pmatrix} + o(1)  = \sigma_1^2 + o(1) \, ,
	\]
	where
	\begin{equation*}
		\sigma_1^2 = a_0 - \frac{a_1^2a_4 - 2a_1a_2a_3 + a_2^3}{a_2a_4 - a_3^2} \, .
	\end{equation*}
	\subsection{The intensity function $\mathfrak{F}$} 
	We wish to compute the leading order asymptotic for the small-ball event~\eqref{eq:A_z_U_after_change_of_coordinates}. Denote by
	\[
	\mathcal{E}_\ast = \Big\{ \max_{w\in \widetilde \Omega_K } |F_n^{\prime\prime}(w)| \le \log^3 (n)\Big\} \, ,
	\]
	and note that by Lemma~\ref{lemma:control_on_maximum_of_polynomial_and_derivatives} (applied both for $F_n(w)$ and $F_n(w^{-1})$) we have $\bP\big[\mathcal{E}_\ast^c\big] \le \exp(- c\log^2 n)$, uniformly for $K>0$ in a compact. 
	\begin{proof}[Proof of Lemma~\ref{lemma:gaussian_asymptotic_for_a_single_probability}]
		In view of the above, it will be enough to compute $\bP\big[A_z(U) \cap \mathcal{E}_\ast\big]$, which we start with computing the conditional probability
		\begin{multline*}
			\bP\Big[ \frac{F_n(w)}{F_n^\prime(w)}  - ne^{w/n} \in \widetilde{\mathsf{R}}_w  \mid \{F_n^\prime(w), F_n^{\prime\prime}(w)\}\Big] \\ = \bP\bigg[ F_n(w) \in F_n^\prime(w) \cdot \Big[-\frac{K}{M_1},\frac{K}{M_1}\Big] \times \Big[-\frac{n}{2M_2},\frac{n}{2M_2}\Big] \,  \mid \{F_n^\prime(w), F_n^{\prime\prime}(w)\} \bigg]\, .
		\end{multline*}
		Recall that $M_1 = \lceil 4K/(\delta n) \rceil$ and $M_2 = \lceil 4/\delta \rceil$, where $\delta = n^{-5/4-\beta}$. On the event $\big\{\frac{2F_n^\prime(w)}{F_n^{\prime\prime}(w)} \in n^{-1/4} U \big\} \cap \mathcal{E}_\ast$, we have $|F_n^\prime(w)| \lesssim n^{-1/4} \log^3(n)$ for all $w\in \widetilde{\Omega}_K$, and hence
		\[
		\max\Big\{|F_n^\prime(w)| M_1^{-1}, n|F_n^\prime(w)| M_2^{-1}\Big\}  \xrightarrow{n\to \infty} 0\, .
		\]
		We also note that by item (B) from the previous section the conditional mean of $F_n(w)$ given $\{F_n^\prime(w),F_n^{\prime\prime}(w)\}$ is 
		\[
		F_n^{\prime\prime}(w) \frac{a_2^2 - a_1a_3}{a_2a_4  - a_3^2} + o(1)\, .
		\]
		Denoting by $\mu_1 = \frac{a_2^2 - a_1a_3}{a_2a_4  - a_3^2}$, we see from~\eqref{eq:lebesgue_diffrentiation_for_gaussian} that
		\begin{align}
			\label{eq:gaussian_asymptotic_for_a_single_probability_first_conditioning} \nonumber
			\mathbf{1}_{\mathcal{E}_\ast}\bP\bigg[ F_n(w) \in F_n^\prime(w) &\cdot \Big[-\frac{K}{M_1},\frac{K}{M_1}\Big] \times \Big[-\frac{n}{2M_2},\frac{n}{2M_2}\Big] \,  \mid \{F_n^\prime(w), F_n^{\prime\prime}(w)\} \bigg] \\ &= 2n\cdot \frac{K}{M_1 M_2} \, |F_n^\prime(w)|^2 \frac{1}{\pi\sigma_1^2} \exp\Big(-|F_n^{\prime\prime}(w)|^2 \frac{\mu_1^2}{\sigma_1^2} \, \Big) \big(1+o(1)\big) \, .
		\end{align}
		Next, we want to compute the leading order asymptotic of the conditional expectation
		\[
		\bE\Big[|F_n^\prime(w)|^2 \,  \mathbf{1}_{\{|F_n^\prime(w)| \in \frac{|F_n^{\prime\prime}(w)|}{2} \, n^{-1/4} U \}} \, \mathbf{1}_{\{|F_n^\prime(w)| \ge \frac{1}{n^{1/4} \log n}\}} \mid F_n^{\prime\prime}(w)\Big] 
		\]
		as $n\to \infty$. Indeed, by item (A) from the previous section we see that 
		\begin{align*}
			\bE\Big[ & |F_n^\prime(w)|^2 \, \mathbf{1}_{\{|F_n^\prime(w)| \in \frac{|F_n^{\prime\prime}(w)|}{2} \, n^{-1/4} U \}} \, \mathbf{1}_{\{|F_n^\prime(w)| \ge \frac{1}{n^{1/4} \log n}\}} \mid F_n^{\prime\prime}(w)\Big] \\ &= \big(1+o(1)\big) \frac{a_4}{\pi\big(a_2a_4 - a_3^2\big)} \exp\Big(-|F_n^{\prime\prime}(w)|^2 \frac{a_3^2}{a_4(a_2a_4 - a_3^2)} \Big) \cdot \Big( \int_{\{|z| \in \frac{|F_n^{\prime\prime}(w)|}{2} \, n^{-1/4} U\}} |z|^2 \, {\rm d}m(z) \Big) \\ &= \big(1+o(1)\big) \frac{|F_n^{\prime\prime}(w)|^4}{2^4 n}\cdot \frac{a_4}{\pi\big(a_2a_4 - a_3^2\big)} \exp\Big(-|F_n^{\prime\prime}(w)|^2 \frac{a_3^2}{a_4(a_2a_4 - a_3^2)} \Big) \cdot \Big( \int_{\{|z'| \in U\}} |z'|^2 \, {\rm d}m(z')\Big) \\ &=  \big(1+o(1)\big) \frac{|F_n^{\prime\prime}(w)|^4}{2^3 n}\cdot \frac{a_4}{\big(a_2a_4 - a_3^2\big)} \exp\Big(-|F_n^{\prime\prime}(w)|^2 \frac{a_3^2}{a_4(a_2a_4 - a_3^2)} \Big) \cdot \Big(\int_{U} t^3 {\rm d}t\Big)\, ,
		\end{align*}
		where in the second equality we applied the change of variables $z = \frac{|F_n^{\prime\prime}(w)|}{2} \, n^{-1/4} z'$. Combining this computation with~\eqref{eq:gaussian_asymptotic_for_a_single_probability_first_conditioning}, we see that the law of total expectation yields		
		\begin{align}
			\label{eq:gaussian_asymptotic_for_a_single_probability_second_conditioning}
			\bE\Big[&\mathbf{1}_{A_z(U) \cap \mathcal{E}_\ast} \mid F_n^{\prime\prime}(w)\Big] \\ \nonumber &= \bE\Big[ \bE\big[\mathbf{1}_{A_z(U) \cap \mathcal{E}_\ast} \mid \big\{F_n^\prime(w), F_n^{\prime\prime}(w)\big\}\big]  \mid F_n^{\prime\prime}(w)\Big]  \\ \nonumber &= \big(1+o(1)\big) \frac{2nK}{M_1 M_2} \frac{1}{\pi\sigma_1^2} \exp\Big(-|F_n^{\prime\prime}(w)|^2 \frac{\mu_1^2}{\sigma_1^2} \, \Big) \cdot \bE\Big[|F_n^\prime(w)|^2\mathbf{1}_{\{|F_n^\prime(w)| \in \frac{|F_n^{\prime\prime}(w)|}{2} \, n^{-1/4} U \}} \mid F_n^{\prime\prime}(w)\Big] \cdot \mathbf{1}_{\mathcal{E}_\ast} \\ \nonumber &= \big(1+o(1)\big) \frac{K}{M_1 M_2} \frac{|F_n^{\prime\prime}(w)|^4}{4\pi \sigma_1^2}\frac{a_4}{\big(a_2a_4 - a_3^2\big)} \exp\Big(-|F_n^{\prime\prime}(w)|^2 \Big\{ \frac{\mu_1^2}{\sigma_1^2} +\frac{a_3^2}{a_4(a_2a_4 - a_3^2)} \Big\} \Big) \cdot \Big(\int_{U} t^3 {\rm d}t\Big) \cdot \mathbf{1}_{\mathcal{E}_\ast} \, .
		\end{align}
		Finally, we apply~\eqref{eq:computation_of_forth_moment_with_tilte} with $$c = \frac{\mu_1^2}{\sigma_1^2} +\frac{a_3^2}{a_4(a_2a_4 - a_3^2)}$$ to get that
		\[
		\bE\bigg[|F_n^{\prime\prime}(w)|^4\exp\Big(-|F_n^{\prime\prime}(w)|^2 \Big\{ \frac{\mu_1^2}{\sigma_1^2} +\frac{a_3^2}{a_4(a_2a_4 - a_3^2)} \Big\} \Big) \bigg] = \frac{2a_4^4}{(1+c a_4^2)^3} \, .
		\]
		It remains to recall that $\mathfrak{F}$ is given by~\eqref{eq:def_of_limiting_intensity_function_F} and note the relations
		\[
		\mu_1^2 = \frac{\Delta_1^2}{\Delta_2^2} \, , \quad  \sigma_1^2 = \frac{\eta}{\Delta_2} \, , \quad c=  \frac{1}{\Delta_2} \Big(\frac{\Delta_1^2}{\eta} + \frac{a_3^2}{a_4}\Big) \, ,
		\]
		where $\Delta_1 = a_1a_3 - a_2^2$ , $\Delta_2 = a_2a_4-a_3^3$ and $\eta$ is given by~\eqref{eq:def_of_eta_function}. In fact, the formula for $\sigma_1^2$ can also be seen from~\eqref{eq:eta_as_a_schur_complement} above. Now, some simple algebra yields that
		\begin{align*}
			\bP\big[A_z(U) \cap \mathcal{E}_\ast\big] &= \bE\Big[\bE\big[\mathbf{1}_{A_z(U) \cap \mathcal{E}_\ast} \mid F_n^{\prime\prime}(w)\big]\Big] \\ &= \big(1+o(1)\big) \frac{K}{M_1 M_2} \frac{1}{2\pi \sigma_1^2}\frac{1}{\big(a_2a_4 - a_3^2\big)}  \frac{a_4^5}{(1+c a_4^2)^3} \, \Big(\int_{U} t^3 {\rm d}t\Big) \\ &= \big(1+o(1)\big)\frac{K}{M_1 M_2} \, \mathfrak{F}(x) \, \Big(\int_{U} t^3 {\rm d}t\Big) \, , 
		\end{align*}
		and we are done. 
	\end{proof}
	
	\pagebreak 
	
	\appendix

	\section{Decay of characteristic function for smooth tuples}
	\label{sec:decay_of_char_function_tuples}
	The goal of this appendix is to prove Theorem~\ref{thm:decay_of_char_function_tuples_near_circle}, which states that the characteristic function on $S_n(\mathbf{z})$ decays very fast for tuples $\mathbf{z} \in \Omega_K^m$ which are smooth and separated. The idea of the proof is similar to the one used to prove Lemma~\ref{lemma:small_ball_probability_smooth_points_bound_single_point}, only that here our situation is technically more complicated as we need to deal with $m$ points and another derivative. For $\mathbf{z} = (z_1,\ldots,z_m) \in \Omega_K^m$ and $\eta \in \bC^{3m}$ we denote by 
	\begin{equation}
		\label{eq:def_of_psi_k_tuples}
		\psi(k) = \psi(k;\mathbf{z},\eta) = \text{Re} \Big[\sum_{j=1}^{m} z_j^k \Big(\eta_j + \frac{k}{n} \eta_{j+m} + \frac{k(k-1)}{n^2} \eta_{j+2m} \Big)\Big] \, .
	\end{equation}
	The following proposition, which is analogous to Lemma~\ref{lemma:sum_of_norm_psi_k_intermidiate_range_is_large_single_point}, is the key result of this appendix.
	\begin{proposition}
		\label{prop:sum_of_distance_to_integers_psi_k_tuples_is_large}
		For an $n^\kappa$-smooth and $\gamma$-spread tuple $\mathbf{z} \in \Omega_K^m$, with $\gamma \ge n^{-{1/30m}}$ and for all $\eta\in \bC^{3m}$ with $n^{-1/6} \le |\eta| \le n^{K_\ast}$ we have
		\[
		\sum_{k=0}^n \| \psi(k) \|_{\bR/\bZ}^2 \ge \log^3 n 
		\]
		for all $n\ge n_0(K_\ast,K,m,\kappa)$ large enough.
	\end{proposition}
	We first show how Proposition~\ref{prop:sum_of_distance_to_integers_psi_k_tuples_is_large} implies Theorem~\ref{thm:decay_of_char_function_tuples_near_circle}.
	\begin{proof}[Proof of Theorem~\ref{thm:decay_of_char_function_tuples_near_circle}]
		We recall a basic inequality for the characteristic function of a sum of independent random vectors (see, for instance~\cite[Eq.~(9.2)]{Cook-Nguyen}):
		\begin{equation}
			\label{eq:proof_of_thm:decay_of_char_function_tuples_near_circle_bound_of_char}
			\Big|\bE\Big[e\big(\text{\normalfont Re}(S_n(\mathbf{z}) \, \overline{u})\big)\Big]\Big| \le  \exp\bigg(-c \inf_{a_1\le a \le a_2} \sum_{k=0}^{n} \| \psi(k;\mathbf{z}, a u/\sqrt{n}) \|_{\bR/\bZ}^2 \bigg)
		\end{equation}
		In fact, we already used this inequality in the course of the proof of Lemma~\ref{lemma:small_ball_probability_smooth_points_bound_single_point}. There, we were dealing with the symmetric random coefficients law $\eta_0$ instead of the present $\xi_0$, so to prove~\eqref{eq:proof_of_thm:decay_of_char_function_tuples_near_circle_bound_of_char} we need to first symmetrize the random coefficients and then use the non-degeneracy and apply~\eqref{eq:bound_of_char_of_symmetric_random_variable_distance_to_integers}. The extra symmetrization needed here only changes the value of the numerical constant $c>0$ in~\eqref{eq:proof_of_thm:decay_of_char_function_tuples_near_circle_bound_of_char}. As the constants $a_1,a_2 \asymp 1$ depend only on the distribution of $\xi_0$, we can apply Proposition~\ref{prop:sum_of_distance_to_integers_psi_k_tuples_is_large} with $\eta = a u/\sqrt{n}$, which is in the correct range. By~\eqref{eq:proof_of_thm:decay_of_char_function_tuples_near_circle_bound_of_char}, this implies the desired bound on the characteristic function of $S_n(\mathbf{z})$. 
	\end{proof}
	
	The remainder of this section is dedicated to proving Proposition \ref{prop:sum_of_distance_to_integers_psi_k_tuples_is_large}. 
	Throughout this section we will denote by $T = \log^3 n $ and suppose towards a contradiction that 
	\begin{equation}
		\label{eq:sum_of_distance_to_integers_psi_k_tuples_is_small_by_contradiction}
		\sum_{k=0}^n \| \psi(k) \|_{\bR/\bZ}^2 < T \, .
	\end{equation}
	By Markov's inequality,~\eqref{eq:sum_of_distance_to_integers_psi_k_tuples_is_small_by_contradiction} implies that
	\[
	\# \Big\{ k\in [0,n] \cap \bZ \, : \, \| \psi(k) \|_{\bR/\bZ}^2 \ge T^{-1} \Big\} \le T^3 \, .
	\]
	Hence, by the pigeonhole principle, there exists an interval $\mathcal{J}\subset [0,n]$ of length at least $n/T^4$ so that
	\begin{equation}
		\label{eq:bound_on_distance_from_integers_of_psi_k_on_interval_for_tuples}
		\sup_{k\in \mathcal{J}} \| \psi(k) \|_{\bR/\bZ} \le T^{-1} \, .
	\end{equation}
	The basic idea is that~\eqref{eq:bound_on_distance_from_integers_of_psi_k_on_interval_for_tuples} combined with~\eqref{eq:sum_of_distance_to_integers_psi_k_tuples_is_small_by_contradiction} implies that the sequence $\{\psi(k)\}$ and its finite differences are all close to integers, which is impossible for smooth $\mathbf{z}$. By the Dirichlet theorem on simultaneous Diophantine approximations (see for instance~\cite[Chapter~1, Theorem~VI]{Cassels}), there exists $q\in [1,n^\kappa]\cap \bZ$ and $s_1,\ldots,s_m\in \bR$ so that
	\begin{equation}
		\label{eq:simulatneous_dirichlet}
		q \, \text{arg}(z_r)/\pi - s_r \in \bZ \, ,
	\end{equation}
	and
	\[
	\max_{1\le r\le m} |s_r| \le   \sqrt{m} n^{-\kappa/m} \, .
	\]
	Thus, for each $1\le r\le m$ we have
	\[
	|1-z_r^q| \le |1-|z_r|^q| + |z_r|^q |1 - e(q \, \text{arg}(z_r))| \lesssim \frac{K}{n^{1-\kappa}} + 2\pi |s_r| \lesssim \sqrt{m} \, n^{-\kappa/m} \, .
	\]
	\subsection{Finite differences revisited}
	For a sequence $g:[n] \to \bC$ and non-negative integers $h,q$ we define the discrete differential of order $h$ with step size $q$ as before~\eqref{eq:def_on_finite_difference}, that is
	\begin{equation*}
		(\Delta_q^h g) ( k) = \sum_{t=0}^{h} \binom{h}{t} (-1)^{t} g(k+tq) \, .
	\end{equation*}
	Recall the definition of the sequences from Section~\ref{sec:gaussian_comparison_for_tuples} given by
	\[
	a_z(k) = z^k \, , \quad b_z(k) = \frac{k}{n} z^k \, ,\quad c_z(k) = \frac{k(k-1)}{n^2} z^k \, .
	\]
	A simple computation of some geometric sums and their derivatives shows that 
	\begin{align*}
		(\Delta_q^h a_z)(k) &= \sum_{t=0}^{h} \binom{h}{t} (-1)^{t}z^{k+tq} = (1-z^q)^h z^k \, , \\ (\Delta_q^h b_z)(k) &= \sum_{t=0}^{h} \binom{h}{t} (-1)^{t} \frac{k+tq}{n} z^{k+tq} = \frac{z}{n}\frac{\partial}{\partial z}\big[(1-z^q)^h z^k\big] \, , \\ (\Delta_q^h c_z)(k) &= \sum_{t=0}^{h} \binom{h}{t} (-1)^{t} \frac{(k+tq)(k+tq-1)}{n^2} z^{k+tq} = \frac{z^2}{n^2}\frac{\partial^2}{\partial z^2 }\big[(1-z^q)^h z^k\big] \, .  
	\end{align*}
	Denoting by
	\begin{equation*}
		f_{z,\ell}(k ) = (1-z^{q\ell})^h z^k
	\end{equation*}
	for $\ell \ge 1$ and $q\in[1,n^\kappa]$ given by~\eqref{eq:simulatneous_dirichlet}, we can combine to above identities with~\eqref{eq:def_of_psi_k_tuples} to get
	\begin{equation}
		\label{eq:discrete_h_difference_of_psi_tuples}
		(\Delta_{\ell q}^h \psi)(k) = \text{Re} \bigg[\sum_{j=1}^{m}  f_{z_j,\ell}(k) \, \eta_j + \frac{z_j}{n} \frac{\partial}{\partial z_j} \Big[f_{z_j,\ell}(k) \Big] \, \eta_{j+m} + \frac{z_j^2}{n^2} \frac{\partial^2}{\partial z_j^2} \Big[f_{z_j,\ell}(k)\Big]\,  \eta_{j+2m} \bigg]\, ,
	\end{equation}
	for all $\ell,h \ge 1$. 
	\begin{lemma}
		\label{lemma:bounding_discrete_differential_for_tuples_in_terms_of_distance_to_integers}
		There exists $h = O_{m,\kappa,K_\ast}(1)$ such that for all $\ell \ge 1$ and for any $k$ with $[k,k+h\ell q] \subset J$ we have
		\[
		\Big|(\Delta_{\ell q}^h \psi)(k)\Big| \lesssim_{m,\kappa} \sum_{s=0}^{h} \| \psi(k+s\ell q) \|_{\bR/\bZ} \, .
		\]
	\end{lemma}
	\begin{proof}
		We start with the case $\ell =1$ and then use that to conclude the lemma for general $\ell \ge 1$. Our choice of $q$ given by~\eqref{eq:simulatneous_dirichlet} implies that for all $j=1,\ldots,m$ we have
		\[
		|f_{z_j,1}(k)| = |z_j|^k |1-z_j^q|^h \le e^{K} m^{h/2} n^{-h\kappa/m} \lesssim n^{-h\kappa/m} \, .
		\]
		Furthermore, Cauchy's estimates~\eqref{eq:cauchy_estimates} gives that
		\[
		\max\Big\{ \frac{z_j}{n} \frac{\partial}{\partial z_j} f_{z_j,\ell}(k) \, , \, \frac{z_j^2}{n^2} \frac{\partial^2}{\partial z_j^2} f_{z_j,\ell}(k)  \Big\} \lesssim n^{-h\kappa/m} \, ,  
		\]
		and hence
		\[
		\Big|(\Delta_{q}^h \psi)(k)\Big| \lesssim n^{-h\kappa/m} \sum_{j=1}^{3m} |\eta_j| \lesssim n^{K_\ast - h\kappa/m} \, .
		\]
		Letting $m(k)$ be the closest integer to $\psi(k)$, the triangle inequality together with~\eqref{eq:bound_on_distance_from_integers_of_psi_k_on_interval_for_tuples} implies that
		\[
		\Big|(\Delta_{q}^h m)(k)\Big| \lesssim  n^{K_\ast - h\kappa/m} + \Big|(\Delta_{q}^h (m-\psi) )(k)\Big| \le n^{K_\ast - h\kappa/m} + 2^h \, T^{-1} \, ,
		\]
		as long as $[k,k+hq] \subset \mathcal{J}$. Taking $h = \lfloor K_\ast m/\kappa\rfloor + 2$, we see that for $n$ large enough the right-hand side is smaller than $1$, and it follows that
		\[
		(\Delta_{q}^h m)(k) = 0 \, .
		\]
		This proves the desired bound for $\ell =1$, as 
		\[
		\Big|(\Delta_{q}^h \psi)(k)\Big| =  \Big|(\Delta_{q}^h \psi)(k) - (\Delta_{q}^h m)(k) \Big| \lesssim 2^h \sum_{s=0}^{h} \| \psi(k+s q) \|_{\bR/\bZ} \, . 
		\]
		For $\ell \ge 2$, repeated application of the above for $k$ running through progressions starting from $k_0, k_0 + q, k_0 + 2q,\ldots,$ gives that for any $k\in \mathcal{J}$ such that $[k,k+hq] \subset \mathcal{J}$ there exists a polynomial $Q_k$ of degree at most $h-1$ so that
		\[
		m(k+tq) = Q_k(t) \, , 
		\] 
		for all $t\ge 0$ such that $[k,k+tq]\subset \mathcal{J}$. Thus, we also have $(\Delta_{\ell q}^{h} m) (k) = 0$ for all $\ell \ge 1$ and for all $k$ such that $[k,k+\ell h q] \subset \mathcal{J}$. The triangle inequality implies that for all such $k$'s we have
		\[
		\Big|(\Delta_{\ell q}^h \psi)(k)\Big| =  \Big|(\Delta_{\ell q}^h \psi)(k) - (\Delta_{\ell q}^h m)(k) \Big| \lesssim 2^h \sum_{s=0}^{h} \| \psi(k+s\ell q) \|_{\bR/\bZ} \, ,
		\]
		as desired.
	\end{proof}
	\subsection{Canceling out variables}
	Recall our assumption that $|\eta| \ge n^{-1/6}$, where $\eta = (\eta_1,\ldots,\eta_{3m}) \in \bC^{3m}$. The proof of Proposition~\ref{prop:sum_of_distance_to_integers_psi_k_tuples_is_large} splits into three cases:
	\begin{enumerate}
		\setlength\itemsep{0.5em}
		\item There exists $j\in \{2m+1,\ldots,3m\}$ such that $|\eta_j| \ge n^{-1/4}$; or 
		\item For all $j\in \{2m+1,\ldots,3m\}$ we have $|\eta_j| < n^{-1/4}$, and there exists $j^\prime \in \{m+1,\ldots,2m\}$ such that $|\eta_{j^\prime}| \ge n^{-1/5}$; or
		\item  For all $j\in \{m+1,\ldots,3m\}$ we have $|\eta_j| < n^{-1/5}$, which implies that for some $j^\prime \in \{1,\ldots,m\}$ we have $|\eta_{j^\prime}|\gtrsim_m n^{-1/6}$.
	\end{enumerate}
	In what follows, we shall deal mostly with the first case (which is the most technical), and comment  in the end of the section what are the necessary adaptations for the other two cases, see Remark~\ref{remark:how_to_deal_with_other_cases_prop:sum_of_distance_to_integers_psi_k_tuples_is_large}. From here on now, we shall assume without loss of generality that
	\begin{equation*}
		|\eta_{3m}| \ge n^{-1/4} \, .
	\end{equation*}
	\begin{lemma}
		\label{lemma:taking_suitable_differences_to_isolate_large_variable}
		For any positive integers $k,\ell, L, L^\prime$ such that $[k,k+h\ell q + 3(m-1)L + 2L^\prime] \subset \mathcal{J}$ with $h\ge 1$ as in Lemma~\ref{lemma:bounding_discrete_differential_for_tuples_in_terms_of_distance_to_integers}, we have
		\begin{multline}
			\label{eq:taking_suitable_differences_to_isolate_large_variable_inequality}
			\frac{(L^\prime)^2}{n^2} \bigg|\eta_{3m} \big(1-z_m^{\ell q}\big)^{h} \big(1-(z_m/|z_m|)^{2L^\prime}\big)^3 \prod_{j=1}^{m-1} \Big(1-\Big(\frac{z_m}{z_r}\Big)^L\Big)^3\Big(1-\Big(\frac{z_m}{\overline z_r}\Big)^L\Big)^3 \bigg| \\ \lesssim_{\kappa,m,K_\ast} \sum_{s=0}^{h} \sum_{a=0}^{6(m-1)} \sum_{b=0}^{4} \| \psi(k+s\ell q + aL + bL^\prime) \|_{\bR/\bZ} \, .
		\end{multline}
	\end{lemma}
	\begin{proof}
		The method of proof is similar to the one applied in the proof of Lemma~\ref{lemma:non-singularity_for_covariance_of_spread_tuple}, namely, we will use repeated applications of the twisted differences to cancel out terms from~\eqref{eq:discrete_h_difference_of_psi_tuples}. Recall that for $\zeta \in \bC$ the third-order discrete differential operator with step size $L\ge 1$ is given by~\eqref{eq:def_third_order_twisted_diff_operator}, i.e.
		\[
		(D_\zeta f)(k) = \sum_{t=0}^{3} \binom{3}{t} (-1)^t \zeta^{-tL} f(k+tL) \, .
		\]
		Recalling that $f_{z,\ell}(k ) = (1-z^{q\ell})^h z^k$, the computation~\eqref{eq:computaion_of_third_order_twissted_differential_for_sequences_a_b_c} shows that
		\begin{equation}
			\label{eq:computation_of_twisted_diff_for_f_z_l_k}
			(D_\zeta f_{z,\ell} )(k) = \big(1-(z/\zeta)^L\big)^3 f_{z,\ell}(k) \, .
		\end{equation}
		As in the proof of Lemma~\ref{lemma:non-singularity_for_covariance_of_spread_tuple}, denote by
		\[
		D = D_{z_1} \circ \ldots \circ D_{z_{m-1}} \circ D_{\overline z_1} \circ \ldots \circ D_{\overline z_{m-1}} \circ D_{\overline z_m}\, .
		\]
		As derivative and sum commute, we see from~\eqref{eq:discrete_h_difference_of_psi_tuples} that
		\begin{equation}
			\label{eq:first_differntial_to_finite_difference_of_order_h_psi_tuples}
			(D\circ \Delta_{\ell q}^h \psi) (k) = \text{Re} \bigg[(D f_{z_m,\ell})(k) \, \eta_m + \frac{z_m}{n} \frac{\partial}{\partial z_m} \Big[(D f_{z_j,\ell})(k) \Big] \, \eta_{2m} + \frac{z_m^2}{n^2} \frac{\partial^2}{\partial z_m^2} \Big[(D f_{z_m,\ell})(k)\Big]\,  \eta_{3m} \bigg] \, .
		\end{equation} 
		With the aid of the linear operator $D$, we have turned a sum of $3m$ terms for $\Delta_{\ell q}^h \psi$ into a sum of $3$ term. Now we want to get rid of the real part. Indeed, let $\delta_L(s) = (1-s^L)$ and note that~\eqref{eq:computation_of_twisted_diff_for_f_z_l_k} implies that
		\[
		(D f_{z_m,\ell} )(k) = \bigg(\prod_{j=1}^{m-1} \Big[\delta_L(z_m/z_r) \delta_L(z_m/\overline z_r)\Big]^3 \, 	\bigg) \cdot f_{z_m,\ell}(k) \,. 
		\]
		Let $L^\prime\ge 1$ be some integer and denote by $\widetilde D_{\overline z_m}$ the twisted differential of order $3$ with ``twist" $\overline z_m$ and step size $L^\prime$ instead of $L$. As before, we have 
		\[
		(\widetilde D_{\overline z_m} f_{\overline z_m , \ell} )(k) = 0 \, ,
		\] 
		and
		\[
		(\widetilde D_{\overline z_m} \circ D f_{\overline z_m , \ell} )(k) = \delta_{2L^\prime}(z_m/|z_m|)^3 \cdot  \bigg(\prod_{j=1}^{m-1} \Big[\delta_L(z_m/z_r) \delta_L(z_m/\overline z_r)\Big]^3 \, 	\bigg) \cdot f_{z_m,\ell}(k) \, .
		\]
		Hence, we obtain from~\eqref{eq:first_differntial_to_finite_difference_of_order_h_psi_tuples} that
		\begin{multline}
			\label{eq:second_differntial_to_finite_difference_of_order_h_psi_tuples}
			(\widetilde{D}_{\overline{z}_m} \circ D\circ \Delta_{\ell q}^h \psi) (k) \\  = \frac{1}{2} \bigg[H(z_m)  f_{z_m,\ell} (k) \, \eta_m + \frac{z_m}{n} \frac{\partial}{\partial z_m} \Big[H(z_m)  f_{z_m,\ell} (k)\Big] \, \eta_{2m} + \frac{z_m^2}{n^2} \frac{\partial^2}{\partial z_m^2} \Big[H(z_m)  f_{z_m,\ell} (k)\Big]\,  \eta_{3m} \bigg] \, ,
		\end{multline}
		where
		\[
		H(z) = \delta_{2L^\prime}(z/|z|)^3 \cdot \prod_{j=1}^{m-1} \Big[\delta_L(z/z_r) \delta_L(z/\overline z_r)\Big]^3 \, .
		\]
		Note that $H$ is a smooth function in a neighborhood of $z_m$. The goal now is to take suitable differences from~\eqref{eq:second_differntial_to_finite_difference_of_order_h_psi_tuples} to isolate the variable $\eta_{3m}$. Denote by $$\widetilde H (z) = (1-z^{\ell q})^h H(z) \, ,$$ which, as would be important in what follows, does not depend on $k$, and by
		\[
		F(k) = 2 \cdot (\widetilde{D}_{\overline{z}_m} \circ D\circ \Delta_{\ell q}^h \psi) (k) \, .
		\]
		Then~\eqref{eq:second_differntial_to_finite_difference_of_order_h_psi_tuples} reads
		\begin{align*}
			F(k) &= \widetilde H(z_m) z_m^k \eta_m + \frac{z_m}{n} \frac{\partial}{\partial z_m} \Big[\widetilde H(z_m)  z_m^k \Big] \, \eta_{2m} + \frac{z_m^2}{n^2} \frac{\partial^2}{\partial z_m^2} \Big[\widetilde H(z_m)  z_m^k \Big]\,  \eta_{3m} \\ &= \widetilde H(z_m) z_m^k \bigg(\frac{k(k-1)}{n^2} \eta_{3m} + \frac{k}{n} \mathsf{F}_1 + \mathsf{F}_2\bigg) \, .
		\end{align*}
		Here $\mathsf{F}_i$, $i=1,2$ are functions that may depend on the variables $\ell,q,n,\eta,\mathbf{z},h$ but not on $k$. This representation implies that
		\begin{multline}
			\label{eq:third_differntial_to_finite_difference_of_order_h_psi_tuples}
			F(k) - 2z_m^{-L^\prime} F(k+L^\prime) + z_m^{-2L^\prime} F(k+2L^\prime) = \frac{(L^\prime)^2}{n^2} \widetilde H(z_m) z_m^k \eta_{3m} \\ = \frac{(L^\prime)^2}{n^2} H(z_m)(1-z_m^{\ell q})^h  z_m^k \eta_{3m}\, .
		\end{multline}
		Taking the absolute value on both sides of~\eqref{eq:third_differntial_to_finite_difference_of_order_h_psi_tuples} and plugging in the formula for $H$, we identify the right-hand side exactly as the left-hand side of~\eqref{eq:taking_suitable_differences_to_isolate_large_variable_inequality}. The proof of the lemma now follows by applying the triangle inequality to the left-hand side of~\eqref{eq:third_differntial_to_finite_difference_of_order_h_psi_tuples} and applying Lemma~\ref{lemma:bounding_discrete_differential_for_tuples_in_terms_of_distance_to_integers} to each term in the sum generated by the difference operator $\widetilde D_{\overline{z}_m} \circ D$. 
	\end{proof}
	With Lemma~\ref{lemma:taking_suitable_differences_to_isolate_large_variable} at our disposal, we are ready to prove Proposition~\ref{prop:sum_of_distance_to_integers_psi_k_tuples_is_large} and conclude this section.
	
	\begin{proof}[Proof of Proposition~\ref{prop:sum_of_distance_to_integers_psi_k_tuples_is_large}]
		Recall our assumption that $|\eta_{3m}|\ge n^{-1/4}$ and that $\mathcal{J}\subset[0,n]$ is an interval of length at least $n/T^4$. Suppose that $\ell \ge 1$ is an integer such that $\ell h q \le |\mathcal{J}|/2$, where $h$ is as in Lemma~\ref{lemma:bounding_discrete_differential_for_tuples_in_terms_of_distance_to_integers} and $q$ is given by~\eqref{eq:simulatneous_dirichlet}. By Claim~\ref{claim:the_L_rescaling_for_tuples}, we can choose $L\asymp n/T^6 = o(|\mathcal{J}|)$ such that
		\[
		\qquad \qquad |z_j^L - z_{j^\prime}^L| \ge \gamma/T^6 \, , \qquad \forall 1\le j < j^\prime \le m \, .
		\] 
		Furthermore, since $z_m$ is $n^\kappa$-smooth, we can choose $L^\prime \asymp n/T^7$ such that
		\[
		\big|1-(z_m/|z_m|)^{2L^\prime}\big| > \gamma^2 \ge n^{-1/15m} \, .
		\]
		With these choices of $\ell , L , L^\prime$, we have that the left-hand side of~\eqref{eq:taking_suitable_differences_to_isolate_large_variable_inequality} is at least
		\[
		\gtrsim_{m} n^{-1/4} T^{-O_m(1)} |1-z_m^{\ell q}|^m \gamma^{6m} \gtrsim n^{-9/20} T^{-O_m(1)} |1-z_m^{\ell q}|^m \, .
		\]
		On the other hand, since we assume by contradiction that~\eqref{eq:sum_of_distance_to_integers_psi_k_tuples_is_small_by_contradiction} holds, the Cauchy-Schwarz inequality implies that
		\[
		\sum_{k=0}^{n} \|\psi(k) \|_{\bR/\bZ} \le \sqrt{n T} \, ,
		\] 
		and by pigeonholing we can find $k\in \mathcal{J}$ so that $[k,k+\ell h q + 3(m-1) L + 2L^\prime] \subset \mathcal{J}$ so that the right-hand side of~\eqref{eq:taking_suitable_differences_to_isolate_large_variable_inequality} is $\lesssim_{m,\kappa} n^{-1/2} T^{O(1)}$. Combining, both observations, Lemma~\ref{lemma:taking_suitable_differences_to_isolate_large_variable} implies that
		\[
		\big|1-z_m^{\ell q}\big| \le T^{O_m(1)} n^{-1/10m} \, .
		\]
		Since $q\in [1,n^\kappa]$ and since $z_m\in \Omega_K$, we get that for all $n$ large enough
		\begin{equation}
			\label{eq:proof_of_prop:sum_of_distance_to_integers_psi_k_tuples_is_large_conclusion_on_angle}
			\|\ell q \text{arg}(z_m) /\pi \|_{\bR/\bZ} \le \big|1-e\big(\ell q \text{arg}(z_m)\big) \big| \le n^{-1/20m} \, ,
		\end{equation}
		for all integers $\ell = 1,\ldots,\lfloor |\mathcal{J}|/qh \rfloor$. It remains to note that for all $\theta\in \bR$ with $\|\theta\|_{\bR/\bZ} \le \tfrac{1}{100}$, we have
		\[
		\|2 \theta \|_{\bR/\bZ} \le 2\| \theta \|_{\bR/\bZ} \, .
		\]
		Iterating this inequality with the observation~\eqref{eq:proof_of_prop:sum_of_distance_to_integers_psi_k_tuples_is_large_conclusion_on_angle} implies that
		\begin{equation*}
			\| q \text{arg}(z_m) /\pi \|_{\bR/\bZ} \lesssim q h |\mathcal{J}|^{-1} n^{-1/20m} = o(n^{\kappa-1})
		\end{equation*}
		and we get a contradiction to the fact that $z_m$ is $n^\kappa$-smooth, which completes the proof.
	\end{proof}
	\begin{remark}
		\label{remark:how_to_deal_with_other_cases_prop:sum_of_distance_to_integers_psi_k_tuples_is_large}
		We now explain how to handle the case where $|\eta_j| < n^{-1/4}$ for all $j\in \{2m+1,\ldots,3m\}$ but without loss of generality we have $|\eta_{2m}|\ge n^{-1/5}$. Then, we can show the following simpler analogue of Lemma~\ref{lemma:taking_suitable_differences_to_isolate_large_variable}: 
	\end{remark}
	\begin{lemma}
		\label{lemma:taking_suitable_differences_to_isolate_large_variable_case_2}
		For any positive integers $k,\ell, L, L^\prime$ such that $[k,k+h\ell q + 3(m-1)L + 2L^\prime] \subset \mathcal{J}$ with $h\ge 1$ as in Lemma~\ref{lemma:bounding_discrete_differential_for_tuples_in_terms_of_distance_to_integers}, we have
		\begin{multline*}
			\label{eq:taking_suitable_differences_to_isolate_large_variable_inequality_case_2}
			\frac{L^\prime}{n} \, \bigg|\eta_{2m} \big(1-z_m^{\ell q}\big)^{h} \big(1-(z_m/|z_m|)^{2L^\prime}\big)^3 \prod_{j=1}^{m-1} \Big(1-\Big(\frac{z_m}{z_r}\Big)^L\Big)^3\Big(1-\Big(\frac{z_m}{\overline z_r}\Big)^L\Big)^3 \bigg| \\ \lesssim_{\kappa,m,K_\ast} \big(1+o(n^{-1/20})\big)\sum_{s=0}^{h} \sum_{a=0}^{6(m-1)} \sum_{b=0}^{4} \| \psi(k+s\ell q + aL + bL^\prime) \|_{\bR/\bZ} \, .
		\end{multline*}
	\end{lemma}
	
	Here the additional $\big(1+o(n^{-1/20})\big)$ term on the right-hand side is caused by estimating from above the contribution of the second-order differential in~\eqref{eq:second_differntial_to_finite_difference_of_order_h_psi_tuples} and showing it is negligible compared to the contribution from the first differential. The proof then proceeds the same, just that in~\eqref{eq:third_differntial_to_finite_difference_of_order_h_psi_tuples} we only need to take the different of $F(k)$ and $z_m^{-L^\prime} F(k+L^\prime)$ as we want to isolate the term multiplying $k/n$ (in particular, on the left-hand side we get the term $L^\prime/n$ instead of $(L^\prime)^2/n^2$, as we did in Lemma~\ref{lemma:taking_suitable_differences_to_isolate_large_variable}). Once Lemma~\ref{lemma:taking_suitable_differences_to_isolate_large_variable_case_2} is established, the proof of Proposition~\ref{prop:sum_of_distance_to_integers_psi_k_tuples_is_large} proceeds in the exact same way. 
	
	Finally, we need to deal with the case where $|\eta_m|\gtrsim_{m} n^{-1/6}$, while for all $j\in\{m+1,\ldots,3m\}$ we have $|\eta_j| \le n^{-1/5}$. This is the simplest case of the three, as we don't need to take the extra differences to cancel out unnecessary variables from~\eqref{eq:second_differntial_to_finite_difference_of_order_h_psi_tuples}. We state the lemma in this case without further comment. 
	\begin{lemma}
		\label{lemma:taking_suitable_differences_to_isolate_large_variable_case_3}
		For any positive integers $k,\ell, L, L^\prime $ such that $[k,k+h\ell q + 3(m-1)L + 2L^\prime] \subset \mathcal{J}$ with $h\ge 1$ as in Lemma~\ref{lemma:bounding_discrete_differential_for_tuples_in_terms_of_distance_to_integers}, we have
		\begin{multline*}
			\bigg|\eta_{m} \big(1-z_m^{\ell q}\big)^{h} \big(1-(z_m/|z_m|)^{2L^\prime}\big)^3 \prod_{j=1}^{m-1} \Big(1-\Big(\frac{z_m}{z_r}\Big)^L\Big)^3\Big(1-\Big(\frac{z_m}{\overline z_r}\Big)^L\Big)^3 \bigg| \\ \lesssim_{\kappa,m,K_\ast} \big(1+o(n^{-1/30})\big)\sum_{s=0}^{h} \sum_{a=0}^{6(m-1)} \sum_{b=0}^{4} \| \psi(k+s\ell q + aL + bL^\prime) \|_{\bR/\bZ} \, .
		\end{multline*}
	\end{lemma}

	\printbibliography[heading=bibliography]

\end{document}